\documentclass[11pt]{amsart}
\usepackage{amsmath,amssymb,enumerate}
\newtheorem{prop}{Proposition}[section]
\newtheorem{coro}[prop]{Corollary}

\newtheorem{lem}[prop]{Lemma}
\newtheorem{rem}[prop]{Remark}
\newtheorem{exa}[prop]{Example}
\newtheorem{exe}[prop]{Exercise}
\newtheorem{defi}[prop]{Definition}
\newtheorem{theo}[prop]{Theorem}

\newcommand{\B}{\mathbb B}

\newcommand{\R}{\mathbb R}
\newcommand{\Q}{\mathbb Q}
\newcommand{\C}{\mathbb C}
\newcommand{\N}{\mathbb N}

\newcommand{\de}{\delta}
\newcommand{\e}{\varepsilon}
\newcommand{\f}{\varphi}

\newcommand{\p}{\psi}

 \newcommand \al {\alpha}
 \newcommand \la {\lambda}

 \newcommand \Om {\Omega}

 \newcommand \Sub {\Subset}
 \newcommand \sub{\subset}

 \newcommand \bd {\partial}

 \numberwithin{equation}{section}

  \title[A viscosity approach to complex Monge-Amp\`ere equations]{A viscosity approach to degenerate \\ complex Monge-Amp\`ere equations }

  \author{ Ahmed Zeriahi} 
 \date{\today}
 
\begin{document}
  
 \maketitle

\section*{Content}

\noindent {\bf 1.} {\it Pluripotential solutions to degenerate complex Monge-Amp\`ere equations}

 1.1. {Basic facts from Pluripotential Theory}  
 
 1.2. Statement of the main results
 
 1.3. Proofs
 
 \smallskip
  \noindent {\bf 2.} {\it The viscosity approach to degenerate elliptic non linear $2^{\text{nd}}$ order PDE's} 
  
 2.1. Definitions and examples
 
 2.2.  The comparison principle : Uniqueness of viscosity solutions 
 
 2.3.  The Perron method : Existence of viscosity solutions 
 
 \smallskip
 \noindent{\bf 3.} {\it  Viscosity solutions to degenerate complex Monge-Amp\`ere equations }
 
 3.1. The viscosity approach to complex Monge-Amp\`ere equations 
 
 3.2. The local viscosity  comparison principle 
 
 3.3. The global viscosity comparison principle
 
 3.4 Viscosity solutions  to complex Monge-Amp\`ere equations
 
 \smallskip
\noindent{\bf 4.} {\it Continuous versions of Calabi-Yau and Aubin-Yau theorems} 

 4.1. A weak version of Aubin-Yau Theorem 
   
 4.2. A weak version of Calabi-Yau theorem 
 
 4.3.Singular K\"ahler-Einstein metrics with continuous potentials 
 
 4.4. Concluding remarks

\section*{Abstract} 

This is the content of the lectures given by the author at the winter school KAWA3 held at the University of Barcelona in 2012 from January 30 to  
 February 3. The main goal was to give an account of viscosity techniques and to apply them to degenerate Complex Monge-Amp\`ere equations following recent works of P. Eyssidieux, V. Guedj and the author.

 We will survey the main techniques used in the viscosity approach
and show how to adapt them to degenerate complex Monge-Amp\`ere equations. 
 The heart of the matter in this approach is the "Comparison Principle" which allows 
 us to prove uniqueness of solutions. 
 
 We will prove a global viscosity comparison principle for degenerate complex Monge-Amp\`ere equations on compact K\"ahler manifolds and show how to combine Viscosity methods and Pluripotential methods to get "continuous versions" of  the
 Calabi-Yau and Aubin-Yau Theorems in some degenerate situations. In particular we prove the existence of singular K\"ahler-Einstein metrics with continuous potentials on compact normal K\"ahler varieties with mild singularities and ample or trivial canonical divisor.

\section*{Introduction}

 In the late seventies, E.~Bedford and B.~A.~Taylor (\cite{BT76}) started developing a new method of potential-theoretic nature adapted to the complex structure for solving degenerate complex Monge-Amp\`ere equations in strictly pseudoconvex domains in $\C^n$. They proved a Comparison Principle and, using the Perron method, they were able to solve the Dirichlet problem for degenerate complex Monge-Amp\`ere equation for continuous data. Then, 
 elaborating on this fundamental work, they succeeded in building a complex potential theory, called nowadays "Pluripotential Theory", to study fine properties of plurisubharmonic functions (see \cite{BT82}, \cite{BT87},  \cite{Dem89}, \cite{Kli91}).
 The Dirichlet problem for non degenerate complex Monge-Amp\`ere equations  with smooth bounday data and a smooth positive volume form on a bounded strongly pseudoconvex domains with smooth boundary was solved independently by L. Caffarelli, J.-J. Kohn, L. Nirenberg and J. Spruck in their fundamental work using methods from elliptic non linear PDE's (\cite{CKNS85}).
 
A quite elaborate theory was developed in the local case, thanks to the contributions
of several authors (see among others  \cite{BT88}, \cite{BT89}, \cite{Ceg84}, \cite{Kol95}, \cite{Ceg98}, \cite{Ceg04}, \cite{Bl06}, \cite{BGZ09}). There are good surveys on these last developments (see \cite{Bed93, Kis00, Kol05}).
 
 Pluripotential theory lies at the foundation of the recent approach to degenerate complex Monge-Amp\`ere equations on compact K\"ahler manifolds, as developed by many authors with applications to K\"ahler Geometry (see \cite{Kol98}, \cite{EGZ08}, \cite{KT08}, \cite{EGZ09}, \cite{BBGZ09},  \cite{DP10}, \cite{Zh06}, \cite{BBEGZ12}, \cite{PS10}, \cite{PSS12}). There is a nice and complete survey on the  recent developements in this area (see \cite{PSS12}).

 On the other hand, a standard approach to non linear second order degenerate elliptic equations is the method of viscosity solutions introduced first  by M.G. Crandall and P.-L. Lions (\cite{CL83}) at the begining of the eighties in order to prove existence and uniqueness of "solutions" in a generalized sense for first order non linear equations  of Hamilton-Jacobi type. But it appeared quickly that this method can be used to prove existence and uniqueness of "generalized weak solutions" to certain fully non linear second order degenerate elliptic  PDE's (\cite{CIL92}), especially for those equations for which the notions of "classical" solution (i.e. smooth solution), "generalized" solution (i.e. a solution in the Sobolev space $W^{2,\infty}$) or "weak" solution (i.e. a solution in the sense of distributions) do not make sense.
 The remarkable fact in this approach is that we can define, as in classical (linear) potential theory for the Laplace operator for example, the notions of subsolution and supersolution in a generalized sense (viscosity sense) for these equations. 
 
The main tool in the viscosity approach is the {\it Comparison Principle},  which allows  comparison of subsolutions and supersolutions with given boundary conditions. This implies uniqueness of viscosity solutions for the associated Dirichlet problem. Then Perron's method can be applied, as in the classical case, to construct the unique solution as the upper envelope of all subsolutions, once we know the existence of a subsolution and a supersolution with the given boundary conditions.

Whereas the viscosity approach has been developed for real Monge-Amp\`ere equations (see \cite{IL90}), the complex case has not been studied until recently.
There have been some recent interest in adapting viscosity methods to solve degenerate elliptic equations on compact or complete Riemannian manifolds (see \cite{AFS08}). This theory can be applied to complex Monge-Amp\`ere equations only in very restrictive cases since it requires the Riemannian curvature tensor to be nonnegative.
 There is a viscosity approach to the Dirichlet problem for the complex Monge-Amp\`ere equations on smooth domains in Stein manifolds in \cite{HL09} and \cite{HL11}. These articles however do not contain any new result for degenerate complex Monge-Amp\`ere equations, since that case is used there as motivation to develop a deep generalization of plurisubharmonic functions to Riemannian manifolds with some special geometric structure. 
 In a recent paper \cite{HL13}, the same authors also develop an interesting application to potential theory in almost complex manifolds and solve the Dirichlet problem in this general context.
 
 The most advanced results about the complex Monge-Amp\`ere equations
 were obtained quite recently in \cite{EGZ11}, and we will mostly follow the presentation given there. The main motivation was the problem of continuity of the potentials of the singular K\"ahler-Einstein metric in a compact K\"ahler manifold of general type constructed in \cite{EGZ09}.
 Since this paper appeared, there have been recent applications of viscosity methods to the Dirichlet problem for the complex Monge-Amp\`ere equation
 (see \cite{Wang10}) and more generally for the complex Hessian equation (see \cite{Ch12}). 
 
 It is worth observing  that there is no general comparison principle 
 which can be applied to a large class of degenerate elliptic fully non linear second order PDE's, including the degenerate complex Monge-Amp\`ere equations we are considering here. 

Nevertheless, viscosity methods can be adapted to the complex case and allow us to prove an appropriate Comparison Principle which leads along the same scheme to uniqueness and existence of viscosity solutions. 

The first aim of these notes is to present the fundamental ideas behind the viscosity approach. All the material we need can be found in the well known survey \cite{CIL92} (see also \cite{CC95}). There are also well written papers available in the literature (see \cite{Bar97}, \cite{DI04}, \cite{Car04})  but we will collect here the main ingredients we will need to adapt the viscosity methods to the complex case. 
The main result which we will use from the viscosity approach is what we call the Jensen-Ishii maximum principle which will be stated whithout proof here, referring to \cite{CIL92}. 
 
In order to compare the two approaches, we will start by reviewing the basic tools from Pluripotential theory we will need, namely the comparison principle for the 
complex Monge-Amp\` ere operator.
 We will use the pluripotential comparison principle and the Perron method to show how pluripotential theory provides bounded weak solutions to the degenerate complex Monge-Amp\`ere equations we are considering.
 
The second aim is to show how to adapt the viscosity methods in the context of complex Monge-Amp\`ere equations on domains as well as on compact K\"ahler manifolds following \cite{EGZ11}.
 We will compare viscosity solutions to pluripotential solutions.
The main advantage of the viscosity approach which we will exploit here is, not only that the notion of subsolution makes sense, but that we can also define the notion of supersolution; then a viscosity solution, if it exists, is necessarily continuous. Observe that in the pluripotential theory framework, we can also define the notion of subsolution, but it is not always clear whether a notion of supersolution makes sense and then the continuity of the pluripotential solution, if it exists, is not obtained for free.

Finally the third aim is to show how to combine pluripotential methods and viscosity methods to prove existence and uniqueness of continuous solutions to some degenerate complex Monge-Amp\`ere equations. 
 Moreover using Kolodziej's a priori $C^0-$estimates as extended in \cite{EGZ09}, we can give a soft proof of the continuous version of Yau's theorem solving the Calabi conjecture which applies for singular compact K\"ahler varieties with mild singularities (in the sense of the MPP programme \cite{BCHM10}) and with ample canonical divisor. In particular we prove that potentials of singular K\"ahler-Einstein metrics obtained previously in \cite{EGZ09} are continuous, whereas they only were known to be bounded.
 Surprisingly, this allows us to prove continuity of solutions to complex Monge-Amp\`ere equations in a degenerate situation where Pluripotential theory yields only boundedness. 
 
A previous version of these notes has been published in a special volume of Annales de la Facult\'e des Sciences de Toulouse (see \cite{Ze13}).

 \vskip 0.3 cm
 
 \noindent{\it Acknowledgements:} The author would like to thank the organizers Vincent Guedj, Joaquim Ortega-Cerd\`a and Pascal Thomas for inviting him to give this course at the third edition of the winter school KAWA in Barcelona in January-February 2012. 
The author would like to also thank  Vincent Guedj for useful discussions on the matter of these notes, Chinh H. Lu for careful reading of the previous version and the anonymous referee for his encouranging report.

\section{Pluripotential solutions to degenerate complex Monge-Amp\`ere equations}
\vskip 0.3 cm
\subsection{Basic facts from Pluripotential Theory}

 Pluripotential theory deals with plurisubharmonic (psh) functions. These functions appear naturally in many problems of complex analysis where  
 they play the role of soft objects compared to holomorphic functions which are more rigid. This philosophy led P. Lelong to the fundamental notion 
 of positive current (\cite{Lel68}) which play an important role not only in Complex Analysis but also in K\"ahler Geometry (see \cite{Dem92}, \cite{Dem}). 
 
 \subsubsection{The complex Monge-Amp\`ere operator}
 Let us recall the construction of Bedford and Taylor and state the main results which will be needed later on.
 Let $\Omega \subset \C^n$ be a domain and $PSH (\Omega) \subset L^1_{loc} (\Omega)$ be the set of plurisubharmonic functions in $\Omega$.
 
 Here we denote by $d = \bd + \overline{\bd} $ and $d^c := \frac{i}{2\pi} (\overline{\bd} - \bd)$ so that $dd^c =  \frac{i}{\pi} \bd  
 \overline{\bd}$. 
 
 Following P. Lelong, a current $T$ of bidmension $(p,p)$ on $\Omega$ is by definition a continuous linear form  acting on $C^{\infty}$-smooth differential forms of bidegree $(p,p)$ with compact support in $\Omega$. It is convenient to view a current of bidmension $(p,p)$ on $\Omega$ as a differential form of bidegree $(n-p,n-p)$ with  coefficients given by distributions in $\Omega$ (see \cite{Lel68}). 
 
 By P.Lelong, if $u \in PSH (\Omega)$ then $dd^c u$ is a closed positive current on $\Omega$, hence a differential form of bidegree $(1,1)$ whose coefficients are complex Borel measures in $\Omega$ (see \cite{Lel68}, \cite{Dem}).
   
 Since plurisubharmonic functions are invariant under holomorphic transformations, the notion of plurisubharmonicity makes sense on complex manifolds. Moreover plurisubharmonic functions appear naturally in complex geometry as local weights for singular metrics on holomorphic line bundles with positive curvature (see \cite{Dem}).
 
 We will review some basic facts on Pluripotential theory and refer to the original papers of Bedford and Taylor \cite{BT76, BT82, BT87} (see also  \cite{Dem89}, \cite{Kli91}).

 Let $u_1, \cdots, u_k$ be $C^2$-smooth psh functions in $\Omega$. Then  the following differential $(k,k)$-form
 $dd^c u_1 \wedge \cdots \wedge dd^c u_k$ has continuous coefficients, hence it can be seen as a closed positive current of bidegree $(k,k)$ in $\Omega$ acting by duality on  (smooth) test $(n-k,n-k)$-forms. 
Moreover we have
 $$
  dd^c (u_1 dd^c u_2 \wedge \cdots \wedge  dd^c u_k) = dd^c u_1\wedge \cdots \wedge  dd^c u_k 
  $$ 
 pointwise and weakly in the sense of currents in $\Omega$.
 
 In particular, for any smooth psh function in $\Omega$, we have
 $$
 (dd^c u)^n = c_n \text{det}\left(\frac{\partial^2 u}{\partial z_j \partial \bar z_k}\right) \beta_n,
 $$
 pointwise as a smooth form of top degree, where $\beta_n$ is the euclidean volume form on $\C^n$ and $c_n >0$ is a numerical constant.
 This formula will be used here  to identify $(dd^c u)^n $ to a positive Borel measure on $\Omega$, called the Monge-Amp\`ere measure of $u$ in $\Omega$.
 
 We want to extend this definition to non smooth psh functions. It is natural to use local approximation.
 Observe first that  by localisation and integration by parts, one can easily prove that for any compact sets $K, L$ such that $K \subset L^{\circ} \Subset \Omega$, there exists a positive constant $C > 0$, depending on $(K,L)$, such that for any smooth psh function  $u_1, \cdots, u_n$ in $\Omega$, we have
 \begin{equation} \label{eq:CLN}
 \int_{K}  dd^c u_1\wedge \cdots \wedge  dd^c u_n \leq C \Pi_{1 \leq k \leq n} \Vert u_k\Vert_{L^{\infty} (L)}.
 \end{equation}
 This inequality, called the {\it Chern-Levine-Nirenberg inequality}, allows to extend easily the definition of the Monge-Amp\`ere operator to continuous non smooth psh functions by regularisation.
 Indeed let $u$ be a continuous psh function in $\Omega$ and $u_j := u \star \chi_j$ its regularisation by convolution against a radial approximation of the Dirac unit mass at the origin. Let us prove that the sequence of (smooth) measures $(dd^c u_j)^n$ converges weakly in the sense of Radon measures in $\Omega$. Since $u$ is continuous, by Dini's lemma, the sequence $(u_j)$ decreases  to $u$, locally uniformly in $\Omega$. 
 Since the sequence $(u_j) $ is locally uniformly bounded in $\Omega$, it follows from Chern-Levine-Nirenberg inequality (\ref{eq:UCLN}), that the sequence of measures $(dd^c u_j)^n$ has locally uniformly bounded mass. Therefore it is enough to prove the convergence of the sequence of measures $(dd^c u_j)^n$ against any smooth test function. 
 This will be a consequence of the following observation. The problem of convergence being local,  it is enough to consider a test function with compact support in a small ball $B \Subset \Omega$. Fix such a smooth test function  $h$ with compact support in $B$. Then for any $C^2-$smooth psh functions $\f$ and $\p$ in  $\Omega$, we have by Stokes formula, 
 
 $$
 \int_{B} h ((dd^c \f)^n - (dd^c \p)^n) = \int_{B} (\f - \p) dd^c h \wedge  T,
 $$
 where $T := \sum_{i = 0}^{n - 1}(dd^c \f)^{i} \wedge (dd^c \p)^{n - 1 - i}).$
 Since $h$ is smooth of compact support, it is possible to write it as $h = w_1 - w_2,$ where $w_1, w_2$ are smooth psh functions in $\Omega$.
 Therefore if  $D $ is a neighbourhood of $\overline B$ such that $ B \Subset D \Subset \Omega$, then by Chern-Levine-Nirenberg inequality, there exists a uniform constant $C > 0$, depending only on a bound of the second derivatives of $h$ and on a uniform bound of $\f$ and $\psi$, such that
 \begin{equation} \label{eq:UCLN}
 \left\vert \int_{B} h ((dd^c \f)^n - (dd^c \p)^n)\right\vert \leq C \Vert \f - \p \Vert_{L^{\infty}(D)}. 
 \end{equation}
 Now let $u$ be a continuous psh function on $\Omega$ and $(u_j)$ its regularizing sequence by convolution. Then by Dini's lemma, the convergence is uniform in each compact set. It follows from Chern-Levine-Nirenberg and (\ref{eq:UCLN}) that the sequence of measures $(dd^c u_j)^n $ is a Cauchy sequence of Radon measures. Then it converges to a positive Radon measure on $\Omega$. Moreover again by (\ref{eq:UCLN}), the limit does not depend on the approximating sequence $(u_j)$ which converges to $u$ locally uniformly in $\Omega$. This limit is defined to be the Monge-Amp\`ere measure of $u$ and denoted by $(dd^c u)^n$.
 
 It turns out that the hypothesis of continuity on the  psh function $u$ is a strong condition. Indeed it is not preserved by standard constructions as upper envelopes, regularized limsup of psh functions which arise naturally when dealing with the Dirichlet problem for the complex Monge-Amp\`ere operator. 
 Therefore it is desirable to define the complex Monge-Amp\`ere operator for non continuous psh functions, say e.g. for bounded psh functions.
 As one may see from the previous reasoning, it is not clear how to define the complex Monge-Amp\`ere measure of $u$ by approximating $u$ by a decreasing sequence of smooth psh functions, since the convergence in not locally uniform anymore. However one of the main results in pluripotential theory says that plurisubharmonic functions are actually quasi continuous (\cite{BT82}) and then the convergence is quasi-uniform and the proof above can be extended to the bounded case.
 
 Actually to pass from continuous to bounded psh functions is one of the main problems when dealing with the complex Monge-Amp\`ere operator in contrast to the real Monge-Amp\`ere operator which deals with convex functions which are continuous.
 
  In their first seminal work \cite{BT76}, E.Bedford and B.A.Taylor were able to extend the definition of the complex Monge-Amp\`ere operator to the class of locally bounded psh functions using the notion of closed positive current. Their main observation is the following. Let $T$ be a closed positive current  
 of bidegree $(k,k)$ ($1 \leq k \leq n - 1$) and $u$ a locally bounded psh function in $\Omega$. It is well known that $T$ can be extended as a differential form with complex  
 Borel measure coefficients on $\Omega$. Then the current $u T$ is well defined by duality, since $u$ is a locally bounded Borel function and hence locally integrable with respect to  all the coefficients of $T$. Therefore we can define the current $dd^c (u T)$ in the weak sense. Now the following simple observation is crucial: the current $dd^c (u T)$  is again a closed positive  current on $\Omega$. Indeed, since the problem is local we can assume that the regularizing sequence $u_j \searrow u$ in $\Omega$. Then   $u_j T \rightharpoonup u T$ in the weak sense of measures in $\Omega$ and by continuity of the operator $dd^c$ for the weak topology, we conclude that 
 $ dd^c (u_j T) \rightharpoonup dd^c (u T)$ weakly in the sense of currents in $\Omega$.
 Now since $u_j$ is smooth, we have by Stokes formula for currents that $dd^c (u_j T) = dd^c u_j \wedge T$ is a positive closed currents.
 Therefore $dd^c (u T)$ is also a closed positive current in $\Omega$, which will be denoted by $dd^c u \wedge T$ (see \cite{Dem}).
 
 It is now clear that we can  repeat this construction: if $u_1, \cdots ,u_k$ are locally bounded psh functions, it is possible to define by induction  the current $dd^c u_1\wedge \cdots \wedge  dd^c u_k $ by the formula 
  $$
  dd^c u_1\wedge \cdots \wedge  dd^c u_k := dd^c (u_1 dd^c u_2 \wedge \cdots \wedge  dd^c u_k), 
  $$ 
 weakly in the sense of currents in $\Omega$, the resulting current being a closed positive current in $\Omega$.
 
 In particular if $u$ is a locally bounded psh function in $\Omega$, then the current of bidegree $(n,n)$ given by 
 $(dd^c u)^n =  dd^c u_1\wedge \cdots \wedge  dd^c u_n $, where $u_1 = \cdots = u_n = u$  can be identified to a positive Borel measure 
 denoted by $(dd^c u)^n$ called the Monge-Amp\`ere measure of $u$.
 
 Likely this definition coincides with the previous one when $u$ is a continuous psh function. More generally, using ingenious integration by parts and local approximations, Bedford and Taylor proved the following important convergence theorem (\cite{BT76}).
 \begin{theo} \label{thm:CV} Let $(u_j)$ and $(v_j)$ be decreasing  sequences of locally bounded psh functions in $\Omega$ converging to  locally bounded psh functions $u$ and $v$ respectively in $\Omega$. Then the sequence of measures $u_j (dd^c v_j)^n$ converges to the measure $ u (dd^c v)^n$ weakly in the sense of measures in $\Omega$.
 The same weak convergence still holds if $(u_j)$ or $(v_j)$  increases almost everywhere in $\Omega$ to $u$ or $v$ respectively.
 \end{theo}
 
 \subsubsection{The Pluripotential Comparison Principle}
 
 From Theorem \ref{thm:CV}, it is possible to derive the following fundamental result, which we will call the (local) maximum principle (\cite{BT87}).
 \begin{theo} (Maximum Principle).
 Let $u,v$ be locally bounded psh functions in $\Omega$. Then we have 
 \begin{equation} \label{eq:MP}
 {\bf 1}_{\{u < v \}} (dd^c \max \{u,v\})^n = {\bf 1}_{\{u < v \}} (dd^c v)^n,
 \end{equation}
 weakly in the sense of Borel measures in $\Omega$.
 \end{theo}
 Observe that the identity (\ref{eq:MP}) is trivial when $v$ is continuous, since the two psh functions $\max \{u,v\}$ and $v$ coincide on the (euclidean) open set $\{u < v \}$. The main difficulty in the proof of (\ref{eq:MP}) is to pass from continuous to bounded psh functions and this is the main feature in Bedford and Taylor work building up a potential theory for plurisubharmonic functions, called Pluripotential Theory (see \cite{BT76}, \cite{BT82}, \cite{BT87}, \cite{Dem89}). It turns out that the set $\{u < v \}$ is actually open for the plurifine topology and it was proved by Bedford and Taylor that the complex Monge-Amp\` ere operator is local in the plurifine topology (see \cite{BT87}).
 
 From the maximum principle, it is easy to deduce its companion, which will be called the Pluripotential Comparison Principle.
 
 \begin{coro} (Comparison Principle). \label{coro:LCP} Let $u,v$ be locally bounded psh functions in $\Omega \Subset \C^n$ such that $u \geq v$ on $\partial \Omega$ i.e. $\{u < v \} = \{z \in \Omega ; u (z) < v (z)\} \Subset \Omega$. Then 
 $$
 \int_{\{u < v \}} (dd^c v)^n \leq \int_{\{u < v \}} (dd^c u)^n.
 $$
 If moreover $(dd^c u)^n \leq (dd^c v)^n$ in the weak sense in $\Omega$ then $ u \geq v$ in $\Omega$.
 \end{coro}
 This result implies uniqueness of the solution to the Dirichlet problem when it exists. Moreover using Perron's method of upper envelopes of subsolutions, Bedford and Taylor were able  to solve the Dirichlet problem for the complex Monge-Amp\`ere operator. 
 Let us state the following consequence of their result   which will be used here (\cite{BT76}).
 \begin{theo} \label{thm:DirPb} 
 Let $B \Subset \C^n$ be an euclidean ball, $\mu \geq 0$ a continuous volume form on $\overline{B}$ and $\gamma$ a continuous function in $\partial B$.  Then there exists a  unique psh function $U$ in $B $, which extends as a continuous function in  $\overline B$ solving the following Dirichlet problem:
 
$$
\left\{\begin{array}{ll}
 (dd^c U)^n = \mu, &  \text{weakly in } \, B, \\
U = \gamma  & \text{in} \, \,  \partial B.
\end{array} \right.
 $$
 \end{theo}
 We will also need the following fundamental consequence, known as the "balayage method".
  \begin{coro} \label{thm:Balayage} 
 Let  $\Omega \subset \C^n$ be an open set, $B \Subset \Omega$ is a given euclidean open ball and $\mu \geq 0$ a volume form on $\overline{B}$ with bounded measurable density. Then for any psh function $u $ in $\Omega$, bounded in a neighbourhood of $\overline \B$ and satisfying $(dd^c u)^n \geq \mu$ in the weak sense in $\B$,  there exists a  psh function $U$ in $\Omega$ such that  $U = u$ in $\Omega \setminus \overline{B}$, $U \geq u$ in $\Omega$ and $(dd^c U)^n = \mu,  \text{weakly in } \, B$.
 \end{coro}                                                                                                                                                                                                                                                                                                                                                                                                                                                                                                                                                                                                                                                                                                                                                                                                                                                                                                                                                                                                                                                                                                                                                                                                                                                                                                                                                                                                                                                                                                                                                                                                                                                                                                                                                                                                                                                                                                                                                                                                                                                                                                            
 This result follows directly from \cite{BT76} when the measure has continuous density on $\bar B$ an was extended to the case of a bounded density by U. Cegrell (see \cite{Ceg84}).
 
 \subsubsection{The complex Monge-Amp\`ere operator on compact K\"ahler manifolds}
 
  Now let us explain how to extend the previous tools to compact K\"ahler manifolds following \cite{GZ05}. The notion of psh function makes sense on any complex manifold since it is invariant under holomorphic transformations.
  
 Let $X$ be a (connected) compact K\"ahler manifold of dimension $n$ and let $\omega$ be a closed smooth $(1,1)-$form on $X$.
 Then it is well known that locally in each small coordinate chart $U \subset X$, there exist a smooth function $\rho_U$ such that $\omega = dd^c \rho_U$, the function $\rho_U$ is psh in $U$ and called a local potential of $\omega$ (see \cite{Dem}). Such a local potential is unique up to addition of a pluriharmonic function in $U$
  
  Recall that a function $\f : X \longrightarrow [- \infty, + \infty[$ is said to be $\omega$-plurisubharmonic in $X$ ($\omega$-psh for short) if
  it is upper semicontinuous in $X$ and locally in each small coordinate chart $U$, the function $u = \f + \rho_U$ is psh in $U$, where $\rho_U$ is any local potential of $\omega$ in $U$. 
 
 Let us denote by $PSH (X,\omega) \subset L^1 (X)$ the convex set of $\omega$-psh functions in $X$, where $L^1 (X)$ the Lebesgue space with respect to a fixed smooth non degenerate volume form $\mu_0$ on $X$. 
 
 Then the $(1,1)$-current $\omega_\f := \omega + dd^c \f$ is a closed positive current in the sense of Lelong since locally in $U$ it can be written as $\omega_\f = dd^c u$, where $u = \f + \rho_U$ is psh in $U$  (\cite{Lel68}, \cite{Dem}). 
 It follows  from what was said previously that the complex Monge-Amp\`ere operator is well defined for any bounded $\f \in PSH (X) \cap L^{\infty} (X)$ as a positive $(n,n)$-current on $X$ defined locally in each coordinate chart $U$ as 
 $$
 \omega_\f^n := (dd^c u)^n.
 $$
  This current can be identified to a positive Borel measure on $X$, which will be denoted by $MA (\f) = MA_{\omega} (\f)$ (see \cite{GZ05}).
 Then the maximum principle still holds in this context (see \cite{GZ07}).
 
 \begin{theo} \label{thm:CP} Let $\f, \p \in PSH (X) \cap L^{\infty} (X)$. Then
 $$
  {\bf 1}_{\{\f < \p\}} MA(\max\{\f,\p\}) = {\bf 1}_{\{\f < \p\}} MA(\p),
 $$
 in the sense of positive Borel measures in $X$. \\
 (Maximum Principle).
 
 In particular
 $$
  \int_{\{\f < \p\}} MA(\p) \leq \int_{\{\f < \p\}} MA(\f).
 $$
 (Comparison Principle).
 \end{theo} 
 There is another important result which will be used later.
 \begin{theo}\label{thm:DP}  Let $\f, \p \in PSH (X) \cap L^{\infty} (X)$. Assume that 
 $\p \leq \f$ almost everywhere in $X$ with respect to the measure  $MA(\f).$ Then $\p \leq \f$ everywhere in $X$ \\
  (Domination Principle).  
 \end{theo}
 For more details on these matters we refer to \cite{GZ05,GZ07, Kol98}.
 
 One of the main tools in recent applications of Pluripotential theory to K\"ahler Geometry is the a priori uniform estimate du to Kolodziej (\cite{Kol98, Kol03}. We will use here the following version (see \cite{BGZ08,EGZ09,GZ12}).
 \begin{theo} \label{thm:StabEst}
 Let $f \in L^{p} (X,\mu_0)$ with $p > 1$ and  $\p \in PSH (X,\omega) \cap L^{\infty} (X)$ with $\sup_X \p = 0$. Then  there exists a constant depending on $ \Vert \p \Vert_{L^{\infty} (X)})$  such that for any $\f \in PSH (X) \cap L^{\infty} (X)$ satisfying $MA (\f) \leq f \mu_0$ with $\sup_X \f = 0,$ we have the following "weak stability" estimates 
  $$
   \sup_X (\p - \f)^+ \leq C \Vert f \Vert_{L^{p} (X)}^{1/n} \Vert (\p - \f)^+ \Vert_{L^1 (X)}^{\gamma},
  $$
  where $\gamma = 1/(nq + 2)$ and $q = p/(p-1)$.
  
 In particular we have the following uniform $L^{\infty}$-estimate 
 $$ 
 \Vert \f \Vert_{L^{\infty}} \leq  A \Vert f \Vert_{L^{p} (X)}^{1/n},
 $$
  where $A > 0$ is a uniform constant independent on $\f$.
 \end{theo}
 \vskip 0.3 cm
\subsection{Solving degenerate complex Monge-Amp\`ere equations}
 
 We are mainly interested here in complex Monge-Amp\`ere equations on compact K\"ahler manifolds related to the Calabi conjecture and the existence of K\"ahler-Einstein metrics. Let $X$ be a compact K\"ahler manifold of dimension $n $ and let $\omega$ be a smooth closed semi-positive form on $X$ such that $\int_X \omega^n > 0$. 
  
 We will consider the following global complex Monge-Amp\`ere equation.
 $$
 (\omega + dd^c \f)^n = e^{\e \f} \mu, \leqno (MA)_{\e,\mu}
 $$
 where $\e \geq 0$ and $\mu = f \mu_0$  is a degenerate volume form on $X$ with a density $0 \leq f \in L^{p} (X,\mu_0)$ ($p > 1$) with respect to a fixed  smooth non degenerate volume form $\mu_0$ on $X$ normalized by the condition

 \begin{equation} \label{eq:Norm}
 \int_X \mu = \int_X \omega^n. 
 \end{equation}
 
 
  In the case $\e = 0,  \omega > 0$ is a K\"ahler form on $X$ and $\mu = f \omega^n > 0$ is a smooth non degenerate volume form, the corresponding equation is known as the Calabi-Yau equation and there is a necessary condition for a solution to exists : $\int_X \mu = \int_X \omega^n$. 
 Moreover adding a constant to a solution gives a new solution. E. Calabi observed that the smooth solution of the equation $(MA)_{0,\mu}$ (if it exists) is unique up to an additive constant (\cite{Cal57}). These two facts  make actually this equation more difficult to handle. It was proved  by Yau  (\cite{Yau78}), answering the celebrated Calabi's conjecture, that this equation has a smooth solution $\f$ on $X$ i.e. there exists $\f \in C^{\infty} (X)$ such that $\omega_\f := \omega + dd^c \f > 0$ is a K\"ahler metric on $X$ satisfying the equation $(MA)_{0,\mu}$. In particular he showed that on compact K\"ahler manifolds for which the first Chern class is zero i.e. $c_1 (X) = 0,$ any K\"ahler class contains a (smooth) Ricci-flat K\"ahler-Einstein metric (see also \cite{Tian00}).
 
 When  $\e > 0, \omega > 0$ is a K\"ahler form and $\mu = f \omega^n > 0$ is a smooth non degnerate volume form, the equation 
 $(MA)_{\e,\mu}$ was considered by Aubin and Yau in connection to the problem of existence of K\"ahler-Einstein metrics on a compact K\"ahler manifolds of negative first Chern class i.e. $c_1 (X) < 0$ (or a positive cananonical class i.e. $K_X > 0$). As we will see this equation is much more simpler than the Calabi-Yau equation. The uniqueness is an easy consequence of the Comparison Principle. The existence  of a smooth solution for the equation $(MA)_{\e,\mu}$ was proved in 1978 independently by Aubin and Yau (\cite{Aub78,Yau78}).  
 
 The approach used by Aubin and Yau relies on the continuity method and a priori estimates of high order. It turns out that the a priori $C^0$ estimate is the main step in their approach.
 
 In 1998, S. Kolodziej (\cite{Kol98})  gave a new proof of the a priori $C^0$-estimate using methods from Pluripotential Theory. Moreover, using Yau's theorem he was able to extend it to a slightly more degenerate situation in the case when $\e = 0$ and  $0 \leq f \in L^{p} (X)$ ($p > 1$), $\omega $ being a K\"ahler form. This allows him to obtain continuous weak solution of the equation $(MA)_{0,\mu}$ in the sense of Bedford and Taylor (a pluripotential solution).
 
 This result was extended in \cite{EGZ09} to a more degenerate situation when $\omega \geq 0$ is a closed smooth and semi-positive $(1,1)-$form on $X$ such that $\int_X \omega^n > 0$. The weak solution obtained there was shown to be a bounded $\omega$-psh function, but the continuity was proved under an extra assumption which is satisfied when $\omega > 0$ is K\"ahler.
 
 We first review the main results obtained in \cite{EGZ09}. However we will give a direct approach using pluripotential techniques as developed recently in \cite{EGZ11}, which do not use  the continuity method and high order a priori estimates of Yau and Aubin.
 Namely  we will prove the following result. 
 \begin{theo} \label{thm:THM1}
 Let $X$ be a compact K\"ahler manifold of dimension $n$ and $\omega \geq 0$ be a smooth closed $(1,1)-$form on $X$ such that $\int_X \omega^n > 0$ and $\mu = f \mu_0$ a volume form on $X$ with density $0 \leq f \in L^p (X;\mu_0)$ ($p > 1$) with respect to a fixed smooth non degenerate volume form $\mu_0 > 0$. Then there is a unique $\f \in PSH (X,\omega) \cap L^{\infty} (X)$ which satisfies the complex Monge-Amp\`ere equation  
 $$
 (\omega + dd^c \f)^n = e^{\f} \mu,
 $$
 in the pluripotential sense in $X$.
 
 Moreover the solution $\f$ is the upper envelope of the family of pluripotential subsolutions of the equation in $X$ i.e. 
 $$
 \f = \sup \mathcal F (X,\omega,\mu),
 $$ 
 where
 $$
 \mathcal F (X,\omega,\mu) := \{\p ; \p \in PSH (X,\omega) \cap L^{\infty} (X), (\omega + dd^c \p)^n \geq e^{\p} \mu\}.
 $$
 
 \end{theo}
 
 This theorem implies that for any $\e > 0$ the complex Monge-Amp\`ere equation $(MA)_{\e,\mu}$ has a unique bounded $\omega-$plurisubharmonic solution $\f_\e$.
  It turns out that the family $(\f_\e)$ is uniformly bounded and converges uniformly on $X$ by Theorem~\ref{thm:StabEst}. More precisely we obtain the following result.
  \begin{theo} \label{thm:THM2}
 Let $\omega \geq 0$ be a smooth closed semi-positive $(1,1)-$form on $X$ such that $\int_X \omega^n > 0$ and $\mu = f \mu_0$ a volume form on $X$ with density $0 \leq f \in L^p (X)$ ($p > 1$) with respect to a smooth non degenerate volume form $\mu_0 > 0$ such that $\int_X \mu = \int_X \omega^n$. Then there is a unique $\f \in PSH (X,\omega) \cap L^{\infty} (X)$ which satisfies the complex Monge-Amp\`ere equation  
 $$
 (\omega + dd^c \f)^n = f \mu_0,  
 $$
 in the pluripotential  sense and normalized by $\int_X \f  \mu = 0$.
  \end{theo}
  
  Theorem 1.9 will be proved in subsection 1.3, while Theorem 1.10 will be proved in subsection 1.4.
  These theorems do not say anything about the continuity of the solution. It is however possible to prove continuity and even H\"older continuity of the solution when $\omega > 0$ is a K\"ahler form using pluripotential methods (see \cite{EGZ09}, \cite{Kol08}, \cite{DDHKGZ11}). 
  But these results do not apply in our degenerate situation. 
  
  Nevertheless, in the last section we will show how to combine pluripotential and viscosity techniques to prove continuity in this more general setting.
  
Altogether this will provide an alternative and independent approach to a weak version of  Calabi conjecture \cite{Yau78}:
we will only use upper envelope constructions (both in the viscosity and pluripotential sense), 
a global viscosity and pluripotential comparison principle and Kolodziej's pluripotential 
techniques providing uniform a priori estimates (\cite{Kol98}, \cite{EGZ09}).

This method applies to degenerate equations but yields solutions that are merely continuous (Yau's work yields
smooth solutions, assuming the cohomology class $\{\omega\}$ is K\"ahler and the volume form $\mu$ is
both positive and smooth).

The pluripotential approach applies equally well to a slightly more degenerate situation (see \cite{EGZ11}, \cite{DDHKGZ11}).
 \vskip 0.3 cm
\subsection{The Perron method of upper envelopes}

Before going into the proofs of the results stated in the last section, we will establish a more general result which shows that the Perron method of upper envelopes will provide us with a solution whenever we are able to find a subsolution. This quite general approach might be useful in other situations.

Here we will consider the following degenerate Monge-Amp\`ere equation
\begin{equation} \label{eq:DMAE}
MA_{\omega} (\f) = e^{\f} \mu, 
\end{equation}
where $\omega \geq 0$ is a closed $(1,1)$-form in $X$ with continuous psh local potentials, $MA_{\omega} (\f) := (\omega + dd^c \f)^n$ is the Monge-Amp\`ere measure of $\f$ defined in the weak sense of Bedford and Taylor and $\mu \geq 0$ is a degenerate volume form with $L^1$-density with respect to a fixed smooth volume form.
 
Our aim here is to show that one can
solve this equation in the weak sense of Bedford and Taylor in a rather elementary way, at least when $\omega > 0$ is a K\"ahler form and $\mu$ has a continuous density, by observing that the (unique) solution is the upper envelope of pluripotential subsolutions. 

\subsubsection{Uniqueness of the solution}
Here we will give an easy consequence of the Comparison Principle which will show that the upper envelope of subsolutions of the equation (\ref{eq:DMAE}) is the unique candidate to be a solution.
\begin{prop} \label{prop:UNIQ} Let $\f \in PSH (X,\omega) \cap L^{\infty} (X)$ be a solution to the Monge-Amp\`ere equation (\ref{eq:DMAE}). Then for any 
$\p \in PSH (X,\omega) \cap L^{\infty} (X)$ satisfying the inequality
$  MA_{\omega} (\p) \geq e^{\p} \mu$ in the weak sense of Borel measures on $X$, we have $\p \leq \f $ in $X$.
In particular, the solution of the complex Monge-Amp\`ere equation (\ref{eq:DMAE}) is unique (if it exists).
\end{prop}
\begin{proof}
We are going to show that the set $\{\f < \p\}$ has zero measure with respect to $\mu$. Indeed by the  comparison principle, it follows that
 \begin{eqnarray*}
 \int_{\{\f < \p\}} e^{\p} \mu  & \leq &
 \int_{\{\f < \p\}} (\omega+dd^c \p)^n \\
 &\leq & \int_{\{\f < \p\}} (\omega+dd^c \f)^n \\
 &= &\int_{\{\f < \p\}} e^{\f} \mu  \leq \int_{\{\f < \p\}} e^{\p} \mu.
 \end{eqnarray*}
 Therefore we conclude that $\int_{\{\f < \p\}} ( e^{\f} - e^{\p}) \mu  = 0$ and since $e^{\f} - e^{\p} \leq 0$ on the set $\{\f < \p\}$, it follows that 
 ${\bf 1}_{\{\f < \p\}} \cdot (e^{\f} - e^{\p}) = 0$ $\mu$-almost everywhere on $X $.
 If we know that $\mu$ has a positive density with respect to a fixed smooth non degenerate volume form on $X$, we will conclude that $ \p \leq \f$ almost everywhere in $X$ and then everywhere in $X$ by submean-value inequality in any local chart. 
 In the general case, since  $ e^{ \f} \mu =  MA (\f)$, it follows that the set $\{\f < \p\}$ has measure $0$ with respect to the Monge-Amp\`ere measure $MA (\f)$ i.e. $\p \leq \f$ almost everywhere with respect to $MA (\f)$. It follows from the Domination Principle Theorem~\ref{thm:DP} that $ \p \leq \f$ on $X$.
 This shows that the equation $(MA)_{1,\mu}$ has at most one  solution. 
\end{proof}
 \subsubsection{ Existence of a solution:}
 The previous subsection suggests a natural candidate to be the solution to the Monge-Amp\`ere equation (\ref{eq:DMAE}): the upper envelope of pluripotential subsolutions, in the spirit of the classical Perron's method used in solving the classical Dirichlet problem.
 
 Therefore it is natural to consider the class ${\mathcal F}= \mathcal F (X,\omega,\mu)$ of all pluripotential subsolutions of the equation $(MA)_{1,\mu})$ defined by 
$$
{\mathcal F}:=\left\{ \f \in PSH(X,\omega) \cap L^{\infty}(X) \, / \, MA (\p) \geq e^{\p} \mu
\text{ in } X \right\}.
$$
Now the problem of the existence of a solution remains to prove that ${\mathcal F} \neq \emptyset$ and its upper envelope $\f := \sup {\mathcal F}$ is again a subsolution.
 
 \begin{lem} \label{lem:UB} The class $\mathcal F (X,\omega,\mu)$ is uniformly bounded from above on $X$ and stable under the regularized supremum.
 Moreover it is  compact in $PSH (X,\omega)$ (for the $L^1(X)$-topology).
 \end{lem}
 \begin{proof} We can assume that ${\mathcal F} \neq \emptyset$. 
 We show first that ${\mathcal F}$ is uniformly bounded from above. We can assume without loss of generality that $\omega$ is normalized so that  $\int_X \omega^n = 1$ and then $\int_X \mu=1$ since $\mu$ satisfies the condition (\ref{eq:Norm}). Fix $\p \in {\mathcal F}$. It follows from the convexity
 of the exponential that
 $$
 \exp \left( \int_X \p \mu \right) \leq \int_X e^{\p} \mu = \int_X MA(\p) = \int_X \omega^n = 1.
 $$
 We infer
 $$
 \sup_X \p \leq \int_X \p \mu  +  C_\mu  \leq  C_\mu ,
 $$
 where $C_\mu $ is  a uniform constant that only depends on the fact that all $\omega$-psh functions are integrable
 with respect to $\mu $ (see \cite{GZ05}). This shows that ${\mathcal F}$ is uniformly bounded from
 above by a constant that only depends on $\mu $ and since it is not empty, it is also uniformly bounded from below.
 
 Stability under finite suprema is an easy consequence of the Maximum Principle Thorem~\ref{thm:CP}. If $\p_1, \p_2 \in PSH (X,\omega)    \cap L^{\infty} (X)$ we have
 $$
 MA(\sup\{\p_1,\p_2\}) \geq {\bf 1}_{\{\p_1 \geq \p_2\}}  MA(\p_1) + {\bf 1}_{\{\p_1 < \p_2\}} MA(\p_2).
 $$
 For an infinite family $\mathcal S$ of subsolutions,  the same reasoning can be applied since the regularized supremum of $\mathcal S$ can be approximated almost everywhere by a non  decreasing sequence of finite suprema of subsolutions, and the conclusion follows from the continuity of the complex Monge-Amp\` ere operator for increasing sequences of uniformly bounded $\omega-$psh functions (Theorem~ \ref{thm:CV}).
  
 The compactness can be proved as follows. By the previous considerations, the family ${\mathcal F}$ is relatively compact in $L^1 (X)$ (see \cite{GZ05}). It is then enough to show that it is closed.  Let $(\p_j)_{j \in \N}$ be a sequence of ${\mathcal F}$ converging to $\p \in PSH (X,\omega)$. We know that $\p$ is bounded. We can assume that $\p_j$ converges almost everywhere to $\p$ in $X$. Set $\overline{\p}_j := (\sup_{k \geq j} \p_k)^*$. Then $(\p_k)_{k \in \N}$ is a non increasing sequence of $ PSH (X,\omega) \cap L^{\infty}$ wich converges to $\p$.
 From the previous facts it follows that $MA (\overline{\p}_j) \geq e^{\overline{\p}_j} \mu$ for any $j$ and again by  the continuity of the complex Monge-Amp\` ere operator for non increasing sequences, we conclude that $MA (\p) \geq e^{\p} \mu$ weakly in $X$.
 \end{proof}

 Now it remains to prove that the upper envelope of ${\mathcal F}$ is a solution. This is the content of the following result.
 \begin{theo} \label{thm:MSUB} Let $\mu \geq 0$ be a Borel volume form on $X$ normalized by the condition  (\ref{eq:Norm}). Assume that  the class ${\mathcal F} = \mathcal F (X,\omega,\mu) \neq \emptyset$ is not empty i.e. the complex Monge-Amp\`ere equation $(MA)_{1,\mu }$ admits a subsolution. Then  its upper envelope given by
 $$
 \f:=\sup\{ \p \, / \,  \p \in {\mathcal F}\},
 $$
 is the unique solution to the complex Monge-Amp\`ere equation $(MA)_{1,\mu }$ in the weak sense in $X$ (pluripotential solution). \\ 
 \end{theo}
\begin{proof}
 Indeed, since the class the class ${\mathcal F}$ is not empty and is compact, it follows that its upper envelope is $\omega-$psh in $X$  (see \cite{Hor94}, Proposition 3.4.4), and then it is a subsolution. 
 Moreover by Choquet's lemma, we can find a sequence $\p_j \in {\mathcal F}$ of bounded $\omega$-psh (pluripotential) subsolutions
 such that 
 $$
 \f = (\sup _{j \in \N } \p_j)^*.
 $$

 Observe that by Lemma~\ref{lem:UB}, the family of bounded pluripotential subsolutions is stable under taking maximum so that we can  assume that the $\p_j$'s  form a non decreasing sequence of subsolutions. 
 To see that $\f$ is a subsolution, we use a local balayage procedure to modify each  $\p_j$ on a given "small ball" $B \subset X$ by constructing a new  bounded $\omega-$psh functions $\tilde \psi_j$ on $X$ so that they satisfy  
 the local Monge-Amp\`ere equation $(\omega + dd^c \tilde \p_j)^n = e^{\p_j} \mu$ on $\B$ and $\tilde \psi_j \geq \psi$ on $X$ and $\tilde \psi_j = \psi_j$ on $X \setminus \B$: this is done using Theorem~\ref{thm:Balayage}. By the comparison principle Corollary~\ref{coro:LCP}, it follows that $(\tilde \psi_j)$ is an non increasing sequence of bounded $\omega-$psh functions which increases almost everywhere in $X$ to the function $\f$.
 Since the Monge-Amp\`ere operator is continuous under increasing sequences by Theorem~\ref{thm:CV},
 it follows  that $\f$ is a pluripotential solution of $(MA)_{1,\mu}$ in $B$, hence in all of $X$, as $B$ was arbitrary.
\end{proof}

 \begin{coro} \label{cor:Dag} Assume that  $\mu$ is a Borel volume form on $X$ normalized by the condition  (\ref{eq:Norm}) and satisfying the following condition: $\exists u \in PSH (X,\omega) \cap L^{\infty} (X) , \exists A >  0$ such that
  $$
  \mu \leq A (\omega + dd^c u)^n, \leqno(\dag)
  $$
 in the weak sense of measures in $X$.
  
 Then the class ${\mathcal F} (X,\omega,\mu)$ is not empty, uniformly upper bounded and its upper envelope $\f := \sup {\mathcal F} (X,\omega,\mu)$ is the unique bounded  pluripotential solution to $(MA)_{1,\mu}$ i.e.
 $$
 (\omega + dd^c \f)^n =  e^\f \mu.
 $$
 \end{coro}
  \begin{proof} Set $M := \sup_X u$ and choose $C > 1$ large constant so that $e^{M - C} A \leq 1$. Then by the condition $(\dag)$, the function   
  $\psi_0 := u - C \in {\mathcal F} (X,\omega,\mu)$ is a pluripotential subsolution to $(MA)_{1,\mu }$. Therefore we can apply the previous Theorem.
  \end{proof}

  \subsubsection{ Proof of Theorem~\ref{thm:THM1} }
  
 We want to apply Corollary~\ref{cor:Dag}.  The fact that the family $\mathcal F (X,\omega,\mu)$ is uniformly upper  
 bounded follows from Lemma~\ref{lem:UB}.
 To prove that it is not empty requires several steps.
 
 {\bf 1.} Assume that $\omega > 0$ is K\"ahler and $\mu = f \mu_0$ has a bounded density i.e. $f \in L^{\infty} (X)$.
 Then for a large constant  $A > 0$ we clearly have $\mu \leq A \omega^n$ and then the condition $(\dag)$ is satisfied. Therefore the conclusion of the Theorem follows from Corollary~\ref{cor:Dag}. 
 
 {\bf 2.} Assume that $\omega  > 0$ and $\mu = f \mu_0$ has a density $f \in L^{p} (X)$.
  Then we approximate $\mu$ by volume forms with bounded densities $\mu_j := \inf\{f,j\} \mu_0$ for $j \in \N$ and apply the previous case to solve the equations
 \begin{equation} \label{eq:MAj}
 (\omega + dd ^c \f_j)^n =  e^{\f_j} \mu_j,
\end{equation}
with $\f_j \in PSH (X,\omega) \cap L^{\infty} (X)$.

Let us prove that $(\f_j)$ is bounded in $L^1 (X)$. By \cite{GZ05}, it is enough to check that the sequence $(\sup_X \f_j)$ is bounded.
By Lemma~\ref{lem:UB}, this sequence is upper bounded. To see that it is lower bounded, observe that
$$
e^{\sup_X \f_{j}} \geq \frac{\int_X \omega^n}{\mu (X)}=\int_X \omega^n
$$
hence the sequence $(\sup_X \f_j)$ is  bounded from below. 

We now assert that $(\f_{j})$ is decreasing as $j$ increases to $+ \infty$. Indeed assume
that $1<j \leq k$ and fix $\delta>0$. It follows from the (pluripotential) comparison principle that
$$
\int_{\{\f_k \geq \f_j+\delta\}} (\omega +dd^c \f_k)^n 
\leq \int_{\{\f_k \geq \f_j+\delta\}} (\omega + dd^c \f_{j})^n .
$$
Then using the equations (\ref{eq:MAj}) and the fact that $\mu_k \geq \mu _j$, we infer
$$
{\bf 1}_{\{\f_k \geq \f_j+\delta\}} (\omega  +dd^c \f_k)^n 
\geq e^{\delta}  {\bf 1}_{\{\f_k \geq \f_j+\delta\}} (\omega+dd^c \f_{j})^n
$$
in the sense of Borel measures on $X$.
Therefore it follows that the set $\{\f_k \geq \f_j+\delta\}$ has zero measure with respect to the Monge-Amp\`ere measure $(\omega+dd^c \f_{j})^n$ i.e. the inequality  $ \f_k - \delta \leq \f_j$ holds $(\omega+dd^c \f_{j})^n$-almost everywhere in $X$. From the domination principle it follows that
$  \f_k - \delta \leq \f_j$ everywhere in $X$. As $\delta>0$ was arbitrary, we infer $\f_k   \leq \f_{j}$ in $X$.

We let $\f=\lim_{j \rightarrow + \infty} \f_{j}$ denote the decreasing limit of the functions $\f_j$.
By construction this is an $\omega$-psh function. It follows from Theorem~\ref{thm:StabEst} that $\f$ is a bounded $\omega$-psh function in $X$. Passing to the limit in  (\ref{eq:MAj}) as $j \to + \infty$, we conclude using Theorem~\ref{thm:CV} $\f$ is a (pluripotential) solution to the Monge-Amp\`ere equation $(\omega+dd^c \f)^n = e^\f \, \mu.
$
This shows that $(\dag)$ is satisfied hence we can use Corollary \ref{cor:Dag} to conclude. 

{\bf 3.} Assume that $\omega \geq 0$ and $\mu = f \mu_0$ with $f \in L^p (X)$. Fix a K\"ahler form $\beta$. 
By the above there exists, for each $0<\e \leq 1$, a unique continuous $(\omega+\e \beta)$-psh function
$u_{\e}$ such that
$$
(\omega+\e \beta+dd^c u_{\e})^n = e^{u_{\e}} \mu.
$$
As in the previous case we see  that $\sup_X u_{\e}$ is bounded, as $0<\e \leq 1$.

We now claim that $(u_{\e})$ is decreasing as $\e$ decreases to $0^+$. The proof goes in the same lines as in the previous case. Indeed assume
that $0<\e' \leq \e$ and fix $\delta>0$. Note that $u_{\e'},u_{\e}$ are both $(\omega+\e \beta)$-plurisubharmonic.
It follows from the (pluripotential) comparison principle Theorem~\ref{thm:CP} that
$$
\int_{\{u_{\e'} \geq u_{\e}+\delta\}} (\omega+\e \beta +dd^c u_{\e'})^n 
\leq \int_{\{u_{\e'} \geq u_{\e}+\delta\}} (\omega+\e \beta +dd^c u_{\e})^n .
$$
Since 
$$
(\omega+\e \beta +dd^c u_{\e'})^n \geq (\omega+\e' \beta +dd^c u_{\e'})^n \geq e^{\delta} (\omega+\e \beta +dd^c u_{\e})^n
$$
on the set $\{u_{\e'} \geq u_{\e}+\delta\}$, this shows that the latter set has zero measure with respect to the measure 
$(\omega+\e \beta +dd^c u_{\e})^n$ hence by the domination principle Theorem~\ref{thm:DP}, it follows that $u_{\e'} \leq u_{\e}+\delta$ everywhere in $X$. As $\delta>0$ was arbitrary, we infer $u_{\e'} \leq u_{\e}$ in $X$.

We let $u=\lim_{\e \searrow  0} u_{\e}$ denote the decreasing limit of the functions $u_{\e}$.
By construction this is an $\omega$-psh function in $X$ and by Theorem~\ref{thm:StabEst}, $u$ is bounded and a (pluripotential) solution of the Monge-Amp\`ere equation $ (\omega+dd^c u)^n = e^u \, \mu.$
This shows that the condition  $(\dag)$ is satisfied hence the conclusion follow from Corollary~\ref{cor:Dag}.
\vskip 0.2 cm
\subsubsection{ Proof of Theorem~\ref{thm:THM2}}. 
 We approximate the equation $(MA)_{0,\mu}$ by the perturbed equations $(MA)_{\e,\mu}$, where $\e \searrow 0$. By Theorem~\ref{thm:THM1}, for each $\e > 0$ we can find $\f_\e \in PSH (X,\omega) \cap L^{\infty} (X)$ such that
 \begin{equation} \label{eq:appr}
 (\omega + dd^c \f_\e)^n = e^{\e \f_\e} \mu,
 \end{equation}
 in the pluripotential sense in $X$.
 By convexity of the exponential function, we conclude that $\int_X \f_\e \mu \leq 0$. Therefore by \cite{GZ05}, it follows that there exists  a constant $M > 0$ independent of $\e$ such $\sup_X \f_\e \leq M$.
 On the other hand from (\ref{eq:appr}), if follows that $\sup_X \f_e \geq 0$. Therefore $(\f_\e)$ in bounded in $L^1 (X)$. Then there exists a subsequence $(\f_{\e_j})$, with $\e_j \searrow 0$, which converges in $L^1 (X)$ to a $\f \in PSH (X,\omega)$ and such that $\f_{\e_j} \to \f$ almost everywhere in $X$. We know that $\f = (\limsup_{j \to + \infty} \f_j)^*$. By Theorem~\ref{thm:StabEst}, if follows that $\f_{\e_j}$ is a bounded sequence in $L^{\infty} (X)$ and then $\f \in PSH (X,\omega) \cap L^{\infty} (X)$. Let us define $\tilde \f_j := (\sup_{k \geq j} \f_{\e_k})^*$. 
 Then $(\tilde \f_j)$ is a non increasing sequence of bounded $\omega$-psh functions which converges to $\f$ in $X$.
 Using the comparison principle as in above we see that for any $j \in \N$, we have
 $$
 MA (\tilde \f_j) \geq \inf_{k \geq j} e^{\e_k \f_{\e_k}} \mu.
 $$
 Since $\e_j \to 0$ and $\f_{\e_j}$  is uniformly bounded, it follows that the right hand side converges weakly  to $\mu$ in $X$, while the left hand side converges weakly to $MA (\f)$  by Theorem~\ref{thm:CV}. Hence  $MA (\f) \geq \mu$ weakly in $X$, which implies $MA (\f) = \mu$, since the two volume forms have the same volume in $X$.
 
 Observe that by integrating the equation (\ref{eq:appr}), we get
 $$
\int_X \f d \mu = \lim_{\e_j \to 0} \int_X \frac{e^{\e_j \f_{\e_j} } - 1}{\e_j} d \mu = 0.
$$
Since the above reasonning can be applyed to any subsequence, this proves by uniqueness that the family $(\f_\e)$ converges to $\f$ in $L^1 (X)$ as $\e \to 0$.  
 
 \begin{rem} The previous considerations suggests the following natural question:
 
 \noindent{\bf Question:} Let $\mu \geq 0$ be a Borel volume form on $X$ such that  $\int_X \mu = \int_X \omega^n$.
 Assume that $\mu$ satisfies the condition $(\dag)$ i.e. there exists $u \in PSH (X,\omega) \cap L^{\infty} (X)$ and a constant $A > 1$ such that 
 $$ 
 \mu \leq A (\omega + dd^c u)^n. \leqno (\dag)
 $$ 
 Does there exists $\f \in PSH (X) \cap L^{\infty} (X)$ such that 
$(\omega + dd^c \f)^n = \mu$ and $\int_X \f d \mu = 0$ ?
 
 Investigating this question, a natural idea is to follow the same strategy as in the proof Theorem~\ref{thm:THM2} above.
 
Indeed,  using the condition $(\dag)$ and Chern-Levine-Nirenberg inequality (se \cite{GZ05}) we can easily get a uniform control on the Monge-Amp\`ere energies of the approximating sequence i.e.  there exists a uniform constant $C > 0$ such that 
 $$
 \int_X \vert \f_\e \vert (\omega + dd^c \f_\e)^n \leq C,
 $$
 for any $\e > 0$. This implies that $\f \in \mathcal E^ 1 (X,\omega)$, the class of potentials of finite Monge-Amp\`ere energy (see \cite{GZ07}).
 Then as in the proof above we can conclude that that $(\omega + dd^c \f)^n = \mu$ weakly on $X$.
 Actually it is possible to show that  $\f \in \mathcal E^p (X,\omega)$ for any $p > 1$ (see \cite{GZ07}).
 Unfortunately we do not know if  $\f$ is bounded in $X$.
 
 Finally observe that by a result of  Ko\l odziej locally in each coordinate chart which is a strongly pseudoconvex domain, such a measure is the Monge-Amp\`ere of a bounded $\omega$-psh function (see \cite{Kol05}).
  It is raisonnable to conjecture that this is also the case globally.
  
  It is interesting to obeserve that when the measure $\mu$ is strongly dominated by the Monge-Amp\`ere capacity, then the solution is continous (see \cite{Kol98}, \cite{Kol03}, \cite{Kol05}, \cite{EGZ09}, \cite{BGZ08}).     
  \end{rem}

\section{The viscosity approach to degenerate non linear PDE's}

Before we introduce the definitions of viscosity sub(super)solutions, let us give as a motivation some examples of degenerate elliptic PDE's to which viscosity methods can be applied. In particular we will give examples where the notion of generalized or weak solution does not make sense.
\subsection{Classical solutions}
Let us start by general considerations.
 A fully non linear second order PDE can be written in the following general form 
 \begin{equation} \label{eq:DEEQ}
 F (x,u,Du,D^2u) = 0,
 \end{equation}
 where $F : \Omega \times \R \times \R^N \times \mathcal S_N \longrightarrow \R$ is a function  satisfying some conditions to be made precise in a   while,  $\Omega \sub \R^N$ is an open set and $\mathcal S_N$ is the space of real symmetric matrices of order $N$. 
 
 We will say that $u : \Om  \longrightarrow \R$ is a classical solution  of the equation (\ref{eq:DEEQ}) if $u$ is $C^2$-smooth in $\Om$ and satisfies the   differential identity  
 $$
 F (x,u (x),Du (x),D^2u (x)) = 0, \forall x \in \Om.
 $$ 
 It is quite natural to split the equation $F = 0$ into the two different inequalities $F \leq 0$ and $F \geq 0$. Then if $u$ satisfies the differential inequality $F (x,u (x),Du (x),D^2u (x)) \leq 0$ (resp. $F (x,u (x),Du (x),D^2u (x)) \geq 0$) pointwise in $\Om$, we will say that $u$ is a classical subsolution (resp. supersolution) of the equation (\ref{eq:DEEQ}).
 Therefore $u$ is a classical solution of the equation (\ref{eq:DEEQ}) iff $u$ is a classical subsolution and a classical supersolution of the equation (\ref{eq:DEEQ})

 In order to apply the viscosity approach to the equation (\ref{eq:DEEQ}), we need to impose the following FUNDAMENTAL condition on $F$. \\
 
 \noindent{\it  \underline{Degenerate ellipticity condition} :} for any $x \in \Omega, s \in \R, p \in R^N, Q_1, Q \in \mathcal S_N $, we have
$$
  Q \geq 0 \Longrightarrow F (x,s,p,Q_1 + Q) \leq F (x,s,p,Q_1). \leqno (DEC)
$$
Here $Q \geq 0$ means that the symmetric matrix $Q $ is semi-positive i.e. all its eigenvalues are non negative.

 The reason why this condition is important for viscosity methods to apply will appear soon. 
 \vskip 0.3 cm
\subsection{Examples} 
 Here we are mainly interested in non linear PDE's. However to enlighten the reader about the necessity of this condition to apply viscosity methods, we will recall some basic facts from the theory of linear elliptic second order  PDE's. 
 
 \vskip 0.2 cm
\noindent{\bf Example 1 : Hamilton-Jacobi equations } \\
These are first order equations of the type
\begin{equation} \label{eq:HJ}
 H (x,u,Du) = 0, \text{in} \, \,  \Om,
\end{equation}
associated to a continuous Hamiltonian function $H : \Om \times \R \times \R^N  \longrightarrow \R$, where $\Om \sub \R^n$ is an open set.

The simplest example to keep in mind is the Eikonal equation corresponding to the Hamiltonian function $H (u) := \vert D u  (x)\vert - 1$ defined on $]- 1,+1[ \times \R$.
This example will help us to understand the viscosity concepts. Let  us consider the Dirichlet problem for the Eikonal equation: 
 \begin{equation} \label{eq:Eik}
  \vert u'(x)\vert - 1 = 0, \ \ u(\pm1) = 0.
  \end{equation}
 It is quite clear that this equation has no classical solution. Indeed, a classical solution to this equation should be a $C^1$-smooth function satisfying the equation $\vert u'(x)\vert = 1$ pointwise in $]-1,1[$ and the
 the boundary condition $u(-1) = u (1) = 0$. Such a function do not exist, since by Rolle's theorem it should have at least a critical point in $]-1,1[$.  
 
 However the differential equation  (\ref{eq:Eik}) has plenty of generalized solutions i.e.  functions $u \in W^{1,\infty} (]-1,1[)$ satisfying the equation $\vert u'(x)\vert = 1$ almost everywhere in $]-1,1[$. Indeed the function $u_0 (x) := 1 - \vert x\vert$ is a generalized solution to the Dirichlet problem (\ref{eq:Eik}). It is easy to cook up piecewise affine functions that satisfies (\ref{eq:Eik}) on $[-1,+1]$, except a given finite set.  Observe that if $u \in W^{1,\infty} (]-1,+1[)$ is a generalized solution to the equation associated to the Hamiltonian function $H (u)$ then $- u$ is  a generalized solution of the equation associated to the Hamiltonian $\tilde H (u) := - H (-u)$. From the point of view of generalized solutions, the two corresponding equations are the same and $u$ and $- u$ are two different solutions to the same Dirichlet problem. However as we will see, from viscosity point of view they should be considered as different since they correspond to different Hamiltonian functions. Namely, we will see in Example 2.11 below that $u_0$ is the unique viscosity solution of the Dirichlet problem for the Hamiltonian function $H$ with boundary values $0$, while $- u_0$ is the unique viscosity solution of the Dirichlet problem for the Hamiltonian function $\tilde H$ with boundary values $0$.

  \vskip 0.2 cm
 \noindent{\bf Example 2 : Elliptic  second order equations: } \\
 An important class of elliptic  second order PDE's are the quasi-linear ones, given by 
\begin{equation} \label{eq:EQL}
 - \sum_{j,k} a^{j,k} (x) \frac{\bd^2 u}{\bd x_j \bd x_k} +  H (x,u,Du) = 0,  \, \, \text{in}  \, \, \Om,
 \end{equation}
 where  $a = (a^{j,k})$ is an $N \times N$ symmetric matrix valued function with continuous entries on $\Om$  satisfying the (uniform) ellipticity condition
  \begin{equation} \label{eq:Elli}
  \sum_{j,k} a^{j,k} (x) \xi_j \xi_k \geq \nu \vert \xi\vert^2, \, \, \forall x \in \Om, \, \, \forall \xi \in \R^N,
  \end{equation}
  where $\nu > 0$ is a uniform constant.
  
  These equations are of the type (\ref{eq:DEEQ}) associated to the following Hamiltonian:
  $$
  F (x,s,p,Q) := - \text{Tr} (A (x) Q) + H (x,s,p),  
  $$
  where  $(x,s,p,Q) \in \Omega \times \R \times \R^N \times \mathcal S_N$. 
  
  Then it is easy to see that the degenerate ellipticity condition for $F$ i.e. the monotonicity property of $F$ with respect to the partial order on symmetric matrices is a generalisation of the ellipticity condition (\ref{eq:Elli}) as the following exercise shows.
 
 \begin{exe}
 Let $A \in \mathcal S_N$ be a real symmetric matrix of order $N$ such that for any $Q \in \mathcal S_N$ with $Q \geq 0$, we have
 $Tr (A\cdot Q) \geq 0$. Then $A \geq 0$.
 \end{exe}
 \vskip 0.2 cm
 When $H (x,p) = <b (x), p> + c (x) s + d (x)$, where $b : \Omega \longrightarrow \R^N$ is continuous vector field and $c, d : \Omega  \longrightarrow \R$ are a continuous functions, the equation is a linear second order PDE.
 
 The simplest and fundamental example is the Laplace equation equation $ - \Delta = f$ or more generally the Helmholtz equation given by $ - \Delta u + c u = f$, where  $c \in \R$ is a constant. 
 
 Denote by  $\Delta_c = := - \Delta + c$ the Helmholtz operator. Then it is  well know that the equation  $\Delta_c u = f$ has a weak solution $u \in W_0^{1,2} (\Om)$ for any $f \in L^2 (\Om)$ iff $- c \notin \Lambda,$  where $\Lambda \sub \R^+$ is the spectrum of the operator $- \Delta$ (for the Dirichlet problem with zero boundary values) which is known to be a discrete sequence of positive real numbers $\la_k  
 \nearrow + \infty$ (this follows from Fredholm's aternative).
 
 In particular when $c \geq 0$, the equation $- \Delta u + c u = f$ has a weak solution $u \in W_0^{1,2} (\Om)$ when $f \in L^2 (\Om)$. Moreover the weak solution  is unique  
 since  the elliptic operator $\Delta_c$ satisfies the maximum principle precisely when $c \geq 0$.
 Recall also that by Schauder's theory for elliptic operators the solutions are smooth whenever $f$ is smooth (see \cite{GT83}).

 A typical example of non linear but quasi-linear second order elliptic equation is the following one
 \begin{equation} \label{eq:HJV}
   - \e \Delta u +  H (x,u,Du) = 0,
 \end{equation}
 where $\e > 0$ is small. This equation can be considered as a small perturbation of the Hamilton-Jacobi equation $H (x,u,Du) = 0$.
 The small perturbation term $- \e \Delta u$ is called a {\it viscosity term} (in Fluid mechanics). In standard cases, the equation is uniformly  
 elliptic and then it's possible to find a unique $C^2$-smooth solution $u_\e$ of the  equation (\ref{eq:HJV}) with suitable boundary conditions  
 and  get uniform $L^{\infty}-$estimates of $u_\e$ and $\nabla u_\e$ independent of $\e > 0$. This implies by Ascoli's theorem that some  
 subsequence will converge uniformly to a continuous function $u$, but the corresponding subsequence $\nabla u_\e$ will converge only weakly in  
 $L^{\infty}$.  This is however not sufficient to pass to the limit in (\ref{eq:HJV}) as $\e \searrow 0$ to get a generalized or weak solution to  
 the Hamilton-Jacobi equation (\ref{eq:HJ}).
 Nevertheless, it is reasonable to consider that the function $u$ should be a solution of the equation (\ref{eq:HJV}) in some sense.
 Indeed, we can show by using an easy stability argument for viscosity solutions, that it will be possible to pass to the limit in the sense of  
 viscosity and get a viscosity solution to the Hamilton-Jacobi equation (\ref{eq:HJ}). This method, known as the "vanishing viscosity method",  
 motivates the introduction of the viscosity concepts and justifies the terminology of viscosity (see \cite{CIL92}).

This method can be applied to the Eikonal equation and explains why we should consider the two Hamiltonians $H (u) = \vert u' \vert - 1$ and $\tilde H (u) = 1 - \vert u'\vert$ as different since the corresponding elliptic perturbations approximating them are different. 
 
 \vskip 0.2 cm
 \noindent{\bf Example 3 :  Degenerate Real Monge-Amp\`ere equations} \\
 This equation is of the following type
 \begin{equation}
 - \text{det} (D^2 u) + f (x,u,Du) = 0, \ \text{in} \, \,  \Om,
 \end{equation}
 where $\Om \sub \R^N$ is a convex domain,  the solution being a convex function $u : \Om \longrightarrow \R$  and $f : \Om \times \R \times \R^N  
 \longrightarrow \R^+$ is a continuous non negative function on $\Om$, non decreasing in $u$.\\
 The above equation in degenerate elliptic if restricted to an appropriate convex subset of the space of symmetric matrices. Namely if we define the Hamiltonian function as follows
 $$
  F (x,s,p,Q) := - \text{det} (Q) + f (x,s,p), \ \text{if} \, Q \geq 0 \  \text{and} \ F (x,s,p,Q) = + \infty, \  \text{if not}. 
 $$
 Then $F$ is lower semi-continuous on $\Om \times \R \times \R^N  \times \mathcal S_n$, continuous on its domain $\{ F < + \infty\}$ and the equation $F (x,u,Du,D^2 u) = 0$ is degenerate elliptic.
 \vskip 0.3 cm
 \noindent{\bf Example 4 : Degenerate Complex Monge-Amp\`ere equations} \\
 
 We will  consider degenerate complex Monge-Amp\`ere equations on open sets $\Omega \subset \C^n$:
 \begin{equation} \label{eq:DMAElocal}
- \text{det} \left(\frac{\bd^2 u}{\bd z_j \bd \bar{z}_k}\right) + f (z,u,Du) = 0, \ \text{in} \ \Om,
 \end{equation}
  the solution should be a bounded plurisubharmonic function $u : \Om \longrightarrow \R$  and $f : \Om \times \R \times \R^{2 n} \longrightarrow \R^+$ is a non negative continuous function, monotone increasing in $u$.
 This equation can be written as
 $ - (dd^c u)^n + f (z,u,Du) \beta^n = 0$.
 As in the real case, this equation is degenerate elliptic when restricted to an appropriate convex subset of the space $\mathcal H_{n}$ of hermitian matrices. More precisely, identifying $\C^n$ with $\R^{2n}$,  let us define the following Hamiltonian function:
 $$
  F (z,s,p,Q) := - \text{det} (Q^{1,1}) + f (z,s,p), \ \text{if} \, Q^{1,1} \geq 0, 
 $$
 and $ \ F (z,s,p,Q) = + \infty, \  \text{if not},$
 where $(z,s,p, Q) \in \Omega \times \R \times \C^n \times \mathcal S_{2n}$ and $Q^{1,1} \in \mathcal H_n$ is the hermitian $(1,1)$-part of
 $Q \in \mathcal S_{2n}$ considered as a real quadratic form $Q $ on $\C^n$.
 
 Then again $F$ is a lower semi-continuous  degnerate elliptic Hamiltonian function  on $\Omega \times \R \times \C^n \times \mathcal S_{2n}$, continuous on its domain $\{ F < + \infty\}$ and the equation can be written as $F (x,u,Du, dd^c u) = 0$ , where $dd^c u (x)$ is the complex hessian of $u$ i.e. precisely the hermitian $(1,1)$-part of the quadratic form $D^2  u (x) \in \mathcal S_{2n}$.
 
 We are mainly interested in degenerate complex Monge-Amp\`ere equations on a compact K\"ahler manifold $X$ of the following type
 $$
 - (\omega + dd^c \f )^n + e^{\e \f} \mu = 0,
 $$
 where $\omega \geq 0$ is a closed real semi-positive $(1,1)-$form  on $X$ such that $\int_X \omega^n > 0$ and $\mu \geq $ is a continuous volume form on $X$ such that $\int_X \mu = \int_X \omega^n$.
 
 Locally this equation can be written as a complex Monge-Amp\`ere equation of the type considered (\ref{eq:DMAElocal}), so the degenerate ellipticity condition will be satisfied in an appropriate sense as we will see in the next section.

\vskip 0.3 cm
 \subsection{Definitions of viscosity concepts}
 We want to consider fully non linear degenerate elliptic equations. As we have seen above, we will mainly consider equations  for  which  we cannot expect in general to find classical solutions (i.e. smooth)  or generalized (i.e. in Sobolev spaces) or   even weak solutions (i.e. distributions). 

 On the other hand, it is well know that the classical Maximum Principle is a fundamental tool in the study of  (uniformly) elliptic and parabolic equations,  when using Schauder theory to get smooth solutions. 
 Indeed the basic idea for solving these equations with prescribed boundary conditions (e.g. in  the Dirichlet problem) lies in the construction of ad hoc barriers i.e. subsolutions and supersolutions satisfying the prescribed boundary conditions and the possibility to compare them by using the Maximum Principle.
 
 Therefore we need to define a new notions of "weak"  subsolution and supersolution and find a substitute for the classical maximum principle which allows  to prove uniqueness of the a solution when subsolutions and supersolutions with appropriate boundary conditions exist. Once the Comparison Principle holds, the existence is usually proved using the Perron method of envelopes of subsolutions.

We will assume in all the rest of this paper that the function $F$ satisfies the following two important conditions which will play a fundamental role in establishing the Viscosity Comparison Principle to get uniqueness of the solution.\\

\noindent{\bf Hypotheses :}

 \noindent{\it 1. \underline {Degenerate ellipticity condition}  :} $\forall  x, \in \Omega, \ s \in \R, \ p \in R^N, \ Q_1, Q \in \mathcal S_N $, 
$$
  Q \geq 0 \Longrightarrow F (x,s,p,Q_1 + Q) \leq F (x,s,p,Q_1). \leqno (DEC)
$$
{\it  2. \underline {Properness condition} :} $\forall x \in \Omega, \forall (s_1, s_2) \in \R^2, \forall p \in R^n, \forall Q \in \mathcal S_N$,
$$
 s_1 \leq s_2  \Longrightarrow F (x,s_1,p,Q) \leq F (x,s_2,p,Q). \leqno (PRC)
$$
Observe that this last condition is satisfied when $F$ does not depend on $u$, but in this case it is sometimes harder to prove a comparison principle.

A function $F$ satisfying the degenerate ellipticity condition $(DEC)$ and the properness condition$(PRC)$ will be called a {\it Hamiltonian function} and the equation (\ref{eq:DEEQ}) will be called the degenerate elliptic equation associated to the Hamiltonian function $F$.

It is important to understand that as in the linear case, when $F$ is a Hamiltonian function in the above sense, the function $- F$ is not unless it does not depend neither on $u$ nor on $D^2 u$. So this means that the methods of viscosity can be applied to $F$ but not to $- F$.
And even when the function $F$ depends only on $Du$ as for the Eikonal example, we should distinguish between the two equations.
 
The fundamental idea behind the notion of viscosity solution is provided by the following elementary result which emphasizes the role of the Maximum principle and  will serve as a motivation for the general defiK\"ahlernition to be introduced below.
\begin{prop} (Smooth solutions). Assume that $F$ is degenerate elliptic and let $u \in C^2 (\Om)$. Then we have the following properties: 

1. The function $u$ is a classical subsolution of the equation (\ref{eq:DEEQ}))  iff the following condition holds : 

{\bf (Sub)}: For any $x_0 \in \Om$ and any $C^2-$smooth function $\f$ in a neighbourhood of $x_0$ such  $u - \f$ takes its local maximum at $x_0$ (we will say that $\f$ touches $u$ from above at $x_0$ and write $u \leq_{x_0} \f$) we have
$$
 F (x_0,u (x_0),D \f (x_0),D^2\f (x_0)) \leq  0.
$$

2. The function $u$ is a classical supersolution of the equation (\ref{eq:DEEQ}))  iff the following condition holds : 

{\bf (Super)}: For any $x_0 \in \Om$ and any $C^2-$smooth function $\p$ in a neighbourhood of $x_0$ such  $u - \p$ takes its local minimum at $x_0$ (we will say that $\p$ touches $u$ from below at $x_0$ and write $u \geq_{x_0} \p$) we have
 $$
 F (x_0,u (x_0),D \p (x_0),D^2\p (x_0)) \geq  0.
 $$
\end{prop}

 A $C^2-$smooth function $\f$ in a neighbourhood of $x_0$ satisfying the condition $u \leq_{x_0} \f$ is called an upper test function for $u$ at $x_0$ and a $C^2-$smooth function $\p$ in a neighbourhood of $x_0$ satisfying the condition $u \geq_{x_0} \p$ is called a lower test function for $u$ at $x_0$.
 
 This result shows that the application of the classical maximum principle and the use of the degenerate ellipticity condition allows to transfer the differentiation from $u$ to upper and lower  $C^2$-test functions in a neighbourhood of each point and ask for the differential inequalities $F \leq 0$ and $F \geq 0$ to hold for the corresponding test function at the given point. 

\begin{proof} It is enough to prove the first part. It is clear that he condition {\bf(Sub)} is sufficient for $u$ to be a classical subsolution. Indeed, since $u$ is $C^2$, it can be taken as an upper test function  at any point and then it satisfies the corresponding differential inequality.

Let us prove that the condition {\bf(Sub)} is necessary for $u$ to be a classical solution. Indeed assume that  $\f$ be a $C^2-$smooth function  in a neighbourhood of $x_0$ such that $u \leq_{x_0} \f$. Then $u - \f$ is a $C^2-$smooth function  in a neighbourhood of $x_0$ which attains its local maximum at $x_0$. By the local maximum principle we have $D (u  - \f) (x_0) = 0$ and $D^2 (u  - \f) (x_0) \leq 0$ in the sense of quadratic forms (or symmetric matrices).
 Since $D \f (x_0) = D u (x_0)$ and  $D^2 u (x_0) \leq D^2 \f (x_0) $ in the sense of symmetric matrices and $u$ is a classical subsolution,
it follows from the degenerate ellipticity condition that
$$
 F (x_0,u (x_0),D \f (x_0),D^2\f (x_0)) \leq  F (x_0,u (x_0),D u (x_0),D^2 u (x_0)) \leq 0,
$$
 which proves the condition {\bf (Sub)}.
  \end{proof}

Observe that the main feature of this  characterization is to show that the conditions ${\bf (Sub)}$ and ${\bf (Super)}$  use only the values of $u$ but not its first nor second derivatives. Therefore it can be used as a motivation for the following general definitions.

 \begin{defi} 1. Let $u : \Om \longrightarrow \R$ be an upper semi-continuous (usc) function  in an open set $\Om \subset \R^N$. We say that 
 $u$ is a viscosity subsolution of the equation $F (x,u, D u, D^2 u) = 0$ on $\Om$ if it satisfies the condition {\bf (Sub)}. We will also  say that $u$ satisfies the differential inequality $F (x,u, D u, D^2 u) \leq 0$ in the viscosity sense on $\Om$. \\
 2. Let $u : \Om \longrightarrow \R$ be a lower semi-continuous (lsc) function  in an open set $\Om \subset \R^N$. We say that 
 $u$ is a viscosity supersolution of the equation $F (x,u, D u, D^2 u) = 0$ on $\Om$ if it satisfies the condition {\bf (Super)}. We  
 will also say that 
 $u$ satisfies the differential inequality $F (x,u, D u, D^2 u) \geq 0$ in the viscosity sense on $\Om$. \\
\end{defi}

 To illustrate the importance of the Properness condition $(PRC)$, let us give a simple case where it helps to prove uniqueness.
 \begin{theo} Let  $\Om \Sub \R^N$ be a bounded domain and assume that the Hamiltonian function  $F (x,s,p,Q)$ is strictly increasing in the variable $s$. Then the classical comparison principle holds i.e. if $u \in C^2 (\Om) \cap C^0 (\overline \Om)$ is a classical subsolution of the equation (\ref{eq:DEEQ}) and $v \in C^2 (\Om) \cap C^0 (\overline \Om) $ is a classical supersolution of the equation (\ref{eq:DEEQ}) such that $u \leq v$ on $\bd \Om$ then $u \leq v$ on $\Om$.
 In particular the equation  (\ref{eq:DEEQ}) has  at most one classical solution with prescribed continuous boundary values.
 \end{theo}
 \begin{proof} Since $u - v$ is continuous on the compact set $\overline \Om$, it attains its maximum at some point $x_0 \in \Om$ i.e.
 $\max_{\overline \Om} (u - v) = u (x_0) - v (x_0)$.
 If $x_0 \in \bd \Om$ then $u (x_0) \leq v (x_0)$ and then we are done.
 Now assume that $x_0 \in \Om$. Since $u$ and $v$ are $C^2$ at $x_0 \in \Om$, it follows from the classical maximum principle that $D u (x_0) = D v (x_0)$ and $D^2 u (x_0) \leq D^2 v (x_0)$.
 Now since $u$ is a classical subsolution we have 
 $$
 F \left(x_0,u (x_0), D u (x_0), D^2 u (x_0)\right) \leq  0 \leq F \left(x_0,v (x_0), D v (x_0), D^2 v (x_0)\right).
 $$
 Therefore by the degenerate ellipticity condition ${\bf (DEC)}$ we have
 $$
 F (x_0,u (x_0), D u_0(x), D^2 u (x_0)) \leq F (x_0,v (x_0), D u_0(x), D^2 u (x_0)).
 $$
 From the Properness condition ${\bf (PRC)}$ it follows that $u (x_0) \leq v (x_0)$.
 \end{proof}K\"ahler
 Observe that the simple reasoning above uses the fact that $F$ is increasing in a crucial way. In the situation where $F$ does not depend on $u$ for example, we cannot conclude so easily. 
 However one can show that the conclusion is still true but the poof requires a more subtle argument based on a more refined Maximum Principle known as the Alexandroff-Backelman-Pucci maximum principle (see \cite{CC95}, \cite{Wang10}). 
 
 \begin{rem}  Observe that in the last result it is enough to assume that only one of the functions is a classical subsolution or a classical supersolution. Indeed assume for example that $u$ is a classical subsolution, while $v$ is a viscosity supersolution with $u \leq v$ on $\bd \Omega$. Then arguing as above, we get the inequality  $ u (x) - u (x_0) + v (x_0) \leq v (x)$ in a neighbourhood of $x_0$; which means that the $C^2-$function $q (x) :=  u (x) - u (x_0) + v (x_0)$ is a lower test function for $v$ at $x_0$. Therefore the VSC inequality for $v$ at $x_0$ implies that
 $$
  F (x_0, v (x_0), D q (x_0),D^2 q (x_0) \geq 0.
 $$
 On the other hand, since $u$ is a classical subsolution, we have 
 $$  F (x_0, u (x_0), D u (x_0),D^2 u (x_0) \leq 0.$$
 Comparing these two inequalities we get 
 $$F (x_0, u (x_0), D u (x_0),D^2 u (x_0) \leq F (x_0, v (x_0), D q (x_0),D^2 q (x_0).$$
 Since $F$ is strictly increasing, it follows that $u (x_0) \leq v (x_0)$.

 \end{rem}
 
 The main goal of the first part of this lecture is to prove a general comparison principle for viscosity solutions. We want to do the same  
 reasoning as above, but our functions are not smooth. We therefore need to approximate them keeping the memory of the viscosity differential inequalities they satisfy. This will be done in the next section.
 \vskip 0.3 cm
 \subsection{Characterization of viscosity concepts by mean of jets}
 
 As we have seen in the previous proofs, the only important thing that matters for the differential inequalities we were proving is the jet of  
 order $2$ of the function $u$ at a given point.
 Therefore to deal with non smooth functions it is useful to develop a sub-differential calculus and define sub(super)-jets of order $2$. This will  lead to a characterization of viscosity concepts by means of sub(super)-jets of order $2$, which is more flexible.

 \begin{defi}
 1. Let $u : \Om \longrightarrow \R$ be an usc function and $x_0 \in \Om$. The super-differential jet of order $2$ of $u$ at $x_0$ is the set
  $J^{2,+} u (x_0)$ of all $(p,Q) \in \R^N \times \mathcal S_N$ such that for any $\xi \in \R^N$ with $\vert \xi\vert << 1,$ the following inequality holds
  $$
   u (x_0 + \xi) \leq u (x_0) + p \cdot\xi + \frac{1}{2} <Q\cdot\xi,\xi> + o (\vert \xi\vert^2).
  $$K\"ahler
  2. Let $u : \Om \longrightarrow \R$ a lsc function and $x_0 \in \Om$. The sub-differential jet of order $2$ of $u$ at $x_0$ is the set
  $J^{2,-} u (x_0)$ of all $(p,Q) \in \R^N \times \mathcal S_N$ such that for any $\xi \in \R^N$ with $\vert \xi\vert << 1,$ the following inequality holds
  $$
   u (x_0 + \xi) \geq u (x_0) + p \cdot\xi + \frac{1}{2} <Q\cdot\xi,\xi> + o (\vert \xi\vert^2).
  $$
  3. If $u$ is continuous we can define the differential jet of order $2$ of $u$ at $x_0$ as the set ${J}^2 u (x_0) := J^{2,+} u (x_0) \cap J^{2,-} u (x_0)$.
 \end{defi}
 Observe that if $u$ is twice differentiable at $x_0$ then 
 $$
 {J}^{2,+} u (x_0)= \{(D u (x_0),Q) ; Q \geq D^2 u (x_0)\},
 $$
 and
 $$
 {J}^{2,-} u (x_0)= \{(D u (x_0),Q) ; Q \leq D^2 u (x_0)\},
 $$
 so that $J^2 u (x_0) = \{(D u (x_0), D^2 u (x_0)\}$.
 
 For an arbitrary upper semi-continuous function, it may happen that the set ${J}^{2,+} u (x_0)$ is empty. However there are many points nearby where this set is not empty as the following remark shows.
 \begin{rem}
 Observe that the function  $ u (x) := \vert x\vert$, which is a convex non negative $\R$, satifies the condition  ${J}^{2,+} u (0) = \emptyset$. However we are giong to see that for an upper semi-continuous function $u$ which is bounded from above, there are many points where $J^{2,+} u (x_0) \neq \emptyset$. 
 Actually the set of super-differentiability of $u$ at second order defined by
 $$
 \mathcal D_{\Omega}^{2,+} u := \{ x \in \Omega ; J^{2,+} u (x) \neq \emptyset\}
 $$ 
 is dense in $\Omega$. 
 
 Indeed, fix a point $x_0 \in \Omega$. Since $u$ is upper semi-continuous at $x_0$, for any  ball $B = B (x_0,r) \Subset \Omega$ with $r > 0$ small enough there exists $A > 0$ such that 
 $ u (x) -  A \vert x - x_0\vert^2 < u (x_0)$ for $\vert x - x_0\vert = r$. Then defining the function $q$ by $q (x) := A \vert x - x_0\vert^2$, we see by upper semi-continuity that the function $u - q$ takes its maximum $M$ in $\bar B$  at some point $\hat x \in \bar B$. Now observe that if $\vert \hat x - x_0 \vert = r$ then  $M = u (\hat x) - q (\hat x) <  u (x_0) = u (x_0) - q_A (x_0)$, which contradicts the fact that $M$ is the maximum of $u$ in the ball $\bar B$. Therefore $\hat x \in B$ and then the function $ \hat q := q - q (\hat x) + u (\hat x)$ is a $C^2$-smooth upper test function for $u$ at the point $\hat x \in B$ which means that  $(D \hat q (\hat x), D^2 \hat q (\hat x)) \in J^{2,+} u (\hat x)$.
 
 As we will see the fundamental theorem of Alexandrov says that for a convex function function, the set $\mathcal D_{\Omega}^{2,+} u$ is not only dense in $\Omega$ but it is of full Lebesgue measure in the sense that its complement in $\Omega$ is  of Lebesgue measure $0$ (see Theorem~\ref{th:A}). 
 The same remarks holds for a lower semi-continuous function which is bounded from below.
 \end{rem}
 Since viscosity sub(super)-solutions are only usc(lsc) functions, it is necessary to extend the previous definitions by introducing the notions of approximate super(sub)-differential jets. 
 \begin{defi}
 Let $u : \Om \longrightarrow \R$ be an usc function and $x_0 \in \Om$ and $(p,Q) \in \R^N \times \mathcal S_N$. 
  We say that $(p,Q) \in  {\bar J}^{2,+} u (x_0)$ if there exists a sequence of points $y_j \to x_0$ in $\Omega$ and a sequence $(p_j,Q_j) \in J^{2,+} u (y_j)$ such that 
 $(p_j,Q_j) \to (p,Q)$.
 In the same way we define ${\bar J}^{2,-} u (x_0)$ for a lower semi-continuous function $u : \Omega \longrightarrow \R$.
 \end{defi}
 Then we have the following important characterization of viscosity solutions which will be useful.
 \begin{theo} 1. Let $u : \Om \longrightarrow \R$ be an usc function and $x_0 \in \Om$. Then
 $u$ is a viscosity subsolution of the equation $F (x,u,Du,D^2 u) = 0$ if and only if for any $x_0 \in \Om$ and any 
  $(p,Q) \in {\bar J}^{2,+} u (x_0)$, we have $F (x_0,u(x_0),p,Q) \leq 0$.
  
  2. Let $u : \Om \longrightarrow \R$ be a lsc function and $x_0 \in \Om$. Then
 $u$ is a viscosity supersolution of the equation $F (x,u,Du,D^2 u) = 0$ if and only if for any $x_0 \in \Om$ and any 
  $(p,Q) \in {\bar J}^{2,-} u (x_0)$, we have $F (x_0,u(x_0),p,Q) \geq 0$.
 \end{theo}
 \begin{proof}
 It is enough to prove the first claim.  To prove that the condition is sufficient, it is enough to prove that if $\phi$ is an upper test function for $u$ at some point $x_0$ then $(D \phi (x_0), D^2 \phi (x_0)) \in J^{2,+} u (x_0)$. Indeed by Taylor's formula for $\vert \xi\vert << 1$ and $x = x_0 + \xi \in \Om$, we have
 $$
  \phi (x) = \phi (x_0) +  D \phi (x_0). \xi + \frac{1}{2}  D^2 \phi (x_0)\cdot (\xi,\xi) + o (\vert \xi\vert^2)
 $$
 Since $u \leq_{x_0} \phi$ with $u (x_0) = \phi (x_0)$, it follows that for $\vert \xi\vert << 1$,
 $$
 u (x) \leq u (x_0)  + D \phi (x_0). \xi + \frac{1}{2}  D^2 \phi (x_0)\cdot (\xi,\xi) + o (\vert \xi\vert^2)
 $$
 which proves that $(D \phi (x_0), D^2 \phi (x_0)) \in J^{2,+} u (x_0)$.
 
 To prove the converse it is enough to assume that $(p,Q) \in J^{2,+} u (x_0)$, since by approximation the results will follow by lower semi-continuity of $F$. This is less trivial and follows from the following elementary but non trivial lemma (see \cite{DI04}, \cite{CIL92}).\end{proof}
 \begin{lem} For any $(p,Q) \in J^{2,+} (u)$ there exists a $C^2$ function near $x_0$ such that $D \phi (x_0) = p$, $D^2 \phi (x_0) = Q$ and $ u \leq_{x_0} \phi$ i.e. $J^2 \phi (x_0) = \{(p,Q)\}$.
 \end{lem} 
 Let us come back to the following simple example to show that viscosity concepts are the right ones to ensure uniqueness of the solutions.
 \begin{exa}
 We have already observed the advantage of VSC solutions in exhibiting the solution to the equation $F (x,u,Du,D^2u) = H (u') = \vert u'\vert - 1 = 0$ on $[-1,+1]$ with the the boundary condition $u(\pm1) = 0$. The piecewise affine function on $[-1,1]$ defined by $u_0 (x) = 1 - \vert x\vert$, which satisfies $\vert u_0' (x)\vert = 1$ except at the origin where it is not differentiable, it is a generalized solution. Observe that the equation has infinitely many piecewise affine generalized solutions with the prescribed boundary condition.
  However it is not difficult to see that among these generalized solutions,  $u_0$ is the only one which is a viscosity solution for the equation associated to the Hamiltonian $H (x,u,u') = \vert u'\vert - 1.$  Indeed observe that the only problem is at the origin. It's easy to see that any upper test function $q$ for $u_0$ at $0$ satisfies the condition $\vert q' (0) \vert \leq 1$, while there is no lower test function for $u$ at the origin.
  
  On the other hand, it is also clear that $u_0$ is not a subsolution to the equation $1 - \vert u'\vert = 0$, since any upper test function $q$ at the origin should satisfy
  the inequality $1 \leq \vert q' (0)\vert$, while by the previous observation it has to satisfy the inequality $\vert q' (0) \vert \leq 1$ hence $\vert q' (0)\vert = 1$, which is obviously not the case. The same reasoning as above actually proves that the function $v_0 (x) = \vert x \vert - 1$ is a viscosity solution to the equation $1 - \vert u'\vert = 0$ with boundary values $0$. 
  \end{exa}
 \vskip 0.3 cm
 \subsection{The Jensen-Ishii maximum principle} 

Let us recall some classical definitions and results used in this approach (see \cite{DI04}, \cite{Car04}, \cite{CIL92}).
 As we have seen in the case of Monge-Amp\`ere equations it is necessary to assume that our Hamiltonian function
  $F : \Omega \times \R \times \R^N \times \mathcal S_N \longrightarrow \R \cup \{+\infty\}$ will be a lower semi-continuous function which is continuous on its domain $\{ F < + \infty\}$. 
 \begin{defi} Let $\f : \Omega \longrightarrow \R$ be a function defined in an open set $\Omega \subset \R^N$. The function $\f$ is said to be semi-convex on $\Omega$ if there exists a real number $k > 0$ such that the function $x \longmapsto \f (x) + \frac{k}{2} \vert x \vert^2$ is convex in  (each convex subset of) $\Omega$. In this case we also say that $\f$ a
  $k-$convex function in $\Omega$. 
  The function $\f$ is said to be $k-$concave in $\Omega$ if $- \f$ is $k-$convex in $\Omega$.
 \end{defi}
 The following notion is quite useful in the context of the viscosity approach.
 \begin{defi} We say that a function $\f : \Omega \longrightarrow \R$ is twice  differentiable at some point $x_0$ (in the sense of Alexandrov)  if there exists $p \in \R^N$ and $Q \in \mathcal S_N$ such that for $\xi \vert << 1$,
 $$
 w (x_0 + \xi) = w (x_0) + <p,\xi> + \frac{1}{2} <Q\cdot \xi,\xi> + o (\vert \xi\vert^2). \leqno (A)
 $$
 \end{defi}
 Some remarks are in order.
 \begin{rem}
 1. The condition $(A)$ means that $J^2 w (x_0) = \{(p,Q)\}$. This implies that $w$ is differentiable at $x_0$ and $D w (x_0) = p$, but in general it does not mean that $u$ is twice differentiable in the usual sense at $x_0$. Actually $w$ do not need to be differentiable in a neighbourhood of $x_0$. 
 However the quadratic form $Q$ satisfying the asymptotic expansion $(A)$ at $x_0$ is unique and given by
 $$
 Q (\xi) = \lim_{t \to 0} \,\frac{w (x_0 + t \xi) + w (x_0 - t \xi) - 2 w (x_0)}{t^2},
 $$
 for $ \xi \in \R^N$.
 We will denote the quadratic form $Q$ by $Q = D^2 w (x_0)$ and then $J^2 w (x_0) = \{(D w (x_0),D^2 w (x_0))\}$. \\ 
 2. It follows from the definitions that if $w$ is a $k$-convex function in $\Omega$ which is twice differentiable at $x_0 \in \Omega$ then  $D^2 w (x_0) \geq - k I_N$ in the sense of quadratic forms on $\R^N$.
 \end{rem}
 The following fundamental result will be useful in the proof of the generalized maximum principle.
 \begin{theo} (A.D. Alexandrov \cite{Ale39}). \label{th:A} Let $\f$ be a $k-$convex function in $\Omega$. Then there exists an exceptional Borel set $E \subset \Omega$ of  
 Lebesgue measure $0$ such that $\f$ is twice differentiable at any point $x_0 \in \Omega \setminus E$, hence $D^2 \f (x_0) \geq - k I_N$.
  \end{theo}
  Let us now state a technical lemma which is one of the main ingredients used in the proof of the comparison  
  principle. It is based on the Alexandrov fundamental theorem. 
  \begin{lem} (R. Jensen \cite{Jen88})
  Let $w$ be a semi-convex function in an open set $\Omega \subset \R^N$. Assume that the function $w $ reaches its local maximum at some point $x_0 \in \Omega$. Then there exists  a sequence $a_j \to x_0$ in $\Omega$ such that $w$ is twice differentiable at each $a_j$  and $(D w (a_j), D^2 w (a_j)) \to (0,Q^+)$ in $\R^N \times \mathcal S_N$  and $Q^+ \leq {\bf 0}$, in particular $(0,Q^+) \in {\bar J}^{2,+} w (x_0)$.
 \end{lem}
  Viscosity sub(super)-solution of our equations need not to be even continuous in general. So to be able to extract some informations from the viscosity differential inequalities they satisfy, it is necessary to approximate them by smooth functions in an appropriate way keeping memory of these differential inequalities. This can be done using sup(inf)-convolution.
  
 Let $u : \Omega \longrightarrow \R$ be a bounded upper semi-continuous function. For $\e > 0$ small enough and $ x \in \Omega_\e$, we define the sup-convolutions of $u$ as follows: 
$$
u^{\e}(x):=\sup_{y \in \Omega} \left\{ u(y) -\frac{1}{2\e^2} |y-x|^2 \right\} = 
\sup_{\{\vert y - x\vert \leq A \e \}} \left\{ u(y) -\frac{1}{2\e^2} |y-x|^2 \right\}, 
$$
where $A > 0$ is large enough so that $A^2  > 2 \text{osc}_{\Omega} u$.
 
 In the same way if $v : \Omega \longrightarrow \R$ is a bounded lower semi-continuous function. For $\e > 0$ small enough and $ x \in \Omega_\e$, we define the inf-convolutions of $u$ as follows: 
$$
v_{\e}(x):=\inf_{y \in \Omega_{\e}} \left\{ v(y) +\frac{1}{2\e^2} |y-x|^2 \right\}
= \inf_{\vert y - x\vert \leq \e } \left\{ v(y) +\frac{1}{2\e^2} |y-x|^2 \right\}.
$$
  Then it easy to show the following result (see \cite{CC95}).
 \begin{prop} \label{prop:supconv} 1. Let $u : \Omega \longrightarrow \R$ be a bounded upper semi-continuous function. Then for $0 < \e < 1$ small enough, $u_\e$ is $\e^{-2}$-convex in $\Omega_\e$ and decreases to $u$ in $\Omega$ as $\e\searrow 0$, hence it is twice differentiable at almost every point in $\Omega$.
 Moreover $u_\e$ is a subsolution of the equation $F_\e (x,w,D w,D^2 w) = 0$, where
 $$
 F_\e (x,s,p,Q) := \inf \{F (y,s,p,Q) ; \vert y - x\vert \leq A \e \}.
 $$
 2.  Let $v : \Omega \longrightarrow \R$ be a bounded lower semi-continuous function. Then for $0 < \e <1$ small enough, $v_\e$ is $\e^{-2}$-concave in $\Omega_\e$ and increases to $v$ in $\Omega$ as $\e\searrow 0$, hence twice differentiable at almost every point in $\Omega$. Moreover $v_\e$ is a supersolution of the equation $F^\e (x,w,D w,D^2 w) = 0$, where
 $$
 F^{\e} (x,s,p,Q) := \sup \{F (y,s,p,Q) ; \vert y - x\vert \leq A \e \}.
 $$
 \end{prop}
  Observe that $F_\e$ (resp. $F^\e$) is a continuous Hamiltonian in its domain which increases (resp. decreases) to $F$ in $\Omega$ as $\e$ decreases to $0$. 
  
   Using the above result, it is possible to derive a more general maximum principle for upper semi-continuous functions, called Ishii's lemma in the literature. We will refer to it as the Jensen-Ishii's maximum principle, because it is based on a powerful idea of Jensen \cite{Jen88}.
  
  \begin{theo} \label{thm:IJ-MP} Let  $u, v : \Omega \longrightarrow \R$ be two bounded functions defined in a domain $\Omega \subset \R^N$ such that  $u $ an upper semi-continuous function in $\Omega $ and $v$ is a lower semi-continuous function in $\Omega$. Let $\phi : \Omega \times \Omega \longrightarrow \R$ be a $C^2$-smooth function. Assume that the function $w (x,y) := u (x) - v(y) - \phi (x,y)$ has  a local maximum at some point $(a,b) \in \Omega \times \Omega$. Then for any $\alpha > 0$ there exists $Q^+, Q^- \in \mathcal S_N$ such that
 $ (p^+,Q^+) \in \bar{J}^{2,+} u (a)$, $(p_-,Q^-) \in \bar J^{2, -} v (b)$ such that $p^+ =  D_x \phi (a,b), p_- = - D_y \phi (a,b)$ and
 $$
   - (\frac{1}{\alpha} + \Vert A\Vert) I_{2N} \leq M (Q^+, - Q^-) \leq A + \al A^2,
 $$
 where $A = D^2 \phi (a,b)$, and $M (Q^+, - Q^-)$ is defined as a quadratic form on $\R^N \times \R^N$ as follows: if $Z = (X,Y) \in \R^N \times \R^N$, then
  $$
  < M (Q^+, - Q^-)\cdot Z , Z >  = < Q^+\cdot X , X> - <Q^- \cdot Y,  Y>.
  $$
 In particular we have $Q^+ \leq Q^-$ as quadratic forms on $\R^N$ if we choose $\phi$ so that $D^2_x \phi (x,y) = - D^2_y \phi (x,y)$.  
  \end{theo}  
 A complete proof is given in \cite{CIL92}. It uses the regularization by sup/inf convolution and the maximum principle of Jensen.  We will see in the next section how it is used to prove the comparison principle.
 \vskip 0.2 cm
 \subsection{The viscosity comparison principle in the local case} 
 
 The main tool for proving uniqueness of solutions with boundary values conditions is the so called (viscosity)  Comparison Principle which we will state now.
 \begin{defi}
  We say that the (viscosity) Comparison Principle holds for the equation $F (x,u,Du,D^2u) = 0$ if for any bounded (viscosity) subsolution $u$ in $\Omega$ and any  bounded (viscosity) supersolution $v$ a in $\Omega$ such that $u \leq v$ on $\partial \Omega$ then $u \leq v$ on $\Omega$.
 \end{defi}
 Under some additional conditions on how the function $F$ depends on the $u$ variable and its gradient , it is possible to prove the comparison principle for the equation $F = 0$ using Jensen-Ishii's Maximum principle (see \cite{CIL92}). Unfortunately there no general satisfactory statement which can be applied in our case. So we will not state any such result here and refer to \cite{IL90, CIL92} for various statements. 
 
 However we will use the same ideas in the next section and rely on Jensen-Ishii's Maximum principle to prove a comparison principle adapted to the complex Monge-Amp\`ere equations we are considering.
 
  Let us mention that the Comparison Principle implies uniqueness of the viscosity solution with prescribed boundary values. 
  Once the comparison principle is valid, it is quite easy to deduce existence of viscosity solutions using the Perron method as far as wecan show the existence of  adequate barriers (see \cite{CIL92}).
  \begin{theo} \label{thm:perron0} Assume that  the family $\mathcal U$ of bounded subsolution of the equation $F (x,u,Du,D^2u) = 0$ in $\Omega$ in non empty and locally upper bounded in $\Omega$. Then the function defined by
   $$
   U := \sup \{ u ; u  \in  \mathcal U\}
   $$
  is the maximal subsolution of the equation $F (x,u,Du,D^2u) = 0$ in $\Omega$.
  
  Moreover if  the viscosity Comparison Principle holds for the equation $F (x,u,Du,D^2u) = 0$ in $\Omega$ and there exists a subsolution $\underline u$ and a supersolution $\overline u$ such that $\underline u_* = \overline u^*$ in $\bd \Omega$, then $U$ is the unique viscosity solution de the equation   $F (x,u,Du,D^2u) = 0$ with boundary values $U = \underline u_* = \overline u^*$ in $\bd \Omega$.
   \end{theo}
  \begin{proof}
  A complete proof is given in \cite{CIL92}. Let us just give an idea of the proof. 
  The fact that $U$ is a subsolution is a standard fact: on shows that the upper semi-continuous regularization $U^*$ is actually a subsolution of the equation $F = 0$, which implies that $U^* \in \mathcal U $ and then $U^* = U$ is a subsolution. Now  the powerful idea of Ishii is to consider the lower semi-continuous regularization $\f_*$ of $\f$ and to show that it is a supersolution of the equation $F = 0$. This is done by contradiction using a bump construction argument (see \cite{Ish89}, \cite{CIL92}). We will give it in details in the complex case in the next section (see Theorem~\ref{thm:perron1}). Then by the comparison principle $ \underline u \leq U \leq \overline u$ in $\Omega$, hence $ \underline u_* \leq U_*  \leq \overline u$ in $\Omega$.  Then at the boundary we will have $\underline u_* \leq U_* $ and $U^* \leq \overline u^* = \underline u_*$, which implies that $U\leq U_*$ at the boundary $\partial \Omega$. Again by the comparison principle we can conclude that $U\leq U_*$ in $\Omega$, which finally implies that $U = U_*$ is a viscosity solution of the equation $F (x,u,Du,D^2u) = 0$ with boundary values $U = \underline u_* = \overline u^*$ in $\bd \Omega$. Uniqueness follows from the comparison principle.
 \end{proof}

\section{The viscosity approach to degenerate complex Monge-Amp\`ere equations}
 
The purpose of this section is to make the connection between the pluripotential theory for the complex Monge-Amp\`ere
operators, as founded by Bedford-Taylor \cite{BT76, BT82}, and the viscosity approach developed by P.L.~Lions and all
(see \cite{IL90,CIL92}). 
\vskip 0.3 cm
\subsection{Viscosity subsolutions in the complex case }

Let $X$ be a (connected) complex manifold of dimension $n$ and $\mu \geq 0$ a semi-positive volume form with continuous density with respect to a fixed smooth non degenerate volume form $\mu_0 > 0$ form on $X$.
In this section $B$ will denote the unit ball of $\C^n$ or its image under a coordinate chart in $X$.

We will consider the following general complex Monge-Amp\`ere type equations 
 $$
    - (dd^c u)^n + e^{g (z,u) + h(D u)} \mu = 0, \leqno (MAE).
  $$
 where $\Omega \Subset \C^n$ is bounded domain, $g$ is a continuous function on $\Om \times \R$ increasing in the $u$ variable, $h$ is continuous function in $X$ and $\mu$ is a continuous positive volume form on $X$. 
 
 To fit in with the viscosity point of view, we identify $\C^n \simeq \R^{2 n}$ and define the Hamiltonian function for $(z,s,p,Q) \in \Omega \times \R \times \C^n \times \mathcal S_{2 n}$ by the formula
 $$
 F (z,s,p,Q) = 
 \left\{\begin{array}{ll}
- (dd^c Q)^n + e^{g (z,s) + h(p)} \mu (z) & \text{ if } Q \ge 0 \\
+ \infty & \text{otherwise}.
\end{array}\right\},
 $$
 where $dd^c Q$ is the hermitian $(1,1)-$part of the (real) quadratic form $Q$ on $\C^n \simeq \R^{2 n}$. Here we identify a hermitian form with the real $(1,1)$-form associated to it.
 
 With this notation, we see that our Hamiltonian is lower semi-continuous in $(z,s,p,Q) \in \Omega \times \R \times \C^n \times \mathcal S_{2 n}$, continuous in its domain and satisfy the degenerate ellipticity condition as well as the properness condition stated in the previous section. Therefore we can use the notions of subsolutions  as in the previous section.
 
 However since the function is not upper semi-continuous in $\Omega \times \R \times \C^n \times \mathcal S_{2 n}$, the definition of supersolutions using this Hamiltonian function will not be dual to the previous one i.e. when $v$ is a supersolution to $F = 0$, it is not clear whether $- v$ is a subsolution.

  Let us denote by $(dd^c Q)_+ = dd^c Q $ if $dd^c Q \geq 0$ and $(dd^c Q)_+ = 0$ if not.
 Then observe that for a lower test function $q$ for $\f$ at $x_0$ i.e. $\f \geq_{x_0} q$, the condition $(dd^c q(x_0))^n_+ \leq \mu(x_0)$ is always satisfied when $dd^c q(x_0)$ is not semi-positive as well as the condition $F (x_0,\f (x_0), D q(x_0),dd^c q (x_0)) \geq 0$. Hence the condition $F (x,s,Q) \geq 0$ is consistent only when $dd^c Q \geq 0$. 
 
 Therefore to define supersolutions, it is natural to introduce the following Hamiltonian function
 $$
  F_+ (z,s,p,Q) := 
- (dd^c Q)_+^n + e^{g (z,s) + h(p)} \mu (z).
 $$
 and then the general definition of a supersolution can be formulated in the following equivalent way:
\begin{defi} \label{def:super}
A supersolution of $(dd^c \f)^n= e^{\e \f} \mu$ is a lower semicontinuous  function 
$\p: \Omega \to \R \cup \{+\infty\}$ such that $\p \not \equiv +\infty$
and the following property is satisfied:
if for any $z_0\in \Omega$ and any $q \in {\mathcal C}^2 (z_0)$, defined in a neighborhood of $z_0$ such that $\p(z_0)=q(z_0)$ and
$
 \p -q \  \text{ has a local minimum at} \ z_0,
 $
 then 
 $$
 (dd^c q(z_0))^n_+ \leq  e^{g (z_0,\p (z_0)) + h(D q (z_0))} \mu (z_0).
 $$
\end{defi}
Observe that the Hamiltonian function $F_+$ is upper semi-continuous everywhere, degenerate elliptic and proper and coincide with $F$ in its domain. Moreover, as we will see later, it turns out that the notion of subsolution is the same for the two Hamiltonians (see Proposition 3.2). 

Observe that the viscosity differential inequality (given by the supersolution property) for a lower test function $q$ at a point $z_0$ does not tell anything about the sign of  $dd^c q (x_0)$ and is certainly satisfied whenever  $dd^c q (x_0)$ is not semi-positive.  
 However the condition is natural since when $v$ is a smooth function which is a supersolution in the classica sense, then for any lower test function $q$ at a given point  $x_0$ we have $dd^c v (x_0) \geq dd^c q (x_0)$ by the classical maximum principle. No if we know that  $ dd^c q (x_0) \geq 0$ we can obtain a consistent estimate that is
 $(dd^q)^n_{x_0} \leq (dd^c v)^n_{x_0} \leq e^{g (z_0,v(z_0)) + h(D q (z_0))} \mu (z_0)$.  But if do not know the signe of  $ dd^c q (x_0)$ we cannot conclude. This means in particular that any smooth function $v$ such that its complex hessian admits at least one negative eigenvalue at any point is a supersolution. In particular any plurisuperhamonic function in $\Omega$ is a supersolution to the above equation.
 
 Actually the only way we will use this definition in the sequel is as follows. If $\f$ is not a supersolution of the equation then there exists a point $z_0$ and a lower test function $q$ at $z_0$ such that $(dd^c q(z_0))^n_+ > \mu(z_0) \geq 0.$ Therefore  $dd^c q (z_0) \geq  0$ and $(dd^c q(z_0))^n > 0$ which implies that $dd^c q(z_0) > 0$. 
 
 Note that if $\mu\ge \mu'$ then a subsolution for $ \mu$ is a subsolution for $\mu'$. This holds
in particular  if $\mu'=0$. 

\subsubsection{Subsolutions of the equation $(dd^c u)^n = \mu$}
Here we restrict ourselves to the special case where $g \equiv 0$ and $h \equiv 0$ and first observe that a function $\f$ satisfies $(dd^c \f)^n \geq 0$ in the viscosity sense if and only if it is plurisubharmonic in $X$.

\begin{prop}  \label{prop:vsc=plp}
The viscosity subsolutions of the complex Monge-Amp\`ere equation $(dd^c\f)^n=0$ are precisely the plurisubharmonic functions on $X$. 
\end{prop}
\begin{proof} 
Let $\f$ be a subsolution of $(dd^c\f)^n=0$.  Let $x_0\in X$ such that $\f(x_0)\not= -\infty$. 
The problem is local so we can assume that $X$ is a domain in $\C^n$.
Let $q\in {\mathcal C}^2(V_{x_0})$ such that $\f-q$ has a local maximum at $x_0$. Then the hermitian matrix $Q=dd^c q_{x_0}$
satisfies  $\det(Q)\ge 0$. Moreover for every hermitian semipositive matrix $H$, we also have $\det(Q+H)\ge 0$ since,
a fortiori for  
$q_H=q+H(x-x_0)$, $\f-q_H$ has a local maximum at $x_0$ too.

It follows from Lemma \ref{lem:semipositive} below that $Q=dd^cq_{x_0}$ is actually semi-positive.
We infer that for every positive definite hermitian matrix $(h^{i \bar j})$
$\Delta_H q (x_0) := h^{i\bar j} \frac{\partial^2 q}{\partial z_i \partial \bar z_j} (x_0) \ge 0$, i.e. $\f$ is a viscosity subsolution 
of the equation $ - \Delta_H \f =0$. In appropriate complex coordinates 
this constant coefficient differential operator is nothing but the Laplace operator.  Hence (\cite{Hor94} Proposition 3.2.10' p. 147) applies to the effect that
$\f$ is $\Delta_H$-subharmonic hence is in $L^1_{loc}(V_{x_0})$ and satisfies 
$\Delta_H  \f \ge 0$ in the sense of distributions. Let $(w^i)$ be any vector in $\C^n$. Consider a 
positive hermitian matrix $(h^{i\bar j})$ degenerating to the rank one matrix $(w^i\bar w^j)$. By continuity, we have $\sum w^i\bar w^j \frac{\partial^2 \f}{\partial z_i \partial \bar z_j}\ge 0$ in the sense of distributions.  Thus $\f$ is plurisubharmonic. 

\smallskip

Conversely, assume $\f$ is plurisubharmonic. Fix $x_0\in X$, $q\in {\mathcal C}^2({V_{x_0}})$ such that
 $
  \f-q \  \text{ has a local maximum at} \ x_0 .
  $
 Then, for every small enough ball $B\subset V_{x_0}$ centered at $x_0$, we have 
 $$
 \f(x_0)-q(x_0) \ge \frac{1}{V(B)} \int_{B} (\f-q) \, dV,
 $$ 
 hence
 $$
 \frac{1}{V(B)} \int_{B} q \, dV -q(x_0) \ge \frac{1}{V(B)} \int_{B} \f  \, dV-\f(x_0) \ge 0. 
 $$
 Letting the radius of $B$ tend to $0$ it follows, since $q$ is ${\mathcal C}^2$ that $\Delta q_{x_0} \ge 0$.
  Using complex ellipsoids instead of balls\footnote{This amounts to a linear change of complex coordinates.},  we conclude that $\Delta_H q (x_0)  \ge 0$ for every positive definite hermitian matrix.
Thus $dd^c q _{x_0} \ge 0$ and  $(dd^c\f)^n \ge 0$ in the viscosity sense. 
\end{proof}

The following lemma is easily proven by diagonalizing $Q$:

\begin{lem} \label{lem:semipositive}
Let $Q$ be an hermitian matrix such that, for every semipositive hermitian matrix $H$, $\det(Q+H)\ge 0$ then $Q$ is semipositive.
\end{lem}

Recall that when $\f$ is plurisubharmonic and locally bounded, its Monge-Amp\`ere measure
$MA(\f)=(dd^c \f)^n$ is well defined \cite{BT76} (as the unique limit of the smooth measures $MA(\f_j)$,
where $\f_j$ is any sequence of smooth psh functions decreasing to $\f$).
Our next result makes the basic connection between this pluripotential notion and its viscosity counterpart.

\begin{prop} \label{pro:visc=pluripot}
Let $\f$ be a locally bounded upper semi-continuous function in $X$.
It satisfies $(dd^c\f)^n \ge \mu$ in the viscosity sense iff 
it is plurisubharmonic and its Monge-Amp\`ere measure satisfies 
$MA(\f) \ge \mu$ in the pluripotential sense.
\end{prop}
\begin{proof}
Assume $\f \in PSH\cap L^{\infty} (B)$ satisfies $MA(\f) \ge \mu$. Consider $q$ a ${\mathcal C}^2$ function such that
$\f-q$ achieves a local maximum at $x_0$ and $\f(x_0)=q(x_0)$. Since $\f$ satisfies $(dd^c\f)^n\ge 0$ in the 
viscosity sense,  $(dd^cq)_{x_0}^n \ge 0$ and 
$dd^cq_{x_0} \ge 0$ by lemma 	\ref{lem:semipositive}. Assume $(dd^cq)_{x_0}^n < \mu_{x_0}$. 
Let $q^{\e}:=q+ \e\| x-x_0 \|^2 $. Choosing $\e>0$ small enough, we have  $0<(dd^cq^{\e}_{x_0})^n < \mu_{x_0}$. 
Since $\mu$ has continuous density, we can chose a small ball $B'$ containing $x_0$ of radius $r>0$ such that
$\bar q^{\e}=q^{\e}-\e \frac{r^2}{2} \ge \f$ near $\partial B'$ and $MA(\bar q^{\e}) \le MA(\f)$. 
The comparison principle (Theorem~\ref{thm:CP})
yields $\bar q^{\e} \ge \f$ on $B'$. But this fails at $x_0$. Hence $(dd^c q )^n_{x_0} \ge \mu_{x_0}$ and $\f$ 
is a viscosity subsolution. 

Conversely assume $\f$ is a viscosity subsolution.  Fix $x_0\in M$ such that $\f(x_0)\not= -\infty$
and $q\in {\mathcal C}^2$ such that $\f-q$ has a local maximum at $x_0$. Then the hermitian matrix $Q=dd^cq_{x_0}$
satisfies      $\det(Q)\ge \mu_{x_0}$.

Recall that the classical trick (due to Krylov) of considering the complex Monge-Amp\`ere equation as a Bellmann equation 
relies on the following:
\begin{lem} \cite{Gav77}
Let $Q$ be a $n\times n $ non negative hermitian matrix, then
$$
 \det(Q)^{1/n}= \inf \{ \rm{tr}(H Q)  \, | \, H \in H_n^+ \text{ and } \det(H)=n^{-n} \},
$$ 
where $H_n^+$ denotes the set of positive hermitian $n \times n$ matrices.
\end{lem}

Applying this to our situation, it follows that for every positive definite hermitian matrix 
$H = (h_{i\bar j})$ with $\det(H)=n^{-n}$, 
$$
\Delta_H q (x_0) := \sum h_{i\bar j} \frac{\partial^2 q}{\partial z_i \partial \bar z_j} (x_0) \ge \mu^{1/n}(x_0),
$$
 i.e. $\f$ is a viscosity subsolution of the linear equation $\Delta_H \f = \mu^{1/n}$.
 
 This is a constant coefficient linear partial differential
equation.  Assume $\mu^{1/n}$ is $C^{\alpha}$ with $\alpha>0$ and choose a ${\mathcal C}^2$ solution of 
$\Delta_H \f = \mu^{1/n}$ in a neighborhood of $x_0$ (see \cite{GT83}). Then $u=\f-f$ satisfies $\Delta_H u\ge 0$ in the viscosity sense. 
Once again, (\cite{Hor94} prop 3.2.10' p. 147) applies to the effect that
$u$ is $\Delta_H$-subharmonic hence  
$\Delta_H \f \ge \mu^{1/n}$ in the weak sense of positive Radon measures. 

Using convolution to regularize $\f$ and setting $\f_{\e}=\f * \rho_{\e}$ we see that 
 $\Delta_H \f_{\e}\ge (\mu^{1/n})_{\e}$. 
 Another application of the above lemma yields 
 $$
 (dd^c\f_{\e})^n \ge ( (\mu^{1/n})_{\e})^n.
 $$ 
 Since $\f_\e$ is decreasing with $\e$,  continuity of $MA(\f)$ with respect to such a sequence  yields $MA(\f) \ge \mu$ by Theorem~\ref{thm:CV}. 
 
 This settles the case when $\mu>0$ and $\mu$ is H\"older continuous. In case $\mu>0$ is merely continuous
 we observe that $\mu=\sup\{\nu | \nu \in {\mathcal C}^{\infty}, \ \mu\ge \nu >0 \}$. 
  Taking into account the fact that
 any subsolution of $(dd^c\f)^n=\mu$ is a subsolution of $(dd^c\f)^n =\nu$ provided $\mu\ge \nu$ we conclude $MA(\f) \ge \mu$. 
 
 In the general case when $\mu \geq 0$, we observe that $\psi_{\e}(z)=\f(z)+\e \| z \|^2$ satisfies $(dd^c\psi_{\e})^n \ge \mu+ \e^n \lambda$ in the viscosity sense with $\lambda$ the euclidean volume form. Hence $MA(\psi_{\e})\ge \mu$, 
 from which we conclude that $MA(\f) \ge \mu$. 
 \end{proof}
 
 \begin{rem}
  The proof actually works in any class of plurisubharmonic functions in which the Monge-Amp\`ere
 operator is continuous by decreasing limits of locally bounded functions and the
 comparison principle holds. When $n \geq 2$, these are  precisely the finite energy classes 
 studied in \cite{Ceg98,GZ07,BGZ09}.\\
 
 The basic idea of the proof is closely related to the method in \cite{BT76} and is the topic treated in \cite{Wi04}. An alternative proof by using sup-convolutions will be given in the next section.
 \end{rem}
 
 We now relax the assumption that $\f$ being bounded and connect viscosity subsolutions to
 pluripotential subsolutions through the following:
 
 \begin{theo} \label{thm:visc=pluripot}
 Assume that there exists a bounded psh function $\rho$ on $ X$ such that $(dd^ c \rho)^n \geq \mu$ in the weak sense in $X$.
 Let $\f$ be an upper semicontinuous function such that $\f \not \equiv -\infty$ on any connected component. The
 following are equivalent:
 
 $(i)$ $\f$ satisfies $(dd^c\f)^n\ge \mu$ in the viscosity sense on $X$;
 
 $(ii)$ $\f$ is plurisubharmonic and for all $c>0$,
 $(dd^ c\sup[\f, \rho-c])^ n \ge \mu$ in the pluripotential sense on $X$. 
 \end{theo}
 
 Observe that these properties are local and that it is possible to find a local strictly psh function such that  locally $(dd^ c \rho)^n \geq \mu$.

 \begin{proof}   
 Assume first that  $\f$ is a viscosity subsolution of $- (dd^c \rho)^n + \mu = 0$.
 Since $\rho-c$ is also a subsolution, it follows from the maximum principle as in the proof of Lemma \ref{lem:UB} that $\sup(\f, \rho-c)$ is a pluripotential subsolution,
 hence Proposition \ref{pro:visc=pluripot} yields $MA(\sup(\f, \rho-c)) \ge \mu$ in the viscosity sense.    

Conversely, fix $x_0\in X$ and assume i) holds. If $\f$ is locally bounded near $x_0$, Proposition \ref{pro:visc=pluripot}
implies that $\f$ is a pluripotential subsolution near $x_0$. 

Assume $\f(x_0)\not= -\infty $ but $\f$ is not locally bounded near $x_0$. Fix 
$q\in {\mathcal C}^2$ such that
$q\ge \f$ near $x_0$ and $q(x_0)=\f(x_0)$. Then for $c>0$ big enough we have $q\ge \f_c=\sup(\f, \rho-c)$ and 
$q(x_0)=\f_c(x_0)$, hence $(dd^cq) _{x_0}^n \ge \mu_{x_0}$ by Proposition \ref{pro:visc=pluripot} again.

Finally if $\f(x_0)=-\infty$ there are no $q$ to be tested against the differential inequality, hence it holds for every test function $q$.
  \end{proof}

 Condition $(ii)$ might seem a bit cumbersome. The point is that
 the Monge-Amp\`ere operator can not be defined on the whole space of plurisubharmonic functions. 
When $\f$ belongs to its
 domain of definition, condition $(ii)$ is equivalent to $MA(\f) \geq \mu$ in the pluripotential sense.
To be more precise, we have:

\begin{coro}
 Let $\Omega\Subset \C^n$ be a hyperconvex domain. Then $\f\in \mathcal{E}(\Omega)$, see \cite{Ceg04} for the notation,
satisfies $(dd^c\f)^n\ge \mu$ in the viscosity sense iff its Monge-Amp\`ere measure $MA(\f)$ satisfies  $MA(\f)\ge \mu$. 
\end{coro}

We do not want to recall the definition of the class ${\mathcal E}(\Omega)$ (see \cite{Ceg04}). It suffices to say that $\mathcal{E}(\Omega)$ coincides with the domain of definition of the complex Monge-Amp\`ere operator (see \cite{Bl06}) and when
$n=2$,  $\mathcal{E}(\Omega) = PSH (\Omega) \cap W^{1,2}_{loc} (\Omega)$ \cite{Blo04}.
 \vskip 0.2 cm
 \subsubsection{Viscosity subsolutions for $(dd^c u)^n = e^{\e \f} \mu$}
 

We shall first consider the complex Monge-Amp\`ere equations
$$
 - (dd^c \f)^n + e^{\e \f} \mu = 0,
$$
where $\mu$ is continuous volume form on $X$. Viscosity techniques actually
mainly apply to the case $\e>0$ and we are going to treat the previous
case $\e=0$ by a limiting process.

\smallskip

When $\f$ is continuous, so is the density of ${\tilde \mu}=e^{\e \f} \mu$:
these definitions are then equivalent to the above ones and the first
basic properties can be applied. When $\f$ is not assumed to be continuous,
one needs to carefully check that subsolutions (resp. supersolutions)
can still be understood equivalently in the pluripotential or viscosity sense.

\begin{prop}
Let $\f: X \rightarrow \R$ be a locally  bounded u.s.c. function. 
Then $\f$ satisfies $(dd^c \f)^n \geq e^{\e \f} \mu$ in the viscosity sense in $X$
if and only if it is plurisubharmonic and it  does  in the pluripotential sense in $X$.
\end{prop}

\begin{proof}
 We can assume without loss of generality that $\e = 1$ and $X = \Omega$ is a domain in $\C^n$. 
 When $\f$ is continuous, so is the density $\tilde \mu := e^ \f v$ and Proposition\ref{pro:visc=pluripot} above implies that 
$\f$ is a viscosity subsolution of the equation $(dd^c \f)^n = e^{\f} \mu$ iff it is a pluripotential subsolution of the same equation.

The general case can be handled by approximation. First assume that $\f$ is a viscosity subsolution and set $\mu = f \beta_n,$ where $f > 0$ is the continuous density of the volume form $\mu$ w.r.t. the euclidean volume form on $\C^n$.
We approximate $\f$ by sup-convolutions defined for $\de > 0$ small enough, by
$$
\f^{\de}(x):=\sup_y \left\{ \f(y)-\frac{1}{2\de^2} |x-y|^2 \right\}, \ \ x \in \Omega.
$$
Observe that if 
$A > 1$ is a large constant so that $A^2 > 2 \text{osc}_{\Omega} \f$, then
\begin{equation}
\f^{\de}(x) = \sup_{\vert y \vert \leq A \de} \left\{ \f(x-y)-\frac{1}{2\de^2} |y|^2 \right\},
\end{equation}
for $\de > 0$ small enough, $ x \in \Omega_{\de}$,
 where $\Omega_{\de} := \{ x \in \Omega ; dist (x,\partial \Omega) > A \de\}$.
 
Thus  $(\f^{\de})$ is a family of psh (and semi-convex) functions on $\Omega_{\de}$, that decrease towards $\f$
as $\de$ decreases to zero. Furthermore, by Proposition~\ref{prop:supconv},   $\f^{\de}$ satisfies the following inequality in the sense of viscosity on $\Omega_{\de}$
$$
(dd^c \f^{\de})^n \geq e^{\f^{\de}} f_{\de} \beta_n,
\text{ with } f_{\de}(x)=\inf \{f(y) \, / |y-x| \leq A \de \}.
$$
Since $\f^{\de}$ is psh and {\it continuous}, we can invoke Proposition \ref{pro:visc=pluripot}
and get that
$$
(dd^c \f^{\de})^n \geq e^{\f^{\de}} f_{\de} \, \beta_n
\geq e^{\f} f_{\de} \, \beta_n,
$$ 
holds in the pluripotential sense.
Since $f_{\de}$ increases towards $f$ and the complex Monge-Amp\`ere operator is continuous along decreasing sequences of bounded
psh functions (see Theorem~\ref{thm:CV}),  we finally
obtain the inequality $(dd^c \f)^n \geq e^{\f} \mu$ in the pluripotential sense.

We now treat the other implication. 
Let $\f$ be a psh function satisfying the inequality
$$
(dd^c \f)^n \geq e^{\f} \mu,
$$
in the pluripotential sense on $\Omega$.
We want to prove that $\f$ satisfies the above differential inequality in the sense of viscosity on $\Omega$. If $\f$ were continuous then we could use \ref{pro:visc=pluripot}. 
But since $\f$ is not necessarily continuous we first approximate $\f$ using sup-convolution $\f^{\de}$ as above.
Lemma \ref{convol} below yields the following :
\begin{equation}
\label{eq:conv-ineq}
(dd^c \f^{\de})^n \geq e^{\f^{\de}} f_{\de} \beta_n
\end{equation}
in the sense of pluripotential theory  in $\Omega_{\de}$.

Since $\f^{\de}$ is continuous we can apply Proposition~\ref{pro:visc=pluripot}  to  conclude that $\f^{\de}$ is a viscosity subsolution of the equation 
 $(dd^c u)^n = e^{ u} f_{\de} \beta_n$ on $\Omega_\de$.
 
 From this we want to deduce that $\f$ is a viscosity subsolution of the equation $(dd^c \f)^n = e^{\f} f \beta_n$ by passing to the limit as $\de$ decreases to $0$.
This is certainly a well know fact in viscosity theory, but let us give a proof here for convenience.

 Let $x_0 \in \Omega$, $q$ be a quadratic polynomial such that $\f (x_0) = q (x_0)$ and $\f \leq q$ on a neighbourhood of $x_0$ say on a ball $2 B$, where $B := B (x_0,r) \Subset \Omega$.
Since $\f$ is psh on $\Omega$, it satisfies $(dd^c \f)^n \geq 0$ in the viscosity sense on $\Omega$ by Proposition~\ref{prop:vsc=plp} and then by lemma~\ref{lem:semipositive}, it follows that $dd^c q (x_0) \geq 0$.  Replacing $q$ by $q (x) + \e \vert x- x_0\vert^2$ and taking $r > 0$ small enough, we can assume that $q$ is psh on the ball $2 B$.
We want to prove that $(dd^c q (x_0))^n \geq e^{\f (x_0)} f (x_0) \beta_n$.

Fix $\e > 0$ small enough.  For $x \in B$, set
 $$
  q_{\e} (x) := q (x) +  2 \e(\vert x - x_0\vert^2 - r^2) + \e r^2.
$$
Observe first that since $\f \leq q$ on $2 B$, we have   the following properties :\\
- if $x \in \partial B,$  $\f^{\de} (x) - q_\e (x) = \f^{\de} (x) - q (x) - \e r^2 < 0$ on $B$, for $\de > $ small enough. \\
- If $x = x_0$, we have $\f_{\de} (x_0) -  q^{\e} (x_0) = \f_{\de} (x_0) - q (x_0) + \e r^2 $. \\
Since $\f_{\de} (x_0) - q (x_0) \to \f (x_0) - q (x_0) + \e r^2 = \e r^2$ as $\de \to 0$, it follows that for $\de$ small enough, the function
$\f^{\de} (x) - q_{\e} (x)$ takes it maximum on $\bar B$ at some interior point $x_{\de} \in B$ and this maximum satisfies the inequality
\begin{equation}
\lim_{\de \to 0} \max_{\bar B} (\f_{\de} - q^{\e}) = \lim_{\de \to 0} (\f_{\de} (x_{\de}) -  q^{\e} (x_{\de})) \geq  \e r^2.
\label{eq:max-minoration}
\end{equation}
Moreover we claim that  $x_{\de} \to x_0$ as $\de \to 0$. Indeed we have
\begin{eqnarray*}
\f^{\de} (x_{\de}) -  q_{\e} (x_{\de}) &= & \f^{\de} (x_{\de}) - q (x_{\de}) - 2 \e (\vert x_{\de} - x_0\vert^2 - r^2) - \e r^2 \\
& = &  q^{\de} (x_{\de}) - q (x_{\de})  - 2 \e \vert x_{\de} - x_0\vert^2 + \e r^2.
\end{eqnarray*}
 Since $q^{\de} (x_{\de}) - q (x_{\de})$ converges to $0$, it follows that if $x'_0$ is a limit point of the family $(x_{\de})$ in $\bar B$, then $\max_{\bar B} (\f_{\de} - q^{\e})$ will converge to a limit  which is less or equal to $- 2 \e \vert x'_{0} - x_0\vert^2 + \e r^2$. By the inequality (\ref{eq:max-minoration}), this limit is $\geq \e r^2$. Therefore we obtain the inequality $- 2 \e \vert x'_{0} - x_0\vert^2  \geq 0$ which implies that $x'_0 = x_0$ and our claim is proved.
 
Since $\f^{\de} - q_{\e}$ takes it maximum on $\bar B$ at the point $x_{\de} \in B$ and $\f^{\de}$ is a viscosity subsolution of the equation $ (dd^c u) \geq e^u f_{\de} \beta_n$, it follows that
$$
 (dd^c q_{\e} (x_{\de}))^n \geq  e^{\f^{\de} (x_{\de})}   f_{\de} (x_{\de})  \beta_n =  e^{\f^{\de} (x_{\de}) - q_{\e} (x_{\de})} e^{q_{\e} (x_{\de})}  f_{\de} (x_{\de}) \beta_n.
$$
Now observe that  
$$\f^{\de} - q_{\e} = (\f^{\de} - q) + (q - q_{\e}) $$ 
and by Dini's lemma 
$$
\limsup_{\de \to 0} \max_{\bar B} (\f^{\de} - q) = \max_{\bar B} (\f - q) = 0.
$$
 Therefore
 $$
 \limsup_{\de \to 0} (\f^{\de} (x_{\de}) - q_{\e} (x_{\de})) \geq \liminf_{\de \to 0} \min_{\bar B}(q - q_{\e}) = \min_{\bar B}(- 2 \e \vert x- x_0\vert^2 + \e r^2) = - \e r^2.
 $$
It follows immediately that
$$
(dd^c  q_{\e} (x_{0}))^n \geq  e^{q (x_0) - 2 \e r^2} f (x_0) \beta_n.
$$
In the same way, we obtain  the required inequality 
$(dd^c q (x_0))^n \geq e^{\f (x_0)} f (x_0) \beta_n$, since $q (x_0) = \f (x_0).
$
\end{proof}

 \begin{lem} \label{convol}
 Let $\f$ be a bounded plurisubharmonic function in a domain $\Omega \Subset \C^n$ such that
$$
(dd^c \f)^n \geq  e^{\f} f \beta_n,
$$ 
in the pluripotentiel sense in $\Omega$,
where $f \geq 0$ is a continuous density. Then the sup-convolutions $(\f^{\de})$  satisfy 
$$
(dd^c \f^{\de})^n \geq  e^{\f^{\de}} f_{\de} \beta_n,
$$
 in the pluripotentiel sense in $\Omega_{\de}$, where $f_{\de} (x) := \inf \{ f (y) ; \vert y - x\vert \leq A \de\}$.
\end{lem}

\begin{proof}
Fix $\de > 0$ small enough. For $y \in B (0,A \de)$, denote by $ \psi_y (x) :=   \f(x-y)-\frac{1}{2\de^2} |y|^2,$ $x \in \Omega_{\de}$ and observe that $\psi_y$ is a bounded psh function on $\Omega_{\de}$ which satisfies the following inequality in the pluripotential sense on  $\Omega_{\de}$ 
$$
 (dd^c \psi_y)^n \geq e^{\psi_y} f_{\de} \beta_n,
$$
thanks to the invariance of the complex Monge-Amp\` ere operator by translation.

Since $\f$ is the upper envelope of the family $\{\psi_y ; y \in B (0,A \de)\}$, it follows from a well known topological lemma of Choquet that there is a sequence 
of points $(y_j)_{j \in \N}$ in the ball $ B (0,A \de)$ such that
$\f^{\de} = (\sup_{j} \psi_{y_j})^*$ on $\Omega_{\de}$. For $j \in \N$, denote by $\theta_j := \sup_{0 \leq k \leq j} \psi_{y_k}$. 
Then $(\theta_j)$ is an increasing sequence of bounded psh functions on $\Omega_{\de}$ which converges a.e. to $\f^{\de}$ on $\Omega_{\de}$.
 Itfollows from the maximum principle Theorem~\ref{thm:CP} as in the proof of Lemma~\ref{lem:UB}that $\theta_j$ is also a pluripotential subsolution of the same equation i.e. 
\begin{equation} \label{eq:SUBSOL}
  (dd^c \theta_j)^n \geq e^{\theta_j} f_{\de} \beta_n, 
\end{equation}
in the pluripotential sense in $\Omega_\de$. 

 Now by  continuity of the complex Monge-Amp\`ere operator along increasing sequences of bounded psh functions and the fact that $\sup_j \theta_j = \f^{\de}$ quasi everywhere (see \cite{BT82}), it follows from (\ref{eq:SUBSOL}) that $ (dd^c \f^{\de})^n \geq e^{\f^{\de}} f_{\de} \beta_n$ in the pluripotential sense on $\Omega_\de$.  
 \end{proof}

\vskip 0.3 cm
\subsection{Viscosity supersolutions}
The definition of supersolutions is more delicate. In the sequel, we use two references on viscosity solutions \cite{CIL92} and \cite{IL90} since both articles contain some technical points not made in the other one. 
 The outline of the real theory given in \cite{IL90}, sect. V.3, although it suggests
 a natural definition for supersolutions in the complex case, seems to rely heavily on 
 the continuity of convex functions. 
 Hence, we will introduce a different notion,
in the spirit of Definition \ref{def:super}.

We will first consider the complex Monge-Amp\`ere equation
$$
 (dd^c \f)^n =  \mu,
$$
 where $\mu \geq 0$ is a continuous volume form on some open set  $\Omega \subset \C^n$ and $\e \geq 0$. 
 
 As we have already seen, to fit in the viscosity formalism presented in previous sections, we define the Hamiltonian function as
$$
 F (x,s,Q) := - (dd^c Q)^n  + \mu (x) = 0 \text{ if} \, \, dd^c Q \geq 0,
 $$
 and $ F (x,s,Q) = + \infty $ if not,
where $dd^c Q$ is the $(1,1)-$form associated to the hermitian $(1,1)-$part of the (real) quadratic form $Q$ on $\C^n$.
 Observe that this Hamiltonian function is lower semi-continous in $\Omega \times \R \times \mathcal S_{2 n}$ and continuous in its domain $\{ F < + \infty\}$.
 
 Let us denote by $(dd^c Q)_+ = dd^c Q $ if $dd^c Q \geq 0$ and $(dd^c Q)_+ = 0$ if not.
 Then observe that for a lower test function $q$ for $\f$ at $x_0$ i.e. $\f \geq_{x_0} q$, the condition $(dd^c q(x_0))^n_+ \leq \mu(x_0)$ is always satisfied when $dd^c q(x_0)$ is not semi-positive as well as the condition $F (x_0,\f (x_0),dd^q (x_0)) \geq 0$. Hence the condition $F (x,s,Q) \geq 0$ is consistent only when $dd^c Q \geq 0$. 
 
 Therefore the general definition of a supersolution can be formulated in the following equivalent way:
\begin{defi} \label{def:super}
A supersolution of $(dd^c \f)^n= e^{\e \f} \mu$ is a lower semicontinuous  function 
$\f: \Omega \to \R \cup \{+\infty\}$ such that $\f\not \equiv +\infty$
and the following property is satisfied:
if for any $x_0\in \omega$ and any $q \in {\mathcal C}^2 (x_0)$, defined in a neighborhood of $x_0$ such that $\f(x_0)=q(x_0)$ and
$
 \f-q \  \text{ has a local minimum at} \ x_0,
 $
 then 
 $$
 (dd^c q(x_0))^n_+ \leq  \mu(x_0).
 $$
\end{defi}

 As we said before, the viscosity differential inequality (given by the supersolution property) for a lower test function $q$ at $x_0$ do not tell anything about the sign of  $dd^c q (x_0)$ and is certainly satisfied whenever  $dd^c q (x_0)$ is not semi-positive.  
 However the condition is natural since when $u$ is a smooth function which is a supersolution, then for any lower test function $q$ at a given point  $x_0$ we have $dd^c u (x_0) \geq dd^c q (x_0)$ by the classical maximum principle. No if we assume that  $ dd^c q (x_0) \geq 0$ we can conclude that 
 $(dd^q)^n_{x_0} \leq (dd^c u)^n_{x_0} \leq \mu (x_0)$.  But if do not assume that  $ dd^q (x_0)$ is non negative we cannot conclude.
 
 The only way we will use this definition in the sequel is as follows. If $\f$ is not a supersolution of the equation then there exists a point $x_0$ and a lower test function $q$ at $x_0$ such that $(dd^c q(x_0))^n_+ > \mu(x_0) \geq 0.$ Therefore  $dd^c q (x_0) \geq  0$ and $(dd^c q(x_0))^n > 0$ which implies that $dd^c q(x_0) > 0$. 
 
 Supersolutions are less classical objects and are not going to live on the same footing as 
subsolutions. Whereas subsolutions are automatically plurisubharmonic, this is not necessarily
the case of supersolutions. Observe that any plurisuperharmonic function in $\Omega$ is a supersolution to the  equation $(dd^c \p)^n = \mu$.spirit

Given  a bounded function $h$, it is natural to consider its plurisubharmonic projection
$$
P(h)(x)= P_{\Omega}(h)(x) := \left(\sup \{ \p(x) \, / \, \p \text{ psh on $\Omega$ and } \p \leq h \}\right)^*,
$$
which is the greatest psh function that lies below $h$ on $\Omega$. Observe that if $h$ is upper semi-continuous on $\Omega$ there is no need of upper regularization and the upper envelope is psh and $\leq h$ in $\Omega$.


 We will see below that in the previous definition, the lower test function $q$ satisfies 
$(dd^c q)_+^n \leq \mu$ if
and only if 
$$
(dd^c P(q))^n \leq \mu.
$$

This can be deduced from the fact that if $q$ is $C^2$ in an euclidean ball $B = B (x_0,r)$ then  the Monge-Amp\`ere measure $(dd^c P(q))^n$ of its projection $P (q) = P_B (q)$ is concentrated on the set where $P(q)=q$, with
$$
(dd^c P(q))^n={\bf 1}_{\{P(q)=q\}} (dd^c q )^n.
$$
This formula can be easily derived from the (more involved) fact that $P(q)$ is
a ${\mathcal C}^{1,1}$-smooth function (see \cite{BT76}, \cite{BD09}).

 Now we can prove the following statement which gives the relationship between the two notions of supersolutions.

 \begin{prop} \label{prop:super} 
 Let  $\Omega \Subset \C^n$ be an open set. 
 
 1. Let $\psi$ be a bounded plurisubharmonic function in $\Omega$ satisfying $(dd^c \psi)^n \leq \mu$ in the
 pluripotential sense in $\Omega$. Then its lower semi-continuous regularization $\psi_*$ is a viscosity supersolution of the equation  $(dd^c  
 \f)^n=\mu$ in $\Omega$. 
 
 2. Let $\f$ be a continuous and bounded viscosity supersolution of the equation  $ - (dd^c u)^n + \mu = 0$ in $\Omega$. Then for any euclidean ball $\B \Subset \Omega$, $ \psi := P_{\B} (\f)$ is  a continuous 
 plurisubharmonic viscosity supersolution of the equation $- (dd^c u)^n + \mu = 0$ in $\B$. 
 
 3. Let $\f$ be a $C^2$-smooth and viscosity supersolution of the equation  $- (dd^c u)^n + \mu = 0$ in $\Omega$. Then for any euclidean ball $\B \Subset \Omega$, we have  $(dd^c P_{\B} (\f))^n  
 \leq  \mu$ in the pluripotential sense in $\B$.
 \end{prop}
 \begin{proof}
 1. We use the same idea as in the proof of Proposition\ref{pro:visc=pluripot}. Assume $\psi \in PSH\cap L^{\infty} (\Omega)$ satisfies $MA(\psi) \le \mu$ in the  
 pluripotential sense on $\Omega$. Consider $q$ a ${\mathcal C}^2$-smooth function  near $x_0$ such that $\psi_*(x_0)=q(x_0)$ and
 $\psi_* - q$ achieves a local minimum at $x_0$. We want to prove that
 $(dd^cq (x_0))_+^n  \le \mu (x_0)$.   Assume that $(dd^c q (x_0))_+^n > \mu_{x_0}$. Then $dd^c q (x_0) \geq 0$ and $(dd^c q (x_0))^n > \mu_{x_0} >  0$ which implies that $dd^c q (x_0) > 0$.  
 Let $q^{\e}:=q - 2 \e (\| x-x_0 \|^2- r^2) - \e r^2$. Since $\mu$ has continuous density, we can choose $\e > 0$ small enough and a small ball $B  
 (x_0,r)$ containing $x_0$ of radius $r>0$ such that  $dd^c q^{\e} > 0$ in $B (x_0,r)$ and $(dd^c q^{\e})^n > \mu$ on the ball $B (x_0,r)$. Thus we  have $q^{\e} = q - \e r^2 <  \psi_* \leq \psi $ near $\partial B (x_0,r)$ while $MA(q^{\e}) \ge \mu \geq MA(\psi)$ in the pluripotential sense on  
  $B (x_0,r)$. 
 The comparison principle Theorem~\ref{thm:CP}
 yields $q^{\e} \le \psi$ on $B (x_0,r)$ hence $q^\e (x_0) = \liminf_{x \to x_0} q^\e (x) \leq \liminf_{x \to x_0} \psi (x) = \psi_* (x_0)$ i.e. $q
 (x_0) + \e r^2 \leq \psi_* (x_0) = q (x_0),$ which is a contradiction.  Hence $(dd^c q)^n_{x_0} \leq \mu_{x_0}$ and $\psi_*$ is a viscosity supersolution. 
 
2. Set $\psi : = P (\f)$. Then $\p$ is a continuous psh function by \cite{Wal68}.  Fix a point $x_0 \in \Omega$ and consider a super test function $q$ for $\psi$ at $x_0$ i.e. $q$ is a $C^2$ function on a small ball $B (x_0,r) \subset \Omega$ such that  $\psi (x_0) = q (x_0)$ and $\psi - q$ attains its  minimum  at $x_0$.
We want to prove that $(dd^c q (x_0))_+^n \leq \mu (x_0)$. Since $\psi \leq \f$, there are two cases:

- if $\psi (x_0) = \f (x_0)$ then $q$ is also a super test function for $\f$ at $x_0$ and then  $(dd^c q (x_0))^n_+ \leq \mu (x_0)$ since $\f$ is a supersolution of the same equation, 

- if $\psi (x_0) < \f (x_0)$, by continuity of $\f$ there exists a ball $B (x_0,s)$ $0 < s < r$ such that $\psi = P (\f) < \f$ on the ball $B (x_0,s)$ and then
$ (dd^c \psi)^n = 0$ on $B (x_0,s)$  since $(dd^c P(\f))^n$ is supported on the contact set $\{P (\f) = \f\}$. Therefore $\psi$ is a continuous psh function satisfying the inequality  $(dd^c \psi)^n = 0 \leq \mu $ in the sense of pluripotential theory on the ball $B (x_0,s)$. 
Assume that $(dd^c q (x_0))_+^n > \mu (x_0)$. Then by definition, $dd^c q (x_0) > 0$ and $(dd^c q (x_0))^n > \mu (x_0)$. Taking $s > 0$ small enough and $\e > 0$ small enough we can assume that $q^\e := q - \e (\vert x - x_0\vert^2 - s^2) $ is psh on $B (x_0,s)$ and $(dd^c q^\e)^n > \mu \geq (dd^c \psi)^n$ on the ball $B (x_0,s)$ while $q^\e = q \leq \psi$ on $\partial B (x_0,s)$. By the pluripotential comparison principle for the complex Monge-Amp\`ere operator, it follows that $q^\e \leq \psi$ on $B (x_0,s)$, thus $q (x_0) + \e s^2 \leq \psi (x_0)$, which is a contradiction.

 3. This follows from the observation made before using the argument by Berman and Demailly (\cite{BD09}).
\end{proof} 
 
 \vskip 0.3 cm
 \subsection{The Comparison Principle in the local case}
 
 We will consider the following more general complex Monge-Amp\`ere type equations 
  \begin{equation} \label{eq:CMAE}
    G (u_{j,\bar k}) + e^{g (z,u) + h(D u)} = 0, 
  \end{equation}
 where $\Omega \Subset \C^n$ is bounded domain, $G$ is a degenerate elliptic  continuous function on the cone $H_n^+$ of semi-positive hermitian forms  on $\C^n$, $g$ is a continuous function on $\Om \times \R$ increasing in the $u$ variable and $h$ is a continuous function on $\C^{n}$. 
 
   Then we will prove the following result.
  \begin{theo} Let $u$ be a subsolution of (\ref{eq:CMAE}) and $v$ a supersolution of (\ref{eq:CMAE}). Assume that $u \leq v$ on $\bd \Om$ then $u \leq v$ on $\Om$.
  \end{theo}
 \begin{proof}
  
The proof is an adaptation of arguments in \cite{CIL92}. The main idea is to apply the maximum principle to the usc function $u - v$. But since this functions are not smooth, we will apply Jensen-Ishii's maximum principle. Since the function $u - v$ is usc on $\overline{\Omega}$, then  its maximum  in $\overline{\Omega}$, defined as 
$$
 M := \sup_{\overline{\Omega}} (u - v).
$$ 
 is attained at some point in $\overline{\Omega}$.
 We want to prove that $M \leq 0$. Since $u \leq v$ on $\partial \Omega,$ we can assume that $S := \{x \in \overline{\Omega} ; u (x) - v (x) = M\} \subset \Omega$. 
 To apply Jensen-Ishii's maximum principle, we need to double the variable and add a penalty term to make the maximum reached asymptotically on the diagonal. Indeed for $\e > 0$, define the function 
 $$
  \p_\e (x,y) := u (x) - v (y) - \frac{1}{2 \e^2} \vert x -y\vert^2,
 $$
 which is upper semi-continuous on $\overline{\Omega} \times  \overline{\Omega}$. Then it takes its maximum on $\bar{\Om} \times \bar {\Omega}$ at some point $(x_\e,y_\e) \in \bar{\Om} \times \bar{\Omega}$ i.e.
 $$
  M_\e := \max_{(x,y) \in \overline{\Omega}^2} \p_\e (x,y) = u (x_\e) - v (y_\e) - \frac{1}{2 \e^2} \vert x_\e -y_\e\vert^2.
 $$
 It is quite easy to prove that (see \cite{CIL92})
 
 $$
 \lim_{\e \to 0^+} \frac{1}{2 \e^2} \vert x_\e -y_\e\vert^2 = 0
 $$
 and there exists a subsequence $(x_{\e_j},y_{\e_j}) \to (\bar x,\bar x) \in \overline{\Omega}^2$ such that 
 $$
 \lim_{\e \to 0} M_\e = M = u (\bar x) - v (\bar x).
 $$
 Then $\bar x \in S.$ Since $ S \subset \Omega$ from our assumption, it follows that $j >> 1$,  $(x_{\e_j},y_{\e_j}) \in \Omega^2$.  Therefore we can apply
Jensen-Ishii  maximum principle. Fix $j >> 1$ and set $p = p (\e_j) := \frac{1}{\e_j^2} (x_{\e_j} - y_{\e_j})$, there exists $Q^{\pm} \in \mathcal S_{2 n}$ such that
 $(p,Q^+) \in \bar{J}^{2,+} u (x_{\e_j}), (p,Q^-) \in \bar{J}^{2,-} v (y_{\e_j})$ and $Q^+ \leq Q^-$. It follows from the fact that $(p,Q^+) \in \bar{J}^{2,+} u (x_\e)$ and the definition of viscosity subsolution that the hermitian $(1,1)-$part $H^+$ of the quadratic form $Q^+$  is semi-positive hence so is  the hermitian $(1,1)-$part $H^-$ of $Q^-$ since $0 \leq H^+ \leq H^-$. Then by the degenerate ellipticity condition on $G$, we get $ - G (H^+) \leq - G (H^-)$. Therefore applying the viscosity inequalities we obtain
 $$
 e^{g (x_{\e_j},u (x_{\e_j})) + h (p)} \leq e^{g (y_{\e_j},v (y_{\e_j})) + h (p)},
 $$ 
 which implies that for $j >> 1$,
 $$
 g (x_{\e_j},u (x_{\e_j}))  \leq g (y_{\e_j},v (y_{\e_j})).
 $$
 Now recall that $(x_{\e_j},y_{\e_j}) \to (\bar x, \bar x) \in \Omega^2$ and $u (x_{\e_j}) - v (y_{\e_j}) \to M$.
 We can always assume that $\lim_j v (x_{\e_j}) = \ell \in \R$ exists and then $\lim_j u (y_{\e_j} = \ell + M$.
  
  Passing to the limit, we get $ g (\bar x, \ell + M) \leq g (\bar x, \ell)$, which implies that $M \leq 0$, since $g$ is increasing in the second  
 variable.

\end{proof} 
\begin{rem} \label{rem:vscdegen}
 The last result cannot be applied when $g$ does not depend on $s$ i.e; the equation do not involve the function $u$ itself. We do not know if the result is still true in this case. However if the function does not depend on $p$ and is only assumed to be non decreasing, it is possible to prove the comparison principle using instead, the so called Alexandroff-Backelman-Pucci maximum principle (see \cite{CC95}, \cite{Wang10}, \cite{Ch12}).

\end{rem}
Let us give the following application of the local comparison principle.
\begin{prop} If $\mu > 0$ is a continuous volume form on a complex manifold $X$ of dimension $n$, then viscosity solutions of the equation $(dd^c \f)^n = e^{g (x,\f)} \mu$ in $X$ are precisely the continuous psh functions $\f$ solutions of the equation $(dd^c \f)^n = e^{g (x,\f)} \mu$ in the pluripotential sense in $X$.
\end{prop}
\begin{proof} We already know by Proposition~\ref{pro:visc=pluripot} and Proposition~\ref{prop:super}  that continuous psh (pluripotential) solutions of the equation $(dd^c \f)^n = e^{g (x,\f} \mu$ on $X$ are viscosity solutions of the equation.  To prove the converse,
assume that $\f$ is a viscosity solution of the equation $(dd^c \f)^n = e^{g (x,\f)} \mu$.
Then by Proposition~\ref{pro:visc=pluripot},  $\f$ is a continuous psh function in $X$ which satisfies the inequality $(dd^c \f)^n \geq e^{g (x,\f} \mu$ in the pluripotential sense in $X$.  To prove equality assume that $B \Subset X$ is a small coordinate chart in $X$ biholomorphic to an euclidean ball in $\C^n$ and use the balayage construction to find a psh function $\p$ such that  $(dd^c \psi)^n = e^{g (x,\f)} \mu$,  $\psi = \f$ on $X \setminus \overline B$ and $\p \geq \f$ using Theorem\ref{thm:Balayage}. Then by the pluripotential comparison principle it follows that $\f \leq \psi$ on $B$. On the other hand, by Proposition~\ref{pro:visc=pluripot}, $\psi$ is a viscosity subsolution of the equation $(dd^c \psi)^n = e^{g (x,\f)} \mu$. Since $\f$ is a viscosity (super)-solution of the $(dd^c \f)^n = e^{g (x,\f)} \mu$ on $B$ and $\f = \psi$ on $\partial B$, it follows from the viscosity comparison principle that $\psi = f$ in $B$. Hence $\f =\psi$ on $B$ and  satisfies the equation $(dd^c \f)^n = \mu$ on $B$. Since $B$ is arbitrary, it follows that $\f$ is a pluripotential solution of the equation $(dd^c \f)^n = e^{g (x,\f)} \mu$ on $\Omega$.
\end{proof} 
 
\vskip 0.3 cm
\subsection{Viscosity solution : the Perron's method}

Once the global comparison principle holds, one easily constructs continuous  solutions
by Perron's method as we now explain.
Consider the following general equation
\begin{equation} \label{eq:GCMAE}
 F(x,u,Du,dd^c u) = 0, 
\end{equation}
on a domain $\Omega \subset \C^n$, where $F : \Omega \times \R \times \R^{2 n} \times H_n^+ \longrightarrow \R $ and extend it as usual to a real quadratic forms as usual by $F (x,s,p,Q) = F (x,s,p,Q^{1,1})$ if the hermitian $(1,1)-$part of $Q$  is semi-positive and by $+ \infty$ if not.
  
\begin{theo} \label{thm:perron1}
Assume the comparison principle holds for the complex Monge-Amp\`ere type equation (\ref{eq:GCMAE}) and the family $\mathcal U$ of bounded subsolutions of the equation (\ref{eq:GCMAE}) is non empty and locally upper bounded in $\Omega$.
Then the following properties: \\
1. The upper envelope 
$$
\f=\sup\{ u \, | \,  u \in \mathcal U  \} 
$$
is the maximal subsolution of the equation (\ref{eq:GCMAE}). 

Let $\gamma$ be a continuous function on $\partial \Omega$ and assume that the equation (\ref{eq:GCMAE}) has a subsolution $\underline u$ and a supersolution $\overline u$ such that $\underline u_* = \gamma = \overline u^*$ on $\partial \Omega$. Then $\f$ is the unique viscosity solution of (\ref{eq:GCMAE}) such that $u = \gamma$ on $\partial \Omega$. 
\end{theo}

\begin{proof} We argue as in
\cite{CIL92} p. 22-24.  Then 
lemma 4.2 there implies that the upper 
envelope $\f$ of the subsolutions of (\ref{eq:GCMAE}) is a 
subsolution of (\ref{eq:GCMAE}) since $F$ is lsc. Hence $\f$ is a subsolution of (\ref{eq:GCMAE}). 

The Ishii's trick is now to consider the lsc regularisation  $\f_*$ of $\f$. We are going
to show that $\f_*$ is  a supersolution of (\ref{eq:GCMAE}).    We argue by contradiction using a bump construction. Assume the converse is true. Then we can find $x_0 \in \Omega$ and  a lower test function $q$ for $\f_*$ at $x_0$ such that  
 $F^+ (x_0,\f_* (x_0), d q (x_0), dd^c q (x_0)) < 0$. This implies that $Q := dd^c q (x_0) \geq 0$ and  $ F (x_0,\f_* (x_0), D q (x_0),Q) < 0$.
 Let $(z^1, .., z^n)$ be a coordinate 
system centered at $x_0$ giving a local isomorphism with the complex unit ball. Define for $\delta > 0$, $r > 0$ small enough and $\vert z\vert < 2 r,$
 $$
  q_{\delta} (z) := q (z) - \delta (\vert z-x_0\vert^2 -  r^2).
 $$ 
 Then 
 \begin{eqnarray*}
  & q_{\delta} (x_0) = \f_* (x_0) + \delta r^2,\\
 & D q_{\delta} (x_0) = D q (x_0),\\
& D^2 q_{\delta} (x_0) =  Q - 2 \delta I_n.
 \end{eqnarray*}
 Then since $ F (x_0,\f_* (x_0), D q (x_0),Q) < 0$, it follows by continuity of $F$ in its domain that for $\delta > 0$ small enough we can find $r > 0$  
 small enough so that for $\vert z\vert < 2 r$,
 $$
  F (z,q_{\delta} (z),D q_{\delta} (z), Q_\delta) < 0,
 $$
 which means that $ q_{\delta}$ is a subsolution of our equation in the ball $\vert z\vert < 2 r$.
Now observe that for $\vert z- x_0\vert = 2 r$, $ q_{\delta} (z) = q (z) - \delta r^2 \leq \f_* - 3 \delta r^2 \leq \f - 3 \delta r^2$.
Therefore the new function defined by $\p (z) := \max \{\f,q_\delta\}$ on the ball $B_r (x_0) : \vert z - x_0\vert < 2 r$ and $\p = \f$ in $\Omega \setminus B_r (x_0)$ is a subsolution of the equation in $\Omega$.
Since $\f$ is the maximal subsolution of the equation on $\Omega$, we conclude that $U \leq \f$ in $\Omega$, which implies that $q_\delta \leq \f$ on the ball $B_r (x_0)$.
On the other hand, since $u_\delta (x_0) - \f_* (x_0) =   \delta r^2$,  there is a sequence $(y_j)$ converging to $x_0$ such that $\lim_{j \to + \infty} q_\delta (y_j)- \f (y_j) = q_\delta (x_0) - \f_* (x_0)$. Then for $j >> 1$ we have  $y_j \in B_r (x_0)$ and $q_\delta (y_j)- \f (y_j) > \delta r^2\slash 2 > 0$, which contradicts the inequality $q_\delta \leq \f$ on the ball $B_r (x_0)$.

Since $\underline u \leq \f \leq \overline u$ it follows that $\underline u_* \leq \f_* \leq \overline u$ in $\Omega$. Then $\f \leq \overline u^* = \gamma$ on $\partial \Omega$, while $ \gamma = \underline u_* \leq \f_*$ on $\partial \Omega$ which implies that $\f \leq \f_*$ in $\partial \Omega$. By the comparison principle it implies that $\f \leq \f_*$ in $\Omega$, hence $\f = \f_*$ is a viscosity solution of the equation (\ref{eq:GCMAE}).
\end{proof}

\begin{coro} Let  $\mu > 0$ be is a continuous volume fom $\Omega \Subset \C^n$ and $g$ is continuous and increasing in the second variable. Assume that the family $\mathcal U$ of bounded viscosity subsolutions of the following complex Monge-Amp\`ere equation 
\begin{equation} \label{eq:MA}
- (dd^c u)^n + e^{g (x,u)} \mu = 0, 
\end{equation}
 is non empty and locally upper bounded in $\Omega$. 
Then the maximal viscosity subsolution  
$$
\f=\sup\{ u \, | \,  u \in \mathcal U  \}, 
$$ 
is a viscosity solution of (\ref{eq:MA}). 
Moreover it is a continuous $\omega$-plurisubharmonic function on $\Omega$ and is also a solution of (\ref{eq:MA})
in the pluripotential sense.
\end{coro}
\begin{proof}
It remains to see that $\f$ is also a solution of (\ref{eq:MA}) in the pluripotential sense. Since this is a local property, it is enough to prove it locally. We argue by balayage. Let $B \Subset X$ be a small coordinate neighbourhood which is biholomorphic to an euclidean ball in $\C^ n$ such that $\omega$ has a local potential on a neighbourhood of $\overline B$. Since $\f$ is continuous on $\overline B$, we can solve the complex Monge-Amp\` ere equation $(dd^c \psi)^ n = e^{ (x,\f)} \mu $ on $ B$ with boundary values equal to $\f$ on $\partial B$ by Theorem~\ref{thm:Balayage}. Then by the pluripotential comparison principle we have $\psi \geq \f$ on $B$. Therefore the function $u := \psi$ on $B$ and $u = \f$ on $\Omega \setminus B$ is a continuous $\omega-$psh function on $\Omega$ and by Proposition~\ref{prop:vsc=plp}, it is a viscosity subsolution of the equation (\ref{eq:MA}).
Therefore by the global comparison principle $u \leq \f$ on $\Omega$, which proves that $\f = \psi$ on $B$ and then $\f$ satisfies the complex Monge-Amp\`ere equation $(\omega + dd^c \f)^n = e^{ (x,\f)} \mu $ in the pluripotential sense on $B$ which proves our statement.
\end{proof}
 \begin{rem} As we observed in Remark~\ref{rem:vscdegen}, the comparison principle is valid in a more general situation where  $\mu \geq 0$ and $g (x,s)$ is non decreasing in $s$. Therefore the Theorem above is still valid in this general situation. For a different proof of this last statement see \cite{Wang10}.
 
 \end{rem}

 \section{The viscosity approach in the compact case}
 
 We now set the basic frame for the viscosity approach to the following degenerate complex Monge-Amp\`ere equation 
 $$
  (\omega+dd^c \f)^n =e^{\e \f} \mu,
 \leqno{(DMA)_{\e,\mu}}
 $$
 where $\omega$ is a closed smooth $(1,1)$-form on a $n$-dimensional connected compact 
complex manifold $X$,
$\mu$ is a volume form with nonnegative continuous density and $\e \in \R_{+}$. 
 
 \smallskip

As we have seen in the last section, the  comparison principle lies at the heart of the viscosity approach. 
Once it is established, Perron's method
can be applied to produce viscosity solutions. Our main goal in this section is to establish the global comparison
principle for the equation $(DMA)_{\e,\mu}$. We only assume $X$ is compact (and $\e>0$): the structural feature of 
$(DMA)_{\e,\mu}$ allows us to avoid any restrictive curvature assumption on $X$ (unlike e.g. in \cite{AFS08}).

 \vskip 0.3 cm
 \subsection{Definitions for the compact case}

 To fit in with the viscosity point of view, we rewrite the Monge-Amp\`ere equation as
$$
 -(\omega+ dd^c \f)^n +  e^{\e \f} \mu = 0 .
$$

Let $x\in X$. If $\kappa \in \Lambda^{1,1} T_{x} X$ we define $\kappa_+^n$ to be $\kappa^n$ if $\kappa \ge 0$ and $0$ otherwise.

We let $PSH(X,\omega)$ denote the set of all $\omega$-plurisubharmonic ($\omega$-psh for short)
 functions on $X$: these are integrable  functions
$\f:X \rightarrow \R \cup \{-\infty\}$ such that $dd^c \f \geq -\omega$
in the sense of currents.

\smallskip

\begin{lem} \label{hyp}
Let $\Omega \subset X$ be an open subset and $z:\Omega\to \C^n$ be a holomorphic coordinate chart. 
Let $h$ be a smooth local potential for $\omega$ defined on $\Omega$.
Then $(MA_{\e,\mu})$ reduces in these $z$-coordinates to the scalar equation
$$
 e^{\e u} W - \det(u_{z\bar z})=0 \leqno{(MA_{\e,\mu|z})}
$$
where $u=(\f+h)|_{\Omega} \circ z^{-1}$, $z_*\mu=e^{\e h_{| \Omega} \circ z^{-1}}W d\lambda$ and $\lambda$ is the Lebesgue measure on $z(\Omega)$. 
\end{lem}

The proof is straightforward. 


In order to deal with degenerate elliptic non linear equations and be able to apply results from \cite{CIL92}, we introduce as in the local case the following Hamiltonian function.

If $\f^{(2)}_x$ is the $2$-jet at $x\in X$ of a ${\mathcal C}^2$ real valued function $\f$
we set
$$
F(x,\f (x),\f^{(2)}_x)=
\left\{
\begin{array}{ll}
e^{\e \f(x)} \mu_x-(\omega_x+ dd^c \f_x)^n & \text{ if } \omega+dd^c\f_x\ge 0 \\
+\infty & \text{otherwise}.
\end{array}
\right.
 $$
Then $F$ satisfies the degenerate ellipticity condition as well as the properness condition, but it is only lower semi-continuous.
However it is continuous on its domain (i.e. where it is finite).
 \vskip 0.3 cm
\subsubsection{Subsolutions}



  
Recall now the following definition from previous sections:

\begin{defi}
 A subsolution of $(DMA)_{\e,\mu}$ is an upper semi-continuous function $\f: X\to \R \cup \{-\infty\}$ such that
$\f\not \equiv -\infty$
and the following property is satisfied:
if $x_0\in X$ and  $q \in {\mathcal C}^2$, defined in a neighborhood of $x_0$, is such that $\f(x_0)=q(x_0)$ and
$$
 \f-q \  \text{ has a local maximum at} \ x_0,
 $$ 
 then $F(x_0,\f(x_0),q^{(2)}_{x_0})\le 0$. 
 \end{defi}

\vskip 0.3 cm
\subsubsection{(Super)solutions}
The definition of supersolutions follows the one given in the local setting:

\begin{defi}
A supersolution of $(DMA)_{\e,\mu}$ is a lower semicontinuous  function 
$\f: X\to \R \cup \{+\infty\}$ such that $\f\not \equiv +\infty$
and the following property is satisfied:
if $x_0\in X$ and  $q \in {\mathcal C}^2$, defined in a neighborhood of $x_0$, is such that $\f(x_0)=q(x_0)$ and
$
 \f-q \  \text{ has a local minimum at} \ x_0,
 $
 then $F_+(x_0,\f(x_0),q^{(2)}_{x_0}) \geq 0$. 
\end{defi}
Here $F_+(x_0,\f(x_0),q^{(2)}_{x_0}) := F(x_0,\f(x_0),q^{(2)}_{x_0})$ if $dd^c q(x_0) \geq 0$ and $0$ otherwise.
 
\begin{defi}
A viscosity solution of $(DMA)_{\e,\mu}$ is a function that is both a sub-and a supersolution. In particular, viscosity solutions
are automatically continuous.
Classical sub/supersolutions are ${\mathcal C}^2$ viscosity sub/supersolutions. 

A pluripotential solution of $(DMA)_{\e,\mu}$ is an usc function $\f \in L^{\infty}\cap PSH(X,\omega)$
such that for every local potential $\psi$ of $\omega$ we have $MA(\psi+\f)= e^{\e\f} \mu$ in the weak sense of currents.

\end{defi}

In this setting, the discussion after Theorem \ref{thm:visc=pluripot}
yields the following:

\begin{coro}
Let $X$ be a compact K\"ahler manifold and $\omega$ a smooth closed $(1,1)$ form 
whose cohomology class $[\omega]$ is big. Let $\f$ be any continuous $\omega$-psh function.
Then
$\f$ satisfies $(\omega+dd^c \f)^n \ge e^{\e \f} \mu$ in the viscosity sense iff $\langle (\omega+dd^c \f)^n \rangle \ge e^{\e \f} \mu$,
where $\langle (\omega+dd^c \f)^n \rangle$ is the non-pluripolar Monge-Amp\`ere measure \cite{BEGZ10}. 
\end{coro}

\vskip 0.3 cm
\subsection{The global viscosity comparison principle}

Since our conditions ($X$ compact, $\mu\ge 0$, $\e>0$) are invariant under dilation, we can always reduce to the
case $\e=1$, a normalisation that we shall often make in the sequel.

We now come to the main result of this section:

\begin{theo}\label{qcp}
The global viscosity comparison principle for $(DMA)_{1,\mu}$ holds, 
provided $\omega$ is a closed $(1,1)$-form on $X$, $\mu>0$, and $X$ is compact.
\end{theo}
 Observe that we do not assume $X$ to be K\"ahler nor $\omega$ to be semi-positive.
\begin{proof}
 We choose a constant $C>0$ such that $\f$ and $\p$ both are $\le C/ 4$ in $L^{\infty}$-norm. Since $\f - \p$
is upper semicontinuous on the compact manifold $X$, it follows that its maximum is achieved at some point $x_0 \in X$. 
 Choose complex coordinates $z = (z^1, \ldots, z^n)$ near $x_0$
defining a biholomorphism identifying an open neighborhood of $x_0$ to the complex ball $B_4 := B(0,4) \subset \C^n$ of radius $4$ sending $x_0$ to the origin in $\C^n$. 

We define $h_{\omega}\in {\mathcal C}^2(\overline{B_4}, \R)$ to be a local potential smooth up to the boundary  for $\omega$ and extend it smoothly to $X$.
We may without lost of generality assume that $\| h_{\omega} \|_{\infty} < C/ 4$.
In particular $dd^c h_{\omega}= \omega$ on $B_4$ and the usc function $u := \f \circ z^{- 1} + h_{\omega} \circ z^{- 1}$ is a viscosity subsolution of 
$$
(dd^c u)^n = e^{u} f \cdot \beta_n
\text{ in } B_4,
\leqno (\star)
$$ 
with $f :=  z^{*} (\mu) \slash \beta_n > 0$ is a positive and continuous volume form on $B_4$. 

On the other hand the lsc function $v := \p \circ z^{- 1} +h_{\omega} \circ z^{- 1}$ is a viscosity supersolution of the same equation.
 
 This is a crucial point:
the modified equation still has the same form as the original one.

We want to estimate $ \max_X (\f  - \p) =  \max_{\bar B_4}(u - v) =  u (0) - v (0) \leq 0$ by applying the classical maximum principle as in the local case. Observe that if the functions $u$ and $v$ were twice differentiable at $x_0$ the inequality follows from the maximum principle and the differential sub/super inequalities satisfied by $u$ and $v$ at $x_0$ respectively.

 In the general case we proceed as in \cite{CIL92} using the penalty method consisting in doubling the variable and adding a penalty function, but we will be adding two penalty functions. We consider  the function $x \longmapsto u (x) - v (x)$ as the restriction to the diagonal in the product $B_3 \times B_3$ of the function $(x,y) \longmapsto u (x) - v  (y) - \theta (x,y) - (1 \slash 2 \de) \vert x - y\vert^2 $ where   $\theta (x,y)$ is the first penality function which vanishes highly on the diagonal near the origin $(0,0)$ and is large enough on the boundary of the ball $B_3 \times B_3$ to force the maximum to be attained at an interior point; the second penalty function forces the maximum to be asymptotically attained  along the diagonal. The fact that
 the second derivative of the penalty function is a quadratic form  on $\R^{2n} \times \R^{2 n}$ which vanishes on the diagonal,  will be crucial.

 \smallskip
 
 We now proceed to the construction of the first penalty function $\theta$. 
We want to construct a smooth function $\theta \in {\mathcal C}^{\infty}(X^2,\R)$ satisfying the following conditions
\begin{itemize}
\item $\theta \ge 0$,
\item $\theta^{-1}(0)= \Delta \cap \{ \theta_2 \le -\eta \}$,
\item $\theta|_{X^2 \setminus B_2^2} > 3C$,
\end{itemize}
where $\eta > 0$ is small enough (see below for the definition of $\theta_2$) and $C > 0$.
 
 First we construct a Riemannian metric on $X$ which coincides with the flat K\"ahler metric $\frac{\sqrt{-1}}{2} dz^k \wedge d\bar z^k$ on the ball of center
$0$ and radius $3$. For $(x,y)\in X\times X$ define $d(x,y)$ to be the corresponding Riemannian distance function. The  continuous function $d^2$ is of class ${\mathcal C}^2$ near the diagonal and $>0$ outside the diagonal $\Delta \subset X^2$.

Next we construct a smooth non negative function $\theta_1$ on $X\times X$ by the following formula:
$$ \theta_1(x,y)=\chi (x,y). \sum_{i=1}^n |z^i(x)-z^i(y)|^{2n+4},$$
where $\chi$ smooth non negative cut off function with $0 \le \chi \le 1$,
$\chi\equiv 1$ on $B_3^2$ $\chi=0$ near $\partial B_4^2$. 

Then we construct  a second smooth function on $X\times X$ with 
$\theta_2|_{B_2^2} <-1$,  $\theta_2|_{M^2 \setminus B_2^2} > 3 C$.

Choose $1\gg \eta >0$ such that  $-\eta$ is a regular value of both 
$\theta_2$ and $\theta_2|_{\Delta}$.

We perform convolution of $(\xi,\xi')\mapsto \max(\xi, \xi')$ by a smooth semipositive function $\rho$ such that $B_{\R^2}(0,\eta) =\{\rho>0\}$ and get a smooth function on $\R^2$
$\max_{\eta}$ such that: 
\begin{itemize}
\item $\max_{\eta} (\xi, \xi')=\max(\xi, \xi')$ if $|\xi-\xi'|\ge\eta$, 
\item $\max_{\eta} (\xi, \xi')>\max(\xi, \xi')$ if $|\xi-\xi'|<\eta$. 
 \end{itemize}
 Then the function $\theta$ defined by $\theta :=\max_{\eta}(\theta_1,\theta_2)$ satisfies our requirements.
 
 Fix $\al > 0$. We want to apply the Jensen-Ishii's maximum principle to the functions $u , v$ and $\phi = \theta - \frac{1}{2 \al} \vert x-y\vert^2$.
 
For ${\al}>0$ small enough, consider $(x_{{\al}},y_{\al}) \in  \bar B_3 \times \bar B_3$ such that 
\begin{eqnarray*}
m_{{\al}}& := &\sup_{(x,y) \in \bar B_3^2} \left\{u (x)- v (y) -\frac{1}{2 \al} \vert x - y\vert^2 - \theta(x,y)\right\}\\
&=& u (x_{{\al}})- v  (y_{{\al}}) -\theta(x_{{\al}},y_{{\al}})-\frac{1}{2\al} \vert x_{\al} - y_{\al}\vert^2.
\end{eqnarray*}

 The supremum is achieved since we are maximizing an usc function on the compact set $\bar B_3^2$. We also have 
 \begin{equation} \label{eq:min}
   m_{{\al}} \ge u (0)- v(0)  = \f(x_0)- \p (x_0)  \geq -  C \slash 2 , 
 \end{equation}
for $\al > 0$ small enough.

 By construction, for $(x,y) \in   B_3^2 \setminus B_2^2$, we also have 
 \begin{equation} \label{eq:max}
 u (x)- v (y) - \theta(x,y)-\frac{1}{2\al} \vert x - y\vert^2 \leq  - 2 C < -  C,
 \end{equation}
 which implies that $(x_{{\al}}, y_{{\al}}) \in B_2^2$.

The following result follows easily from the above properties (see \cite[Proposition 3.7]{CIL92}):

 \begin{lem} \label{asymp}
 For ${\al} > 0$ small enough we have $ \vert x_{\al} - y_{\al}\vert^2  = o ({\al})$. Every limit
point $(\hat x, \hat y)$ of $(x_{{\al}}, y_{{\al}})$ satisfies $\hat x= \hat y$, $(\hat x , \hat x) \in \Delta \cap \{  \theta_2 \le -\eta \}$
and 
\begin{eqnarray*}
\lim_{{\al} \to 0}  (u (x_{{\al}})- v (y_{{\al}})) 
&=&  u (\hat x)- v (\hat x) \\
& = & \f(x_0)- \p (x_0.
\end{eqnarray*}
\end{lem}

Next, we use Jensen-Ishii's  maximum principle with
 $\phi =\frac{1}{2\alpha} d^2 +\theta$. 
For $0 < \alpha << 1$, everything is localized to $B(0,2)$ hence
$d$ reduces to  the euclidean distance function.
Using the usual formula for the first and second derivatives 
of its square,   
we get the following:

\begin{lem}\label{main}
$\forall \e >0$, we can find $(p_*, Q_*), (p^*, Q^*)\in \C^n \times Sym_{\R}^2 (\C^n)$
s.t.
\begin{enumerate}
 \item $(p_*, Q_*)\in \overline{J}^{2+} u (x_{\alpha})$,
\item  $(p^*, Q^*)\in \overline{J}^{2-} v(y_{\alpha})$,
where $p^* = D_x \theta (x_\al,y_\al) + \frac{1}{2 \al} (x_\al - y_\al)$ and $p_* = - D_y \theta (x_\al,y_\al) + \frac{1}{2 \al} (x_\al - y_\al)$
\item The block diagonal matrix with entries $(Q_*, Q^*)$ satisfies:
$$
-(\e^{-1}+ \| A \| ) I \le 
\left(
\begin{array}{cc}
Q_* & 0 \\
0 & -Q^*
\end{array}
\right)
\le A+\e A^2, 
$$
where $A=D^2\phi(x_{\alpha}, y_{\alpha})$, i.e.
$$A =\alpha^{-1}
\left(
\begin{array}{cc}
I &  -I\\
-I &  I
\end{array}
\right) +D^2\theta (x_{\alpha}, y_{\alpha})$$
and $\| A \|$ is the spectral radius of $A$ (maximum of the absolute values for the eigenvalues of this symmetric
matrix). 
\end{enumerate}
\end{lem}

By construction, the Taylor series of $\theta$ at any point in
 $\Delta \cap \{ \theta_2 < -\eta \}$ vanishes up to order $2n$. By transversality, 
$\Delta \cap \{  \theta_2 < -\eta \}$
is dense in $\Delta \cap \{  \theta_2 \le -\eta \}$, and this
 Taylor series  vanishes up to order $2n$ on $\Delta \cap \{  \theta_2 \le -\eta \}$.
In particular, 
$$
 D^2\theta (x_{\alpha}, y_{\alpha}) =O(d(x_{\alpha}, y_{\alpha})^{2n})
=o(\alpha^{n}).$$
This implies $\|A \|\simeq 1 \slash \alpha$. We choose $\alpha = \e$ and deduce
$$
-(2\alpha^{-1} ) I \le 
\left(
\begin{array}{cc}
Q_* & 0 \\
0 & -Q^*
\end{array}
\right)
\le \frac{3}{\alpha} \left(
\begin{array}{cc}
I &  -I\\
-I &  I
\end{array}
\right) + o(\alpha^{n})
$$

Looking at the upper and lower diagonal terms we deduce that the eigenvalues
of $Q_*, Q^*$ are $O(\alpha^{-1})$. Evaluating the inequality on vectors of the form 
$(Z,Z)$ we deduce from the $\le$ that the eigenvalues
of $Q_{*}-Q^{*}$ are $ o(\alpha^{n})$.

\smallskip

For a fixed $Q\in Sym_{\R}^2 (\C^n)$, denote by $H = Q^{1,1}$ its $(1,1)$-part. It is a hermitian matrix. 
 Obviously the eigenvalues of $H_* := Q_*^{1,1}, H^* := Q^{*1,1}$ are $O(\alpha^{-1})$ but those
of $H_{*} - H^*$ are $o(\alpha^{n})$. 
Since $(p_*, Q_*)\in \overline{J^{2+} }w_*(x_{\alpha})$ 
we deduce from the definition of viscosity solutions 
that $H_*$ is positive definite and that the product of its $n$
eigenvalues is $\ge c>0$ uniformly in $\alpha$. In particular its 
smallest eigenvalue is $\ge c\alpha^{n-1}$. The relation $ H_{*} + o(\alpha^{n})\le H^{*}$
forces $H^{*} >0$ for $\al > 0$ small enough. Then we have
$ \text{det} H_* \leq \text{det} H^*  + o(\alpha^{n})$.

By viscosity inequalities, we get
 \begin{eqnarray*} 
 e^{ u (x_{\alpha})}  \leq  \text{det} (H_*) 
 & \leq & \text{det} (H^*)  + o(\alpha^{n}) \\
 & \leq & e^{v (y_{\alpha})} + o(\alpha^{n}).
\end{eqnarray*} 
 Passing to the limit as $\al \to 0$, we obtain the inequality $e^{u (\hat x)} \leq e^{v (\hat x)}$, which implies that $u (\hat x) \leq   
 v (\hat x)$. 
\end{proof}

\begin{rem}
  The miracle with the complex Monge Amp\`ere equation we are studying is that
 the equation does not depend on the gradient in complex coordinates. In fact, it  takes the 
form $F(Q)-f(x)=0$. The localisation technique would fail without this structural feature. 
\end{rem}

\begin{rem}
In the global case when $\e=0$, i.e. for $(\omega+dd^c \f)^n=\mu$ on a compact K\"ahler manifold, 
 a subsolution is already a solution and then this method seems to be of no help. However the global comparison principle whould imply uniqueness even in this case.
 Indeed assume that $\f_1, \f_2$ are bounded viscosity solutions to  the equation $(\omega+dd^c \f)^n=\mu $. Then $\f_1, \f_2$ are continuous psh functions on $X$. Let $x_0 \in X$ such that $\f_1 (x_0) - \f_2 (x_0) = \max_X (\f_1 - \f_2)$. Then the function $\p_1 := \f - \f_1 (x_0)$ and $\p_2 = \f_2 - \f_2 (x_0)$ are viscosity solutions of the same equation such that $\p_1 \leq \p_2$. We want to prove equality. Assume that at some point $y \in X$ we have $\p_1 (y) < \p_2 (y)$. Then the open set $ \Omega := \{ \p_1 < \p_2\}$ is not empty and $\p_1 \geq \p_2$ on the boundary $\bd \Omega$.
 By the comparison principle we have $\p_1 \geq \p_2$ on $\Omega$, which is a contradiction.
 \end{rem}

\vskip 0.3 cm
\subsection{Perron's method}

Once the global comparison principle holds, one easily constructs continuous (viscosity=pluripotential) solutions
by Perron's method as we explained in the last section.

\begin{theo} \label{thm:perron2}
Assume the global comparison principle holds for $ (DMA)_{\e,\mu}$ and that $(DMA)_{\e,\mu}$  has a bounded subsolution $\underline{u}$
and a bounded supersolution $\overline{u}$. 
Then the maximal subsolution, 
$$
\f=\sup\{ w \, | \,  \underline{u} \le w
\le \overline{u} \ \text{and} \ w \ \text{is a viscosity subsolution of } (DMA)_{\e,\mu} \} 
$$
is the unique viscosity solution of $(DMA)_{\e,\mu}$. 

In particular, it is a continuous $\omega$-plurisubharmonic function in $X$ which is also a solution of $(DMA)_{\e,\mu}$
in the pluripotential sense.
\end{theo}

\begin{exa}
Assume $X$ is a complex projective manifold such that $K_X$ is ample. \
Let $\omega>0$ be a K\"ahler representative of $[K_X]$ and $\mu$
a smooth non degenerate volume form on $X$ with $Ric(\mu)=-\omega$. Then the Monge-Amp\`ere equation 
$ (\omega+dd^c\f)^n =e^{\f} \mu$ satisfies all the hypotheses 
of Theorem \ref{thm:perron2} and has a unique (viscosity=pluripotential) solution $\f$. 
On the other hand, the Aubin-Yau theorem \cite{Aub78},\cite{Yau78} 
implies that it has a unique smooth solution $\f_{KE}$ (and $\omega+dd^c\f_{KE}$ is 
the canonical K\"ahler-Einstein metric on $X$). 
Uniqueness of the pluripotential solution insures $\f=\f_{KE}$
hence the potential of the canonical KE metric on $X$ is the 
envelope of the subsolutions to  $ (\omega+dd^c\f)^n =e^{\f} \mu$. 
\end{exa}

  \section{Weak versions of Calabi-Yau and Aubin-Yau theorems}
  
In this section we apply the viscosity approach to show that the canonical 
singular K\"ahler-Einstein metrics
constructed in \cite{EGZ09} have continuous potentials.

\subsection{Manifolds of general type}

Assume $X$ is compact K\"ahler and $\mu$ is a continuous volume form with semi-positive density. 
Fix $\beta$ a K\"ahler form on $X$.  

\begin{coro} \label{cor:ContSol} Assume that $\omega \geq 0$ is a closed $(1,1)-$form and $\mu > 0$ is a continuous positive volume form. Then
$(DMA)_{\e,\mu}$ has a unique viscosity solution $\f$, 
which is also the unique solution in the pluripotential sense. Hence it is a continuous $\omega-$psh function. 
\end{coro}
\begin{proof}
Indeed the global comparison principle holds in this case and Theorem~\ref{thm:MSUB} and Theorem~\ref{thm:perron2} enable us to conclude. 
\end{proof}

 We are now ready to establish that the (pluripotential) solutions of some Monge-Amp\`ere
 equations considered in the first section are continuous.

\begin{theo} \label{thm:exp}
Assume $X$ is a compact K\"ahler manifold, $\omega$ is a semipositive $(1,1)$-form with 
$\int_{X}\omega^n>0$  and $\mu \geq 0$ is a semi-positive  continuous 
volume form  on $X$ normalized by $\mu (X) = 1$. Then  there exists a unique
continuous $\omega$-plurisubharmonic function $\f$ which is
the viscosity (equivalently pluripotential) solution to the 
degenerate complex Monge-Amp\`ere equation
$$
(\omega+dd^c\f)^n =e^{\f} \mu 
$$
\end{theo}

\begin{proof} 
Observe that if moreover $\mu$ has positive density, the result is an immediate consequence of 
Corollary~\ref{cor:ContSol} together with the unicity statement Proposition~\ref{prop:UNIQ}.  

It remains to relax the positivity assumption made on $\mu$. From now on
$\omega$ is semi-positive and big and $\mu$ is a probability measure
with semi-positive continuous density.
We can solve 
$$
(\omega+dd^c \f_{\e})^n = e^{\f_{\e}} [\mu+\e \beta^n]
$$
where $\f_{\e}$ are continuous $\omega$-psh functions and $0<\e \leq 1$. As we already observed in the firs section this implies that
 the family $M_{\e}:=\sup_X \f_{\e}, \e \in ]0,1]$ is  bounded.

We infer that $(\f_{\e})$ is relatively compact in $L^1(X)$. 
It follows from Theorem~\ref{thm:StabEst} that $(\f_{\e})$ is actually uniformly bounded,
as $\e$ decreases to zero. 

Using again the stability estimates Theorem~\ref{thm:StabEst}, we get
$$
|| \f_{\e}-\f_{\e'} ||_{L^{\infty}} \leq C 
\left( || \f_{\e}-\f_{\e'}||_{L^1} \right)^{\frac{1}{n+2}}.
$$
Thus, if $(\epsilon_j)$ is a sequence decreasing
to zero as $j$ goes to $+\infty$ such that $(\f_{\e_j})_j$ converges in $L^1$,
 $(\f_{\e_j})$ is actually a Cauchy sequence of continuous functions,
 hence it uniformly converges,
 to the unique continuous pluripotential solution $\f$
of $(DMA)_{1,\mu}$. From this, it follows that
$(\f_{\e})$ has a unique cluster value 
in $L^1$ when $\epsilon$ decreases to $0$
hence converges in $L^1$. The preceding argument yields uniform convergence. 

Theorem \ref{thm:perron2}
insures that  $\f$ is also a viscosity subsolution. Remark
6.3 p. 35 in \cite{CIL92} actually enables one to conclude
 that $\f$ is indeed a viscosity solution. 

\end{proof}

\begin{coro}
If $X^{can}$ is a canonical model of a general type projective manifold then the canonical singular K\"ahler-Einstein metric
on $X^{can}$ constructed in \cite{EGZ09} has continuous potentials. 
\end{coro}
\begin{proof}
This is a straightforward consequence of the above theorem, working in  a log resolution of $X^{can}$,
where $\omega=c_1(K_X,h)$ is the pull-back of the Fubini-Study form from $X^{can}$
and $v=v(h)$ has continuous semi-positive density, since $X^{can}$ has canonical
singularities.
\end{proof}

\vskip 0.3 cm
\subsection{Continuous Ricci flat metrics}
 
We now turn to the study of  the degenerate equations $(DMA)_{0,\mu}$
$$
(\omega+dd^c \f)^n=\mu
$$
on a given compact K\"ahler manifold $X$. 
Here $\mu = f \mu_0$ is a degenerate volume form
with  density $f \in L^p (X)$ ($p > 1$) and $\omega$ is a smooth semi-positive closed $(1,1)$ form on $X$.
We assume that $\mu$ is normalized so that 
$$
\mu(X)=\int_X \omega^n.
$$
This is an obvious necessary condition in order to solve the equation
$$
(\omega+dd^c \f)^n=\mu, \leqno (DMA)_{0,\mu}
$$
on $X$. Bounded solutions to such equations have been provided in \cite{EGZ09} when 
$\mu$ has $L^p$-density, $p>1$, by adapting the arguments of \cite{Kol98}.
Our aim here is to show that these are actually {\it continuous}.

\begin{theo} \label{thm:flat} Let $\mu = f \mu_0$ be a degenerate volume form
with  density $0 \leq f \in L^p (X)$ ($p > 1$) and $\omega \geq 0$ is a smooth semi-positive closed $(1,1)$ form on $X$.
We assume that $\mu$ is normalized so that 
$$
\mu(X)=\int_X \omega^n.
$$
Then the complex Monge-Amp\`ere equation $(DMA)_{0,\mu}$ has a unique continuous pluripotential solution $\f$ such that $\int_X \f \mu = 0$.
\end{theo}
The plan is to combine the viscosity approach for the family of equations
$(\omega+dd^c \f)^n=e^{\e \f} \mu$, together with the pluripotential tools developed in
\cite{Kol98,Ceg98,GZ05,EGZ09,EGZ11}.

\begin{proof} Assume first that $\mu $ is a continuous positive volume form. 
For $\e>0$ we let $\f_{\e}$ denote the unique viscosity (or equivalently pluripotential) $\omega$-psh 
continuous solution of the equation
$$
(\omega+dd^c \f_{\e})^n=e^{\e \f_{\e}} \mu, 
$$
given by Theorem~\ref{cor:ContSol}.
 As before we see that $M_{\e}:=\sup_X \f_{\e}$ is uniformly bounded. We infer that $(\f_{\e})$ is bounded in $L^1$ and the Monge-Amp\`ere
measures $(\omega+dd^c \f_{\e})^n$ have uniformly bounded densities in $L^{\infty}$.
Once again by Theorem\ref{thm:StabEst}  this family of continuous
$\omega$-psh functions is uniformly Cauchy hence converges to a {\it continuous} pluripotential
solution of $(DMA)_{0,\mu}$.
This pluripotential solution is also a viscosity solution by (\cite{CIL92}, Remark 6.3).
\smallskip
As we already observed in section 1, the solutions of $(DMA)_{0,\mu}$ are unique, up to an additive constant.
It is natural to wonder which solution is reached by the the family $\f_{\e}$. Observe that
$\int _X e^{\e \f_{\e}} \mu =\int_X \mu =\int _X \omega^n$ thus
$$
0=\int_X \frac{e^{\e \f_{\e}}-1}{\e} \mu =\int_X \f_{\e} \mu +o(1)
$$
hence the limit $\f$ of $\f_{\e}$ as $\e$ decreases to zero is the unique solution of $(DMA)_{0,\mu}$
that is normalized by $\int_X \f \, \mu = 0$. \\
Now assume that $\mu = f \mu_0$ has an $L^p-$density with $p > 1$. Let $f_j$ a sequence of smooth positive functions on $X$ such $f_j \to f$ in $L^p (X)$. 

By the previous case there for each $j \in \N$, there exists a continuous solution $\f_j \in PSH (X,\omega)$ to the equation
 $$
(\omega+dd^c \f_{j})^n= f_j \mu_0, 
$$
 with $\int_X \f_j \mu = 0$.
By \cite{GZ05} the sequence $\f_j$ is bounded in $L^1 (X)$ and again by Theorem~\ref{thm:StabEst}, the sequence $(f_j)_{j \in \N}$ is a Cauchy sequence of continuous $\omega-$psh functions for the uniform norm on $X$, hence it converges to a continuous $\omega-$psh function $\f $ which is a solution to  the equation $(DMA)_{0,\mu}$.
\end{proof}

Note that the way we have produced solutions (by approximation through the non flat case)
is independent of \cite{Aub78,Yau78}.

 Now we can prove that the Ricci-flat singular metrics constructed in (\cite{EGZ09}, Theorem 7.5) have continuous potentials. 
\begin{coro}
Let $X$ be a compact $\Q$-Calabi-Yau K\"ahler space.
Then $X$ admits a Ricci-flat singular metric  with
continuous potentials.
\end{coro}

\section{Concluding remarks}

\vskip 0.3 cm
\subsection{The continuous Calabi conjecture}

The combination of viscosity methods and pluripotential techniques
yields a soft approach to solving degenerate complex Monge-Amp\`ere
equations of the form
$$
(\omega+dd^c \f)^n=e^{\e \f} \mu
$$
when $\e \geq 0$.

Recall that here $X$ is a compact K\"ahler n-dimensional manifold, 
$\mu$ is a semi-positive volume form with $L^p $-density $p > 1$ and $\omega$ is 
smooth closed $(1,1)$-form whose cohomology class is  semi-positive and big
(i.e. $\{\omega\}^n>0$).

Altogether this provides an alternative and independent approach to Yau's solution of the Calabi conjecture \cite{Yau78}:
we have only used upper envelope constructions (both in the viscosity and pluripotential sense), 
a global (viscosity) comparison principle  and Kolodziej's pluripotential 
techniques (\cite{Kol98, EGZ09}.

It applies to degenerate equations but yields solutions that are merely continuous (Yau's work yields
smooth solutions, assuming the cohomology class $\{\omega\}$ is K\"ahler and the measure $\mu$ is
both positive and smooth). However it is possible to prove that the solutions are H\"older continuous locally in the ample locus $\Omega_{\alpha}$ of the class $\{\omega\}$ (see \cite{DDHKGZ11}).

Note that a third (variational) approach has been studied recently in \cite{BBGZ09}. It applies to even
more degenerate situations where $\mu$ might be singular, providing solutions with less regularity (that belong to the so called
class of finite energy).

\subsection{The case of a big class}

Our approach applies equally well to a slightly more degenerate situation.
We still assume here that $(X,\omega_X)$ is a compact K\"ahler manifold of dimension $n$, but $\mu=f \mu_0$ is merely assumed
to have density $f \geq 0$ in $L^{\infty}$ and moreover the smooth real closed $(1,1)$-form $\omega$ is no longer
assumed to be semi-positive: we simply assume that its cohomology class $\alpha:=[\omega] \in H^{1,1}(X,\R)$
is {\it big}, i.e. contains a K\"ahler current.

It follows from the work of Demailly \cite{Dem92} that one can find a K\"ahler current in $\alpha$ with analytic 
singularities: there exists an $\omega$-psh function $\p_0$ which is smooth in a Zariski open set
$\Omega_{\alpha}$ and has logarithmic singularities of analytic type along $X \setminus \Omega_{\alpha}=\{\p_0=-\infty\}$,
such that $T_0=\omega+dd^c \p_0 \geq \e_0 \omega_X$ dominates the K\"ahler form $\e_0 \omega_X$,
$\e_0>0$.

We refer the reader to \cite{BEGZ10} for more preliminary material on this situation. Our aim here is to show that one can
solve $(DMA)_{1,\mu}$ in a rather elementary way by observing as in the first section that the (unique) solution is the upper envelope of subsolutions.
We let as before
$$
{\mathcal F}:=\left\{ \f \in PSH(X,\omega) \cap L^{\infty}_{loc}(\Omega_{\alpha}) \, / \, (\omega+dd^c \f)^n \geq e^{\f} v
\text{ in } \Omega_{\alpha} \right\}
$$
denote the set of all (pluripotential) subsolutions to $(DMA)_{1,\mu}$ (which only makes sense in $\Omega_{\alpha}$).

Observe that ${\mathcal F}$ is not empty: since $T_0^n$ dominates a volume form and $\mu$ has density in
$L^{\infty} (X)$, the function $\p_0-C$ belongs to ${\mathcal F}$ for $C$ large enough. We assume for
simplicity $C=0$ (so that $\p_0 \in {\mathcal F}$) and set
$$
{\mathcal F}_0:=\{ \f \in {\mathcal F} \, / \, \f \geq \p_0 \}.
$$

\begin{prop}
The class ${\mathcal F}_0$ is uniformly upper bounded on $X$.
It  compact (for the $L^1$-topology).
\end{prop}

\begin{proof} The proof is the same as for Lemma~\ref{lem:UB}.
We first show that ${\mathcal F}_0$ is uniformly bounded from above (by definition it is
bounded rom below by $\p_0$). We can assume without loss of generality that $\mu$ is
normalized so that $\mu(X)=1$. Fix $\p \in {\mathcal F}_0$. It follows from the convexity
of the exponential that
$$
\exp \left( \int \p \mu \right) \leq \int e^{\p} \mu  \leq \int (\omega+dd^c \p)^n \leq Vol(\alpha).
$$
All integrals here are computed on the Zariski open set $\Omega_{\alpha}$.
We refer the reader to \cite{BEGZ10} for the definition of the volume of a big class.

We infer
$$
\sup_X \p \leq \int \p \mu +C_\mu \leq \log Vol(\alpha)+C_\mu,
$$
where $C_\mu$ is  a uniform constant that only depends on the fact  that all $\omega$-psh functions are integrable
with respect to $\mu$ (see \cite{GZ05}). This shows that ${\mathcal F}_0$ is uniformly bounded from
above by a constant that only depends on $\mu$ and $Vol(\alpha)$.

We now check that ${\mathcal F}_0$ is  compact for the $L^1$-topology.
Fix $\p_j \in {\mathcal F}_0^{\N}$. We can extract a subsequence
that converges in $L^1$ and almost everywhere to a function $\p \in PSH(X,\omega)$.
Since $\p \geq \p_0$, it has a well defined Monge-Amp\`ere measure in $\Omega_{\alpha}$
and we need to check that $(\omega+dd^c \p)^n \geq e^{\p} \mu$.
We proceed in the same way as Lemma\~ref{lem:UB}.
\end{proof}

It follows that
$$
\p:=\sup\{ \f \, / \,  \f \in {\mathcal F}_0 \},
$$
the upper envelope of pluripotential subsolutions to $(DMA)_{1,\mu}$, is a well defined $\omega$-psh function
which is locally bounded in $\Omega_{\alpha}$.

\begin{theo}
The function $\p$ is a pluripotential solution to $(DMA)_{1,\mu}$.   
\end{theo}

\begin{proof} The proof proceeds by balayage locally in $\Omega_{\alpha}$ as in the proof of Theorem~\ref{thm:MSUB}.
\end{proof}

\begin{rem}
The situation considered above covers in particular the construction
of a K\"ahler-Einstein current on  a variety $V$ with
ample canonical bundle $K_V$ and canonical singularities, since the
canonical volume form becomes, after passing to a desingularisation $X$, a volume form
$\mu =f \mu_0$ with density $f \in L^{\infty}$.

The more general case of log-terminal singularities yields density $f \in L^p$, $p>1$.
One can treat this case by an easy approximation argument:
setting $f_j=\min (f,j) \in L^{\infty}$, one first solves 
$(\omega+dd^c \f_j)^n=e^{\f_j} f_j \mu_0$ and observe (by using the comparison principle)
that the $\f_j's$ form a decreasing sequence which converges to the unique
solution of $(\omega+dd^c \f)^n=e^{\f} f \mu_0$.
\end{rem}

\subsection{More comparison principles}

Let again $B \subset \C^n$ denote the open unit ball
 and let $B'=(1+\eta)B$ with $\eta>0$
be a slightly larger open ball. Let $u, u'\in PSH(B')$ be  
plurisubharmonic functions. By convolution with an adequate non negative kernel
of the form
$\rho_{\epsilon}(z)=\epsilon^{-2n}\rho_1(\frac{z}{\epsilon}) $
we construct $(u_{\epsilon})_{\eta>\epsilon>0}$ a family of smooth
plurisubharmonic functions decreasing to $u$ as $\epsilon$ decreases to $0$.

\begin{lem}\label{mol}
$$ \forall z \in B \  u (z) + u'(z)=
\limsup_{n\to \infty} \sup\{ u'(x)+u_{1/j}(x) | j\ge n, \  |x-z|\le 1/n \}
$$
\end{lem}

\begin{proof}
Indeed, we have, if $2/n <\eta$: 
 \begin{eqnarray*}
u(z)+u'(z)&\le& \sup \{ u'(z)+ u_{1/j}(z) | j\ge n \}
\\
&\le &
\sup\{ u'(x)+u_{1/j}(x) | j\ge n, \  |x-z|\le 1/n \} \\  
&\le &\sup\{u'(x)+ u(x)| \quad |x-z| \le 2/n \}.
 \end{eqnarray*}
Since $u+u'$ is upper semicontinuous, we have: 
$$ u (z) + u'(z)=(u+u')^*(z)= \lim_{n\to \infty} \sup\{ u+u'(x)| \quad |x-z| \le 2/n \}.$$
\end{proof}

\begin{lem} \label{lemcle}
Let $\phi$ a bounded psh function on $B$ and $\mu$ 
a continuous non negative volume form 
such that $e^{-\phi}(dd^c\phi)^n \ge \mu$ in the viscosity sense. 

Let $\psi$ be a bounded psh function and $\nu$ 
a continuous positive volume form, both defined on $B'$
 such that $(dd^c\psi)^n \ge \nu$. 

Then
$\exists C,c >0$ depending only on 
$\| \psi \| _{L^{\infty}}, \| \phi \|_{L^{\infty}}$ 
such that for every $\epsilon \in [0,1]$ 
$\Phi=\phi+\epsilon \psi$ satisfies: 

$$ e^{-\Phi} (dd^c\Phi)^n \ge (1-\epsilon)^n e^{-C\epsilon}  \mu + c \epsilon^n \nu
$$
in the viscosity sense in $ B$. 
\end{lem}

\begin{proof} We may assume $\epsilon >0$ and $\nu$ to be smooth.
Let us begin by the case when $\psi$ is of class $C^2$. 
Let $x_0\in B$ and $q\in C^2$ such that $q(x_0)=\Phi (x_0)$ 
and $\Phi-q$ has a local maximum at $x_0$. 
Then, $\phi- (q-\epsilon \psi)$ has a local maximum at $x_0$.

  We deduce:
$$ dd^c (q-\epsilon \psi)_{x_0} \ge 0 $$
$$ e^{- q(x_0) + \epsilon \psi(x_0) } (dd^c (q-\epsilon \psi))_{x_0})^n \ge \mu_{x_0}.$$
Using the inequality $(dd^c q)^n_{x_0} \ge (dd^c (q-\epsilon \psi))_{x_0})^n + \epsilon^n (dd^c\psi)^n$, we conclude. 

\smallskip

We now treat the general case.
Since $\psi$ is defined on $B'$ we can construct by the above classical 
mollification  a sequence of $C^2$ psh functions $(\psi_{1/k})$ converging to $\psi$
as $k$ goes to $+\infty$. 

We know from the proof of Proposition \ref{pro:visc=pluripot} 
that $(dd^c\psi_k)^n \ge ((\nu^{1/n} )_{1/k} )^n=\nu_k$ in both the pluripotential 
and viscosity sense.

We conclude from the previous case  that
 $\Phi_k= \phi + \epsilon \psi_k$ satisfies
$$
c \epsilon^n \nu_k + (1-\epsilon)^n e^{-C\epsilon}  \mu \le e^{-\Phi_k} (dd^c \Phi_k)^n
$$
in the viscosity sense. 
Since $\mu_k>0$, we have:
 $$ c \epsilon^n \nu_k + (1-\epsilon)^n e^{-C\epsilon} 
 \mu -e^{-\Phi_k} (dd^c \Phi_k)_+ ^n \le 0$$ in the viscosity sense. 

By Lemma 6.1 p. 34 and  Remark 6.3 p. 35 in \cite{CIL92}, we conclude that 
$$\bar\Phi=
\limsup_{n\to \infty} \sup\{ \Phi_j(x) | j\ge n, \  |x-z|\le 1/n \} $$
 satisfies the limit inequation
$$ 
e^{-\bar\Phi} (dd^c\bar\Phi)_+^n \ge (1-\epsilon)^n e^{-C\epsilon}  \mu + c \epsilon^n \nu 
$$ 
in the viscosity sense.  
Now Lemma \ref{mol} implies that $\bar \Phi=\Phi$. 
Since $\nu>0$, the proof is complete. 
\end{proof}

\begin{theo}
Let $X$ be a compact K\"ahler manifold 
and $\omega \ge 0$ be a semi-k\"ahler smooth form. 

Then, the global viscosity comparison principle holds 
for $(DMA)_{1,\mu}$ for any non negative continuous  measure $\mu$ with $\mu (X) > 0$. 
\end{theo}

\begin{proof} This is a variant of
 the  argument sketched in \cite{IL90} sect. V.3 p. 56. 

Let $\overline{u}$  be a  supersolution 
and  $\underline{u}$ be a 
subsolution. 
Perturb  the supersolution $\overline{u}$ 
setting  $\overline{u}_{\delta}=\overline{u}+\delta$. 
This  $\overline{u}_{\delta}$ is a supersolution 
to $(DMA)_{1, {\tilde w}}$ for every continuous volume form $\tilde w$
such that $\tilde w\ge e^{-\delta} \mu$. 

We  can always assume that $\mu (X) = \int_X \omega^n$. Choose $\nu>0$ a continuous positive volume form such that $\nu (X) = \int_X \omega^n$. 
 We can construct $\psi$ a continuous quasiplurisubharmonic functions 
such that, in the viscosity sense
$$
(\omega+dd^c\psi)^n=\nu.
$$

 Perturb  the 
subsolution $\underline{u}$ setting
 $$
 \underline{u}_{\epsilon}=(1-\epsilon)\underline{u} +\epsilon \psi.
 $$ 
By Lemma \ref{lemcle}, $\underline{u}_{\epsilon}$ satisfies, in the viscosity sense
$$
e^{-(1+\epsilon)u} (\omega+dd^c u)^n \ge 
\left(\frac{1-\epsilon}{1+\epsilon}\right)^ne^{-C\epsilon} \mu + c\left(\frac{\epsilon}{1+\epsilon}\right)^n \nu
$$
 This in turn implies that $\underline{u}_{\epsilon}$ satisfies, in the viscosity sense: 
$$
e^{-u} (\omega+dd^c u)^n \ge e^{-\epsilon \|u \|_{\infty}}  
\left[ \left(\frac{1-\epsilon}{1+\epsilon}\right)^n e^{-C\epsilon} \mu+ c\left(\frac{\epsilon}{1+\epsilon}\right)^n \nu \right].
$$

 Hence $\underline{u}_{\epsilon}$ satisfies, in the viscosity sense: 
$$
e^{-u} (\omega+dd^c u)^n \ge \tilde \nu$$ whenever 
$\tilde \nu \le e^{-\epsilon \|u \|_{\infty}}(
(\frac{1-\epsilon}{1+\epsilon})^ne^{-C\epsilon}v+ c(\frac{\epsilon}{1+\epsilon})^n \nu).
$

Choosing
 $1\gg \delta \gg \epsilon >0$, we find a continuous volume form 
$\tilde \nu >0$ such
that   $\overline{u}_{\delta}$
is a supersolution and $\underline{u}_{\epsilon}$ is a viscosity subsolution 
of $e^{-u}(\omega+dd^c u)^n=\tilde \nu$. Using the viscosity  
comparison principle for $\tilde \nu$, we conclude that
 $\overline{u}_{\delta}\ge \underline{u}_\epsilon$. Letting 
$\delta\to 0$, we infer $\overline{u}\ge \underline{u}$.
\end{proof}

This comparison principle has been inserted here for completeness. 
It could have been used instead of the pluripotential-theoretic arguments to establish existence 
of a viscosity solution in the case $\mu\ge 0$ of Theorem  \ref{thm:exp}. 
This could be useful in dealing with similar problems where pluripotential tools are
less efficient.

\vskip 0.3 cm
\noindent Institut de Math\'ematiques de Toulouse \\
Universit\'e Paul Salatier, \\
118 Route de Narbonne, \\
31062 Toulouse cedex 09

\end{document}